\documentclass[12pt]{amsart}

\usepackage{times,amsfonts,amsmath,amstext,amsbsy,amssymb,
  amsopn,amsthm,upref,eucal, amscd, caption}
\usepackage[T1]{fontenc}
\usepackage{color}
\usepackage[pdftex]{graphicx}  
\usepackage{longtable}

\usepackage[all,arc,knot]{xy}
\usepackage{mathtools}
\xyoption{arc}
\usepackage{verbatim} 

\newtheorem{theorem}{Theorem}[section]
\newtheorem{lemma}[theorem]{Lemma}
\newtheorem{corollary}[theorem]{Corollary}
\newtheorem{proposition}[theorem]{Proposition}

\numberwithin{equation}{section}

\theoremstyle{definition}
\newtheorem{definition}[theorem]{Definition}

\newtheorem{remark}[theorem]{Remark}

\newcommand{\m}{\mathfrak{m}}

\newcommand{\bpm}{\begin{pmatrix}}
\newcommand{\epm}{\end{pmatrix}}
\newcommand{\fs}{\footnotesize}
\newcommand{\ZZ}{\mathbb{Z}}
\newcommand{\HH}{\mathbb{H}}
\newcommand{\CC}{\mathbb{C}}
\newcommand{\RR}{\mathbb{R}}
\newcommand{\NN}{\mathbb{N}}

\newcommand{\HHH}{\mathcal{H}}

\newcommand{\scs}{\scriptsize}
\newcommand{\sst}{\scriptstyle}

\newcommand{\G}[1]{\Gamma_6({#1})}
\newcommand{\PG}[1]{\mbox{P}\Gamma_6({#1})}
\newcommand{\emPG}[1]{\mbox{\em P}\Gamma_6({#1})}
\newcommand{\scsPG}[1]{\mbox{\scs P}\Gamma_6({#1})}
\newcommand{\Gammasechs}{\Gamma_6}
\newcommand{\PGammasechs}{\mbox{P}\Gamma_6}
\newcommand{\Gammazwei}{\Gamma_2}
\newcommand{\PGammazwei}{\mbox{P}\Gamma_2}

\newcommand{\Gaffzwei}{\mathcal{G}_2}
\newcommand{\Gaffzweihat}{\widehat{\mathcal{G}}_2}

\newcommand{\stab}{\mbox{Stab}}
\newcommand{\emstab}{\mbox{\em Stab}}

\newcommand{\Ezwei}{E[2]}

\newcommand{\emtrans}{\mbox{\em Trans}}

\newcommand{\slzwei}{\mbox{SL}_2}
\newcommand{\emslzwei}{\mbox{\em SL}_2}
\newcommand{\pslzwei}{\mbox{PSL}_2}
\newcommand{\empslzwei}{\mbox{\em PSL}_2}

\newcommand{\aff}{\mbox{Aff}}
\newcommand{\emaff}{\mbox{\em Aff}}

\newcommand{\A}{P}
\newcommand{\B}{Q}
\newcommand{\Cc}{R}
\newcommand{\D}{S}

\newcommand{\Cim}{\Cc^m_i}

\newcommand{\Aij}{\A^j_i}
\newcommand{\Bij}{\B^j_i}
\newcommand{\Cij}{\Cc^j_i}

\newcommand{\Dij}{\D^j_i}

\newcommand{\fhathat}{\hat{\hat{f}}}

\newcommand{\rmd}{\mbox{rm}}
\newcommand{\emrmd}{\mbox{\em rm}}

\newcommand{\univZstar}{\widetilde{Z^*}}

\newcommand{\deck}{\mbox{Deck}}
\newcommand{\emdeck}{\mbox{\em Deck}}
\newcommand{\cusps}{\mathcal{C}ps}

\newcommand{\thegroup}{\pm\Gamma_6(2k)}
\newcommand{\id}{\mbox{\footnotesize id}}
\newcommand{\proj}{\mathbb{P}^1}
\newcommand{\projppp}{{{\overbracket[0pt][-2pt]{\mathbb P}^{...}}}{} ^{1}}
\newcommand{\projpppp}{{{\overbracket[0pt][-2pt]{\mathbb P}^{....}}}{} ^{1}}
\newcommand{\into}{\hookrightarrow}
\newcommand{\PGamma}{\mbox{P}\Gamma}

\newcommand{\Qquer}{\overline{\QQ}}
\newcommand{\QQ}{\mathbb{Q}}

\begin{document}

\title{Explicit Teichm\"uller curves with complementary series}

\author{Carlos Matheus}
\address{Carlos Matheus: LAGA, Institut Galil\'ee, Universit\'e Paris 13, 99, av. Jean-Baptiste Cl\'ement, 93430 Villetaneuse, France}
\email{matheus@impa.br}
\urladdr{http://www.impa.br/~cmateus}
\thanks{C.M. was partially supported by Balzan research project of J. Palis.}
\author{Gabriela Weitze-Schmith\"usen}
\address{Gabriela Weitze-Schmith\"usen: Institute for Algebra and Geometry, Karlsruhe Institute of Technology, 
  76128 Karlsruhe, Germany}
\email{weitze-schmithuesen@kit.edu}
\urladdr{http://www.math.kit.edu/iag3/~schmithuesen/}
\date{\today}

\begin{abstract}
We construct an explicit family of arithmetic Teichm\"uller curves  $\mathcal{C}_{2k}$, $k\in\mathbb{N}$, supporting $\textrm{SL}(2,\mathbb{R})$-invariant probabilities $\mu_{2k}$ such that the associated $\textrm{SL}(2,\mathbb{R})$-representation on $L^2(\mathcal{C}_{2k}, \mu_{2k})$ has complementary series for every $k\geq 3$. Actually, the size of the spectral gap along this family goes to zero. In particular, the Teichm\"uller geodesic flow restricted to these explicit arithmetic Teichm\"uller curves $\mathcal{C}_{2k}$ has arbitrarily slow rate of exponential mixing. 
\end{abstract}
\maketitle

\section{Introduction}\label{intro}
Let $\mathcal{H}_g$ be the moduli space of unit area Abelian differentials on a genus $g\geq 1$ Riemann surface. This moduli space is naturally stratified by prescribing the list of orders of zeroes of Abelian differentials: more precisely, 
$$\mathcal{H}_g=\bigcup\limits_{\kappa=(k_1,\dots,k_\sigma)}\mathcal{H}(\kappa)$$
where $\mathcal{H}(\kappa)$ is the set of unit area Abelian differentials with zeroes of orders $k_1, \dots, k_{\sigma}$. Here, we have the constraint $\sum\limits_{j=1}^{\sigma}k_j=2g-2$ coming from Poincar\'e-Hopf formula. The terminology ``stratification'' here is justified by the fact that $\mathcal{H}(\kappa)$ is the subset of unit area Abelian differentials of the complex orbifold of complex dimension $2g+\sigma-1$ of Abelian differentials with list of orders of zeroes $\kappa$.  See~\cite{Masur1,Veech1,Veech2,Zorich} for further details. 

In general, the strata $\mathcal{H}(\kappa)$ are not connected but the complete classification of their connected components was performed by M.~Kontsevich and A.~Zorich~\cite{KZ}. As a by-product of~\cite{KZ}, we know that every stratum has 3 connected components at most (and they are distinguished by certain invariants).

The moduli space $\mathcal{H}_g$ is endowed with a natural $\textrm{SL}(2,\mathbb{R})$-action such that the action of the diagonal subgroup $g_t:=\textrm{diag}(e^t,e^{-t})$ corresponds to the so-called Teichm\"uller geodesic flow (see e.g. \cite{Zorich}). 

After the seminal works of H. Masur~\cite{Masur1} and W.~Veech~\cite{Veech1}, we know that any connected component $\mathcal{C}$ of a stratum $\mathcal{H}(\kappa)$ carries a unique $\textrm{SL}(2,\mathbb{R})$-invariant probability measure $\mu_{\mathcal{C}}$ which is absolutely continuous with respect to the Lebesgue measure in local (period) coordinates. Furthermore, this measure is ergodic and mixing with respect to the Teichm\"uller flow $g_t$. In the literature, this measure is sometimes called \emph{Masur-Veech measure}.

For Masur-Veech measures, A. Avila, S. Gou\"ezel and J.-C. Yoccoz~\cite{AGY} established that the Teichm\"uller flow is exponential mixing with respect to them. Also, using this exponential mixing result and M.~Ratner's work~\cite{Rt} on the relationship between rates of mixing of geodesic flows and \emph{spectral gap} property of $\textrm{SL}(2,\mathbb{R})$-representations, they were able to deduce that the $\textrm{SL}(2,\mathbb{R})$ representation $L^2(\mathcal{C}, \mu_{\mathcal{C}})$ has spectral gap. 

More recently, A.~Avila and S.~Gou\"ezel~\cite{AG} were able to extend the previous exponential mixing and spectral gap results to general \emph{affine} $\textrm{SL}(2,\mathbb{R})$-invariant measures.\footnote{By affine measure we mean that it is supported on a locally affine (on period coordinates) suborbifold and its density is locally constant in affine coordinates. Conjecturally, all $\textrm{SL}(2,\mathbb{R})$-invariant measures on $\mathcal{H}_g$ are affine, and, as it turns out, this conjecture was recently proved by A.~Eskin and M.~Mirzakhani~\cite{EsM}.} 

Once we know that there is spectral gap for these representations, a natural question concerns the existence of \emph{uniform} spectral gap. For Masur-Veech measures, this was informally conjectured by J.-C. Yoccoz (personal communication) in analogy with Selberg's conjecture~\cite{Sel}. On the other hand, as it was recently noticed by A. Avila, 
J.-C. Yoccoz and the first author during a conversation, one can use a recent work of J.~Ellenberg and D.~McReynolds~\cite{em} to produce $\textrm{SL}(2,\mathbb{R})$-invariant measures supported on the $\textrm{SL}(2,\mathbb{R})$-orbits of \emph{arithmetic Teichm\"uller curves} (i.e., \emph{square-tiled surfaces}) along the lines of Selberg's argument to construct non-congruence finite index subgroups $\Gamma$ of $\textrm{SL}(2,\mathbb{Z})$ with arbitrarily small spectral gap. We will outline this argument in Appendix~\ref{a.A}. 

However, it is not easy to use the previous argument to exhibit \emph{explicit} examples of arithmetic Teichm\"uller curves with arbitrarily small spectral gap. Indeed, as we're going to see in Appendix~\ref{a.A}, the basic idea to get arithmetic Teichm\"uller curves with arbitrarily small spectral gap is to appropriately choose a finite index subgroup $\Gamma_2(N)$ of the principal congruence subgroup $\Gamma_2$ of level 2 so that $\mathbb{H}/\Gamma_2(N)$ corresponds to $N$ copies of $\mathbb{H}/\Gamma_2$ arranged \emph{cyclically} in order to slow down the rate of mixing of the geodesic flow (since to go from the $0$th copy to the $[N/2]$th copy of $\mathbb{H}/\Gamma_2$ it takes a time $\sim N$). In this way, it is not hard to apply Ratner's work~\cite{Rt} to get a bound of the form $1\lesssim Ne^{-\sigma(N)\cdot N}$ where $\sigma(N)$ is the size of the spectral gap of $\mathbb{H}/\Gamma_2(N)$. Hence, we get that the size $\sigma(N)$ of the spectral gap decays as $\lesssim\ln(N)/N$ along the family $\mathbb{H}/\Gamma_2(N)$. Thus, we will be done once we can realize $\mathbb{H}/\Gamma_2(N)$ as an arithmetic Teichm\"uller curve, and, in fact, this is always the case by the work of J.~Ellenberg and D.~McReynolds~\cite{em}: the quotient $\mathbb{H}/\Gamma$ can be realized by an arithmetic Teichm\"uller curve whenever $\Gamma$ is a finite index subgroup of $\Gamma_2$ containing $\{\pm Id\}$ (such as $\Gamma_2(N)$). In principle, this could be made explicit, but one has to pay attention in two parts of the argument: firstly, one needs to derive explicit constants in Ratner's work (which is a tedious but straightforward work of bookkeeping constants); secondly, one needs rework J.~Ellenberg and D.~McReynolds article to the situation at hand (i.e., trying to realize $\mathbb{H}/\Gamma_2(N)$ as an arithmetic Teichm\"uller curve). In particular, since the spectral gap decays slowly ($\lesssim\ln(N)/N$) along the family $\mathbb{H}/\Gamma_2(N)$ and the Ellenberg-McReynolds construction involves taking several branched coverings, even exhibiting a single arithmetic Teichm\"uller curve with complementary series  demands a certain amount of effort. 

In this note, we propose an alternative way of constructing arithmetic Teichm\"uller curves with arbitrarily small spectral gap. Firstly, instead of getting a small spectral gap from slow of mixing of the geodesic flow, a sort of ``dynamical-geometrical'' estimate, we employ the so-called \emph{Buser inequality} to get small spectral gap from the Cheeger constant, a more ``geometrical'' constant measuring the ratio between the length of separating multicurves and the area of the regions bounded by these multicurves on the arithmetic Teichm\"uller curve. As a by-product of this procedure, we will have a family $\Gamma_6(2k)$ of finite index subgroups of $\Gamma_6$ (the level $6$ principal congruence subgroup) also obtained by a cyclic construction such that the size of the spectral gap of 
$\mathbb{H}/\Gamma_6(2k)$ decays as $\lesssim 1/k$ (where the implied constant can be computed effectively). Secondly, we combine some parts of Ellenberg-McReynolds methods~\cite{em} with the ones of the second author~\cite{Sc1} to explicitly describe a family of arithmetic Teich\-m\"uller curves birational
to a covering of $\mathbb{H}/\Gamma_6(2k)$ (that is, the \emph{Veech group} of the underlying square-tiled surface is a subgroup $\Gamma_6(2k)$). As a by-product of this discussion, we show the following result:

\begin{theorem}\label{t.A}
Suppose that $k\geq 3$.
\begin{enumerate}
\item[i)]
  For any origami whose Veech group $\Gamma$ is a subgroup of $\thegroup$, its
  Teichm\"uller curve exhibits complementary series and the spectral gap
  of the regular representation associated to $\HH/\Gamma$ is smaller than $1/k$.
\item[ii)]
  The Veech group of the origami $Z_k$ (defined in Definition~\ref{origamiz})
  of genus $48k+3$ and $192k$ squares 
  is contained in $\pm\Gamma_6(2k)$. In particular, its  Teichm\"uller curve
  $\mathcal{C}_{2k}$  exhibits complementary series and this family of origamis 
  gives an example that there is no uniform lower bound for the spectral gaps
  associated to Teichm\"uller curves. 
\end{enumerate} 
\end{theorem}

We organise this note as follows. In Section~\ref{s.CB}, we present the cyclic construction 
leading to the family of finite index subgroups $\Gamma_6(2k)$, $k\in\mathbb{N}$, 
of $\Gamma_6$. We show that the size of the spectral gap of $\mathbb{H}/\Gamma_6(2k)$ 
decays as $\lesssim 1/k$. In a nutshell, we consider the genus 1 curve $\mathbb{H}/\Gamma_6$, 
cut along an appropriate closed geodesic, and glue cyclically $2k$ copies of $\mathbb{H}/\Gamma_6$. 
This will produce the desired family $\mathbb{H}/\Gamma_6(2k)$ because the multicurves consisting of the 
two copies of our closed geodesic at the $0$th and $k$th copies of $\mathbb{H}/\Gamma_6$ divide 
$\mathbb{H}/\Gamma_6(2k)$ into two parts of equal area, see Figure~\ref{f.1}. 
Thus, since the area of $\mathbb{H}/\Gamma_6(2k)$ grows linearly with $k$ while the length of the 
multicurves remains bounded, we'll see that the Cheeger constant decays as $\lesssim 1/k$, and, 
a fortiori, the size of the spectral gap decays as $\lesssim 1/k$ by Buser inequality. 
Proposition~\ref{p.A} and Remark~\ref{gap} then show i) of Theorem~\ref{t.A}.
In Section~\ref{s.EMcRSch}, we describe an explicit family of square-tiled surfaces whose Veech group is 
$\textrm{SL}(2,\mathbb{Z})$. Then, we give a condition that coverings of them have a Veech group
which is contained in $\thegroup$, i.e., its Teichm\"uller curve is birational to a 
covering of $\mathbb{H}/\Gamma_6(2k)$. These two sections can be read independently from each other. 
Finally, in the last section, we prove ii) of Theorem~\ref{t.A} by constructing explicit 
origamis which satisfy the given conditions. In particular, as our ``smallest'' example, we construct
an origami with 576 squares whose Veech group is a subgroup of $\Gamma_6(6)$ and, 
\emph{a fortiori}, whose Teichm\"uller curve exhibits complementary series (see Corollary~\ref{c.origamiZ3}).

\medskip
\noindent\textbf{A word on notation.} During our discussion, sometimes we will need to shift our considerations from $\textrm{SL}_2(\mathbb{Z})$ to $\pslzwei(\ZZ)=\textrm{SL}_2(\mathbb{Z})/\{\pm\textrm{Id}\}$ (and vice-versa). So, in order to avoid potential confusion, each time we have a subgroup $\Gamma$ of $\textrm{SL}_2(\mathbb{Z})$, we will denote by $\textrm{P}\Gamma$ the image of $\Gamma$ under the natural map $\textrm{SL}_2(\mathbb{Z})\to\pslzwei(\ZZ)$.

\medskip
\noindent\textbf{A word on background.} In the sequel, we will assume some familiarity with the notions of \emph{origamis} (i.e., \emph{square-tiled surfaces}), {\em Veech groups} and {\em affine diffeomorphisms}. The reader who wishes more information on these topics may consult e.g. the survey~\cite{HuSc} of P. Hubert and T. Schmidt for a nice account on the subject. Note in particular that in this article the Veech 
group is a subgroup of $\slzwei(\RR)$ and we call its image in $\pslzwei(\RR)$
the {\em projective Veech group}. We denote both, the derivative map from the affine group
to $\slzwei(\RR)$ and its composition with the projection to $\pslzwei(\RR)$, by D and call
also the latter map {\em derivative map}.

\section{Subgroups $\PGammasechs(2k)$ of $\PGammasechs$ with complementary series}\label{s.CB}

Recall that $\slzwei(\ZZ)$ is generated by $T$ and $L$ with 
\[T = \bpm 1 & 1 \\ 0&1 \epm \mbox{ and } L = \bpm 1 & 0\\ 1 & 1 \epm.\]
We will also denote their images in $\pslzwei(\ZZ)$ by $T$ and $L$.  
Recall furthermore that $\PGammazwei$, the image of the principal congruence\footnote{The \emph{principal congruence subgroup} 
$\Gamma_N\subset \slzwei(\ZZ)$ of \emph{level} $N\in\mathbb{N}$ is the kernel of the homomorphism $\slzwei(\ZZ)\to\slzwei(\ZZ/N\ZZ)$ given by reduction modulo $N$ of entries. See e.g. \cite{Bergeron} for more details on these important subgroups of $\slzwei(\ZZ)$.} 
group $\Gammazwei$ of level $2$ in $\pslzwei(\ZZ)$,
is generated by $x = T^2$ and $y = L^2$.
Figure $\ref{cayleygraph}$ below shows the Cayley graph of $\PGammazwei/\PGammasechs$ with
respect to their images, where $\PGammasechs$ is the image 
of the principal congruence group $\Gammasechs$ of level $6$ in $\pslzwei(\ZZ)$.
The Cayley graph is embedded into a genus 1 surface.
\begin{figure}[hb!]
\includegraphics[scale=0.5]{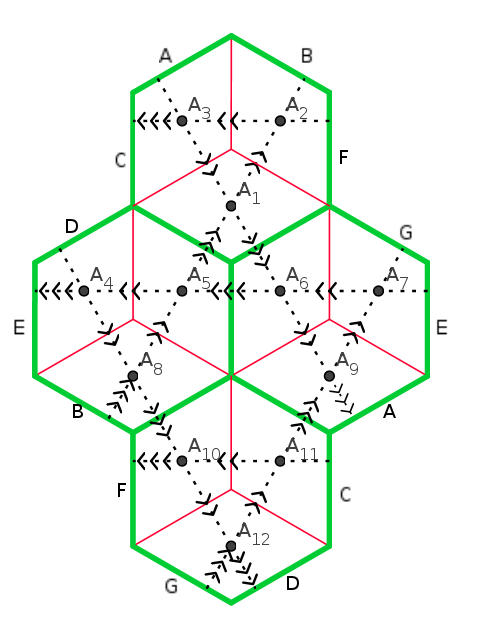}
\caption{ The Cayley graph of $\PGammazwei/\PGammasechs$ with respect to
the generators $T^2$ and $L^2$ drawn on the surface $\HH/\PGammazwei$. 
(Edges with same labels are glued.)\label{cayleygraph} 
Arrows with 
double marking denote multiplication with $x = T^2$,
arrows with triple marking multiplication with
$y = L^2$. The matrices $A_1$, \ldots, $A_{12}$
are given in Remark~\ref{matrices}.}
\end{figure}

\subsection{The group $\PG{2k}$}\label{pg}

In this subsection we present one of our main actors, the group $\PG{2k}$,
and another group (namely, $G_6(2k)$) closely related to it that will be crucial 
for our subsequent constructions.
For this we give a closer look at the group $\PGammasechs$.
As we mentioned above, the two elements
\begin{equation}\label{xy}
  x = \bpm 1 & 2\\0&1\epm \mbox{ and } y = \bpm 1 & 0\\ 2&1 \epm
\end{equation}
generate $\PGammazwei$. They even freely generate it, i.e. 
$\PGammazwei \cong F(x,y)$, the free group in the two generators $x$ and $y$. 
Therefore the action of $\PGammazwei$ on the upper half plane
as Fuchsian group is free and we may consider $\PGammazwei$
as well as the fundamental group of $\HH/\PGammazwei$.
Next, we note that the exact sequence:
\[
  1 \to \PGammasechs \to \pslzwei(\ZZ) \to \pslzwei(\ZZ/6\ZZ) 
  \cong \pslzwei(\ZZ/2\ZZ) \times \pslzwei(\ZZ/3\ZZ)  \to 1
\]
restricts to the exact sequence
\begin{equation}\label{quotient}
  1 \to  \PGammasechs \to \PGammazwei \to {1} \times \pslzwei(\ZZ/3\ZZ) \to 1
\end{equation}
Thus $\PGammasechs$ is normal in $\PGammazwei$ of index 12 and the quotient is isomorphic to 
$\pslzwei(\ZZ/3\ZZ)$. 
\begin{remark}\label{matrices}
The following matrices are a system of 
coset representatives of $\PGammasechs$ in $\PGammazwei$ 
(compare Figure~\ref{cayleygraph}).
\[\begin{array}{l}
A_1 = \mbox{id},\;\;  
A_2  = x \equiv \begin{pmatrix}1&2\\0&1\end{pmatrix},\;
A_3 = x^2 \equiv \begin{pmatrix}1&4\\0&1\end{pmatrix},\\[3mm]
A_4 = y^{-1}x \equiv \begin{pmatrix}1&2\\4&3\end{pmatrix},\;\;
A_5 = y^{-1} \equiv \begin{pmatrix}1&0\\4&1\end{pmatrix} ,\;\;
A_6 = y \equiv \begin{pmatrix}1&0\\2&1\end{pmatrix},\\[2mm]
A_7 = yx^{-1} \equiv  \begin{pmatrix}1&4\\2&3\end{pmatrix},\;\;
A_8 = y^{-1}x^{-1} \cong \begin{pmatrix}1&4\\4&5\end{pmatrix},\\[3mm]
A_9 = yx \cong \begin{pmatrix}1&2\\2&5\end{pmatrix},\;\;
A_{10} = y^{-1}x^{-1}y \cong \begin{pmatrix}3&2\\4&1\end{pmatrix},\\[3mm]
A_{11} = yxy^{-1} \cong \begin{pmatrix}3&2\\4&5\end{pmatrix},\;\;
A_{12} = yxy^{-1}x^{-1 }\equiv  \begin{pmatrix}3&2\\4&3\end{pmatrix} 
\end{array}
\]
\end{remark}

Recall that $\HH/\PGammazwei$ has genus 0 and three cusps, and that one can choose
a fundamental domain which is a hyperbolic geodesic quadrilateral with all four
vertices on the boundary of $\HH$ and such that the edges are paired into neighboured
edges which are identified (such a fundamental domain appears under the name $\mathcal{F}_2$ in Appendix~\ref{a.A}). Since $\PGammasechs$ is a subgroup of $\PGammazwei$
and $\PGammazwei$ acts freely, we obtain an unramified covering 
$\HH/\PGammasechs \to \HH/\PGammazwei$ of degree 12. We have indicated
this covering in Figure~\ref{cayleygraph}. The whole surface is $\HH/\PGammasechs$.
Edges labelled by the same letter are identified by $\PGammasechs$.
All vertices are cusps (also the ones inside the polygon!). Altogether
the surface has 12 cusps and its genus is 1. 
The tessellation into quadrilaterals shows the covering map onto 
$\HH/\PGammazwei$. The thicker edges are all mapped to the same edge
on  $\HH/\PGammazwei$ and the same holds for the thinner edges 

Observe that the fundamental group of $\HH/\PGammasechs$ 
is generated by the directed closed paths indicated in Figure~\ref{generators} below, 
one for each
edge (we call the closed path as well as the corresponding element in
$\PGammazwei$  by the letter that is labelling the edge) and by
positively oriented simple loops around the cusps, one for
each vertex inside the polygon (which we denote by the same letter as the vertex).
Altogether $\PGammasechs$ is generated by the elements $A$, \ldots, $G$
and the loops $L_1$, \ldots, $L_{6}$ (see Lemma~\ref{change} for the exact
definition of the loops)
and it is isomorphic to the
free group $F_{13}$.

\begin{figure}[hb!]
  \includegraphics[scale=0.51]{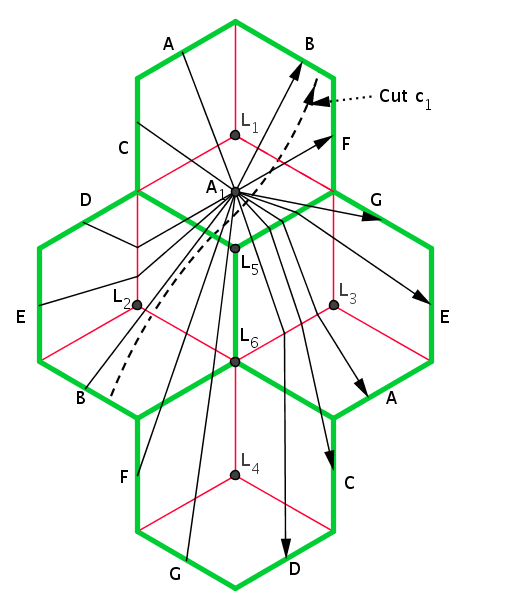}
  \caption
      {
	The generators $A$, \ldots, $G$ of $\pi_1(\HH/\Gammasechs)$. 
	Crossing a thin edge
	counter clockwise is multiplication by $x$;  crossing
	a thick edge clockwise is multiplication by $y$ 
	(compare explanation on p.~\pageref{orientation}).
      }
  \label{generators}
\end{figure}

Furthermore, notice that one nicely obtains
the Cayley graph of $\PGammazwei/\PGammasechs$ embedded
on $\HH/\PGammazwei$, see Figure~\ref{cayleygraph}.
It
is the dual graph to the tessellated surface: each quadrilateral
represents one vertex, i.e. one coset. Two vertices are
connected by an edge if and only if they share a common edge.
For simpler notations
let us choose an orientation on the edges: each edge connects
a vertex with all emanating edges of the same thickness with 
a vertex that is adjacent to edges of different thickness. Choose 
the orientation of an edge from the uni-thickness vertex to
the mixed vertex.
Crossing a thin edge from right to the left then corresponds to 
multiplication by $x$. Crossing a thick edge from left to the
right corresponds to multiplication by $y$, compare 
Figure~\ref{cayleygraph}.\label{orientation}

We now consider a cyclic cover of degree $2k$
of $\HH/\PGammasechs$:
we cut $\HH/\PGammasechs$ along the simple closed path $c_1$ 
crossing the edge $B$ indicated
in Figure~\ref{generators}, take $2k$ copies of this
slitted surface 
and reglue them in a cyclic order.\footnote{Here, we started with $\HH/\PGammasechs$ instead of $\HH/\PGammazwei$ because an efficient usage of Buser inequality (to detect complementary series for cyclic covers) depends on the fact that we can select a non-separating loop $c_1$ on our initial surface. Of course, such a choice of $c_1$ is possible on the genus $1$ surface $\HH/\PGammasechs$ but not on the genus $0$ surface $\HH/\PGammazwei$.} See Figure~\ref{f.1} below for a schematic picture. The fundamental 
group of this covering will be the group $\PG{2k}$.
We give a definition of this group using the monodromy of the covering.

\begin{figure}[h!]
\includegraphics[scale=0.6]{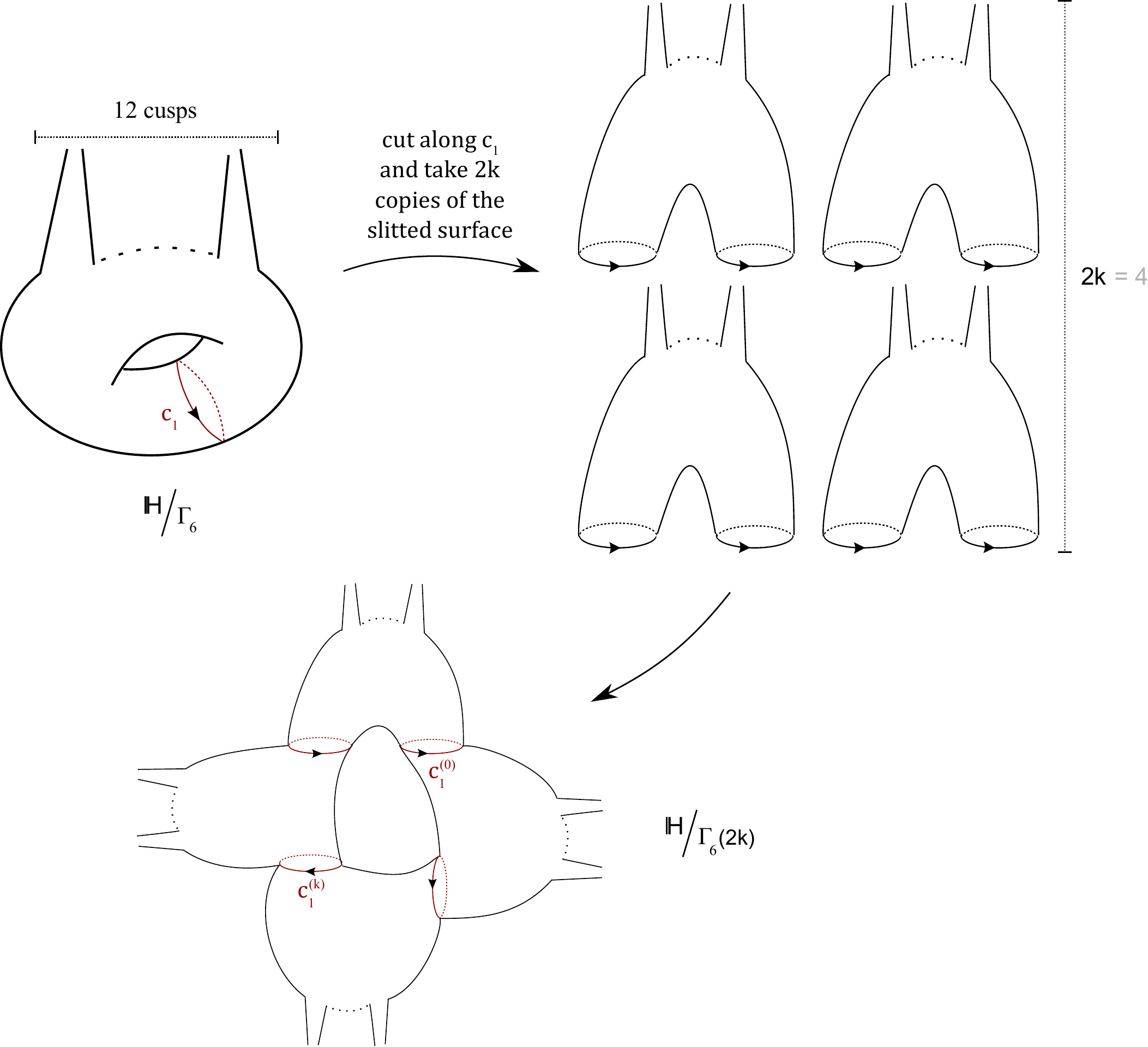}
\caption{Schematic description of $\mathbb{H}/\PG{2k}$}\label{f.1}
\end{figure}

\begin{definition}\label{defpg}
Let $m: \PGammasechs \to \ZZ/(2k\ZZ)$ be the group homomorphism
defined by:
\[
  \begin{array}{l}
    A \mapsto 1,\ B \mapsto 0,\ C \mapsto 1, 
    D \mapsto 1,\ E \mapsto 1,\  F \mapsto 0,\ 
    G \mapsto 0,\\
    L_i \mapsto 0 \mbox{ for } i \in \{1, \ldots, 6\}
  \end{array}
\]
Define $\PG{2k}$ to be the kernel of $m$.
In particular $\PG{2k}$ is a normal subgroup
of index $2k$ in $\PGammasechs$ and the quotient is 
the cyclic group $\ZZ/(2k\ZZ)$. Denote its
preimage in $\slzwei(\ZZ)$ by $\pm \Gammasechs(2k)$.
\end{definition}

Observe that the map $m$ defined in Definition~\ref{defpg}
assigns each element in the fundamental group of  $\HH/\PGammasechs$
the oriented intersection number with the curve $c_1$ modulo $2k$.\\

One of the main goals of this note is to obtain origamis whose Veech groups
are contained in $\PG{2k}$. However for the proof we will for technical reasons
mainly work
with another subgroup $G_6(2k)$ of $\PGammasechs$ (which turns out 
to be the image of $\PG{2k}$ under an isomorphism $\gamma$ of $\PGammasechs$) defined as follows.

\begin{definition}\label{defpg-1}
Let $m_2: \PGammasechs \to \ZZ/(2k\ZZ)$ be the group homomorphism
defined by: 
\[
  \begin{array}{l}
    A \mapsto 1,\ B \mapsto -1,\ C \mapsto 1,\ 
    D \mapsto 1,\ E \mapsto 0,\  F \mapsto -1,\ 
    G \mapsto -1,\\ 
    L_i \mapsto 0\ \mbox{ for } i \in \{1, \ldots, 6\}
  \end{array}
\]
Define $G_6(2k)$ to be the kernel of $m_2$.
In particular $G_6(2k)$ is again a normal subgroup
of index $2k$ in $\PGammasechs$ and the quotient is 
the cyclic group $\ZZ/(2k\ZZ)$. 
\end{definition}

One directly observes from Figure~\ref{cuttwo} below 
that $m_2$ gives the oriented intersection number
modulo $2k$
with the simple closed curve $c_2$ crossing the
edges $D$ and $G$ shown in 
the figure.

\begin{lemma}\label{change}
Let $\gamma$ be the automorphism of $\PGammazwei \cong F_2$ defined by
\[\gamma: x \mapsto y,\ y \mapsto x^{-1}y\]
Then $G_6(2k) = \gamma(\PG{2k})$.
\end{lemma}

\begin{proof}
First, observe that $\gamma$ (considered as an automorphism of $\PGammazwei$)
restricts to an automorphism of $\PGammasechs$. To see this, we write
the generators of $\Gammasechs$ in terms of $x$ and $y$. This
can again be read off from Figure~\ref{generators} (we do some choices
for the loops $L_1$, \ldots, $L_6$ here):
\begin{equation}
  \begin{array}{l}
    A = yxyx,\ B = xyxy,\ C = yxy^{-2}x,\ 
    D = yxy^{-1}x^{-1}yx^{-1}y,\\[1mm]
    E = yx^{-1}y^{-1}x^{-1}y,\ 
    F = xy^{-2}xy, G = yx^{-1}yx^{-1}y^{-1}xy\\[1mm]
    L_1 = x^3,\ L_2 = y^{-1}x^{3}y,\ L_3 = yx^3y^{-1},\ 
    L_4 = y^{-1}x^{-1}yx^3y^{-1}xy,\\[1mm]
    L_5 = y^{-3},\ 
    L_6 = y^{-1}(x^{-1}y)^3y 
  \end{array}
\end{equation}
Now we apply $\gamma^{-1}:  x \mapsto xy^{-1}, y \mapsto x$ to them and obtain:
\[
   \begin{array}{l}
     \gamma^{-1}(A) = x^2y^{-1}x^2y^{-1} = L_1A^{-1}L_3,\\[1mm]
     \gamma^{-1}(B) = xy^{-1}x^2y^{-1}x = FL_4L_6L_5C,\\[1mm]
     \gamma^{-1}(C) = x^2y^{-1}x^{-1}y^{-1} = L_1A^{-1},  \quad\\[1mm]
     \gamma^{-1}(D) =  x^2y^{-1}x^{-1}y^2 = L_1A^{-1}L_5^{-1},\\[1mm]
     \gamma^{-1}(E) = xyx^{-2}y = BL_2^{-1}, \quad
     \gamma^{-1}(F) = xy^{-1}x^{-1}y^{-1}x = FL_6L_5C,\\[1mm]
     \gamma^{-1}(G) = xy^2x^{-1}y^{-1}x = BL_6L_5C,  \quad
     \gamma^{-1}(L_1) = (xy^{-1})^3 = FG^{-1},\\[1mm]
     \gamma^{-1}(L_2) = (y^{-1}x)^3 = D^{-1}C,\\[1mm]
     \gamma^{-1}(L_3) = x(xy^{-1})^3x^{-1} = L_1A^{-1}L_3EL_2B^{-1}\\[1mm]
     \gamma^{-1}(L_4) = x^{-1}y(xy^{-1})^3y^{-1}x = C^{-1}L_5^{-1}L_6^{-1}L_5C,  \quad
     \gamma^{-1}(L_5) = x^{-3} = L_1^{-1}\\[1mm]
     \gamma^{-1}(L_6) = x^{-1}y^3x = C^{-1}A \\
   \end{array} 
\]
Observe now that for each generator $X$ we have $ m(\gamma^{-1}(X)) = -m_2(X)$.
We thus obtain:
\begin{eqnarray*}
G_6(2k) &=& \mbox{kernel}(m_2) = \mbox{kernel}( -m_2) =\mbox{kernel}(m \circ \gamma^{-1})
  \\ &=& \gamma(\mbox{kernel}(m)) = \gamma(\PG{2k}).
\end{eqnarray*}
\end{proof}

\begin{remark} Alternatively, for the proof of Lemma \ref{change} one could
  check that $\gamma$ preserves loops, 
  therefore induces a homeomorphism of the surface by the
  Dehn-Nielsen Theorem for punctured surfaces and that this one
  maps $c_1$ to $c_2$.
\end{remark}

For later usage, we want to further describe
the action of $\PGammazwei$ on the cosets
of $G_6(2k)$ in $\PGammazwei$. Denote in the following
by $\mathcal{C}(H:U)$ the set of cosets $U\cdot h$ ($h \in H$)
of a subgroup $U$ in a group $H$.
Firstly, we use the system of coset representatives $A_1$, \ldots, $A_{12}$
of $\PGammasechs$ in $\PGammazwei$ defined in Remark~\ref{matrices}.
Secondly, for each $A_i$ we define a drift $j_i$
as follows:
\[
    \begin{array}{l}
      j_1 = 0, j_2 = 0, j_3 = 0, j_4 = 0, j_5 = 0, 
      j_6 = 0,\\[1mm] 
      j_7 = 1, j_8 = 1, j_9 = 1, j_{10} = 1,
      j_{11} = 1, j_{12} = 1 
    \end{array}
\]

\begin{remark}\label{identify}
We fix the following identification 
between $\mathcal{C}(\PGammazwei:G_6(2k))$ and
$\mathcal{C}(\PGammazwei:\PGammasechs) \times \ZZ/(2k\ZZ)$:
For $A \in \PGammazwei$,
write $A = A'\cdot A_i$ with $A_i$ the coset representative
of $A$ from  Remark~\ref{matrices}, $A' \in \PGammasechs$ 
and define $\delta_{A} = m_2(A') + j_i$.  Then we define the bijection:
\[
   \mathcal{C}(\PGammazwei:G_6(2k)) \to 
   \mathcal{C}(\PGammazwei:\PGammasechs) \times \ZZ/(2k\ZZ),\;\; 
   G_6(2k)\cdot A \mapsto (\PGammasechs\cdot A, \delta_{A})  
\]
For easier notation, we will label the elements in $\mathcal{C}(\PGammazwei:G_6(2k))$
by the corresponding pair in $\{1, \ldots, 12\} \times \ZZ/(2k\ZZ)$:
\[
  G_6(2k)\cdot A \mapsto (i,\delta_A) \mbox{ with $i$ the $i$ from the $A_i$ above 
  and $\delta_{A}$ as above.}
\]
\end{remark}

Observe that the choice of the coset representatives $A_1$, \ldots, $A_{12}$
above corresponds to the choice of fixed paths on $\HH/\PGammasechs$
between the midpoint of the quadrilateral corresponding to $A_1$
and the midpoint  of the quadrilateral corresponding to the respective
$A_i$, see Figure~\ref{cuttwo}.
The drift $\delta_A$ for some $A = A'\cdot A_i$ in $\PGammazwei$
is then the intersection number of
the path corresponding to $A'\cdot A_i$ with $c_2$.
Crossing  $c_2$ from left to right leads one copy higher in our $2k$ copies.

\begin{figure}[h1]
  \includegraphics[scale = .5]{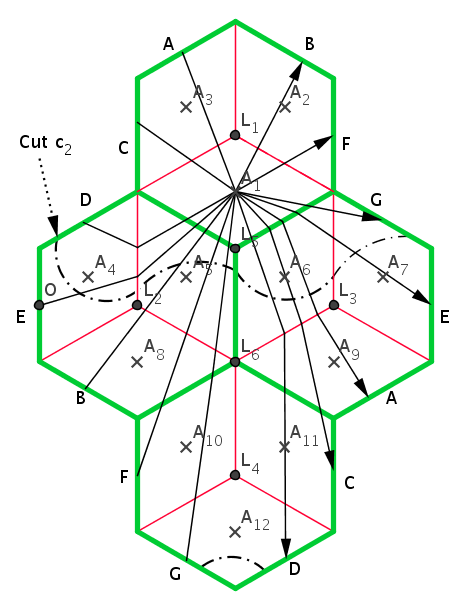}
  \caption{The second cut on the surface $\HH/\PGammasechs$. 
  Preimages of the base point of 
    $\HH/\PGammazwei$ are indicated by little crosses.
    	Crossing a thin edge
	counter clockwise is multiplication by $x$;  crossing
	a thick edge clockwise is multiplication by $y$ 
	(compare explanation on p.~\pageref{orientation}).
    \label{cuttwo}
  }
\end{figure}

In particular, we get the following proposition:
\begin{proposition}\label{actiononcosets}
Identify by Remark~\ref{identify} the set of cosets $\mathcal{C}(\PGammazwei:G_6(2k))$ 
with $\{1, \ldots, 12\}\times \ZZ/(2k\ZZ)$.
The action of $\PGammazwei$ on $\mathcal{C}(\PGammazwei:G_6(2k))$ 
by right multiplication
is then given by the permutations:
\[
   \begin{array}{ll}
     x:& ((1,j)\ (2,j)\ (3,j))\circ ((4,j)\ (8,j+1)\ (5,j+1)) \circ\\ 
       &  ((6,j)\ (9,j+1)\ (7,j+1))\circ ((10,j)\ (12,j)\ (11,j))\\
     y:&  ((1,j)\ (6,j)\ (5,j+1)) \circ ((2,j)\ (8,j)\ (10,j)) \circ \\
       &  ((3,j)\ (11,j)\ (9,j)) \circ ((4,j)\ (7,j+1)\ (12,j))
   \end{array}  
\]
\end{proposition}

\begin{proof}
Using the comment in the paragraph before the proposition,
this can be read off easily from Figure~\ref{cuttwo}.
\end{proof}

\subsection{Size of the spectral gap of \boldmath{$\mathbb{H}/\PG{2k}$}}

In the sequel, we want to estimate $\lambda_{2k}:=\lambda_1(\mathbb{H}/\PG{2k})$ the first eigenvalue of the Laplacian on the hyperbolic surface $\mathbb{H}/\PG{2k}$. In order to do so, we recall the Buser inequality (see Buser~\cite{B} and Lubotzky~\cite[p. 44]{L}):
\begin{equation}\label{e.Buser-Cheeger}
\sqrt{10\lambda_{2k}+1}\leq 10 h_{2k}+1
\end{equation}
where 
$$h_{2k}:=\min\limits_{\substack{\gamma \textrm{ multicurve of }\mathbb{H}/\scsPG{2k} \\ \textrm{ separating it into} \\ \textrm{ two connected components }A \textrm{ and }B}}  
   \frac{\ell(\gamma)}{\min\{\textrm{area}(A),\textrm{area}(B)\}}.
$$

For the case at hand, we apply Buser's inequality with the multicurve obtained by the disjoint union of two copies $c_1^{(0)}$ and $c_1^{(k)}$ of $c_1$ on the $0$th and $k$th copies of $\mathbb{H}/\Gamma_6$ inside $\mathbb{H}/\PG{2k}$, see Figure~\ref{f.1}. Here, we recall that $c_1$ is the simple closed geodesic of $\mathbb{H}/\Gamma_6$ in Figure~\ref{generators} 
connecting the ``B-sides'' indicated in the picture, that is, $c_1$ is the simple closed geodesic along which we cut $\mathbb{H}/\Gamma_6$, take several copies of the outcome of this, and reglue them cyclically (see Definition~\ref{defpg}). In this way, by definition, we have that 
\begin{equation}\label{e.Cheeger-constant}h_{2k}\leq \frac{2\ell(c_1)}{k\cdot\textrm{area}(\mathbb{H}/\Gamma_6)}
\end{equation}
On the other hand, one can check that the hyperbolic matrix in $\textrm{P}\Gamma_6$ associated to $c_1$ is $\rho(c_1)=\left(\begin{array}{cc}29 & 12 \\ 12 & 5\end{array}\right)$, so that $\ell(c_1)=2\textrm{ arc cosh}(\frac{|\textrm{tr}\rho(c_1)|}{2})=2\textrm{ arc cosh}(17)$. Combining this with the fact $\textrm{area}(\mathbb{H}/\Gamma_6) = 24\pi$, we can conclude from \eqref{e.Buser-Cheeger} and \eqref{e.Cheeger-constant} that 
$$\sqrt{10\lambda_{2k}+1}\leq\frac{10\cdot 2\cdot 2\textrm{ arc cosh(17)}}{k\cdot 24\pi}+1,$$
i.e.,
$$\lambda_{2k}\leq\frac{10\cdot\textrm{ arc cosh}(17)^2}{9\pi^2(2k)^2}+ \frac{2\cdot\textrm{ arc cosh(17)}}{3\pi (2k)}.$$
Since $\textrm{ arc cosh}(17)<3.5255$, we get, for every $k\geq 3$, 
\begin{equation}\label{e.l2k}
\lambda_{2k}\leq \left(\frac{10\cdot(3.5255)^2}{9\pi^2(2k)}+\frac{2\cdot(3.5255)}{3\pi}\right)\cdot\frac{1}{2k} < \frac{1}{2k}\leq \frac{1}{6}<\frac{1}{4}
\end{equation} 

Next, denoting by $\lambda(\Gamma)$ the first eigenvalue of the Laplacian on a finite area hyperbolic surface $\mathbb{H}/\Gamma$, we note that $\lambda(\Gamma)\leq \lambda(\Gamma')$ whenever $\Gamma$ is a finite index subgroup of $\Gamma'$, that is, the first eigenvalue of the Laplacian of a finite area hyperbolic surface can't increase under finite covers. 

Finally, we recall that the presence of complementary series on the regular representation of $SL(2,\mathbb{R})$ on $L^2(\mathbb{H}/\Gamma)$ is equivalent to the fact that $\lambda(\Gamma)<1/4$ (see~\cite{Rt} and references therein for more details). 

In other words, by putting these facts together, we proved the following result:
\begin{proposition}\label{p.A}Let $\Gamma$ a finite index subgroup of $\emPG{2k}$. 
If $k\geq 2$, then the regular representation of $\textrm{SL}(2,\mathbb{R})$ on $L^2(\mathbb{H}/\Gamma)$ exhibits complementary series.
\end{proposition}

\begin{remark}\label{gap}
  In fact, since the size of the spectral gap $\sigma(\Gamma)$ of the regular representation 
  associated to $\mathbb{H}/\Gamma$ relates to the first eigenvalue $\lambda(\Gamma)$ through the equation
  $$\sigma(\Gamma)=1-\sqrt{1-4\lambda(\Gamma)}$$ 
  whenever $\Gamma$ has complementary series (see~\cite{Rt}), we see that the spectral gap 
  $\sigma(\Gamma)$ for $\Gamma$ a finite index subgroup of $\PG{2k}$ obeys the following inequality
  $$\sigma(\Gamma)\leq \sigma(\PG{2k})\leq \frac{2\lambda_{2k}}{\sqrt{1-4\lambda_{2k}}}\leq 
  2\sqrt{3}\lambda_{2k}<\frac{\sqrt{3}}{k}$$
  for all $k\geq 3$. Here, we used that, from~\eqref{e.l2k}, one has $1/\sqrt{1-4\lambda_{2k}}\leq \sqrt{3}$ and $\lambda_{2k}\leq 1/(2k)$ when $k\geq 3$.
\end{remark} 

\begin{remark}\label{noncongruence}
It follows from Selberg's $3/16$ Theorem~\cite{Sel} and the estimate~\eqref{e.l2k} above that, for each $k\geq 3$, $\PG{2k}$ is \emph{not} a congruence subgroup of $SL(2,\mathbb{Z})$, i.e., there is no $N\in\mathbb{N}$ such that $\PG{2k}$ contains the principal congruence group $\Gamma_N$ of level $N$ (consisting of all matrices of $SL(2,\mathbb{Z})$ which are the identity modulo $N$). Indeed, if $\Gamma_N\subset \PG{2k}$ for some $N$, we would have $\lambda(\Gamma_N)\leq \lambda_{2k}\leq 1/6$, a contradiction with Selberg's $3/16$ Theorem (saying that $\lambda(\Gamma_N)\geq 3/16$ for all $N\in\mathbb{N}$).
\end{remark}

\section{Square-tiled surfaces with Veech group inside $\pm\Gamma_6(2k)$}\label{s.EMcRSch}

In this section we describe a construction to obtain 
translation surfaces 
whose Veech groups are subgroups of $\pm\G{2k}$ ($k \in \NN$), see Definition~\ref{defpg}.
We will use very special translation surfaces, called 
{\em origamis} or {\em square-tiled surfaces}.
Recall that an origami is a finite covering $p:X \to E$
from a closed surface $X$ to the torus $E = \CC/(\ZZ \oplus \ZZ i)$
which is ramified at most over the point $\infty = (0,0)$ on $E$, see e.g.  \cite{Sc2}.
We may pull back the natural Euclidean translation
structure on $E = \CC/(\ZZ\oplus\ZZ i)$ to $X$ and obtain a  translation surface
with cone points singularities
which is tiled by $d$ squares, where $d$ is the degree of the covering $p:X \to E$.
Origamis can equivalently be defined to be translation surfaces
obtained from gluing finitely many copies of the Euclidean unit square
along their edges by translations. 
We can present an origami by a pair of 
transitive permutations $\sigma_a$ (horizontal
gluings) and $\sigma_b$ (vertical gluings) in the symmetric group $S_d$
well defined up to simultaneous conjugation.
Figure~\ref{Ezwei} shows e.g. the origami $E[2]$ which is an 
origami of degree 4. Indeed,  as a translation surface
it is isomorphic to $\CC/(2\ZZ \oplus 2\ZZ i)$ and the multiplication by 2
is the corresponding covering $[\cdot 2]: E[2] \to E$ of degree 4
to the torus. It can be described by the pair of permutations $\sigma_a = (1,2)(3,4)$
and $\sigma_b = (1,3)(2,4)$. 

\begin{figure}[ht!]
\begin{center}
\setlength{\unitlength}{.75cm}
\begin{picture}(2,2)
\put(0,0){\framebox(1,1){}}
\put(1,0){\framebox(1,1){}}
\put(0,1){\framebox(1,1){}}
\put(1,1){\framebox(1,1){}}
\put(-.5,-.5){$\A$}
\put(.5,-.5){$\B$}
\put(-.5,.52){$\Cc$}
\put(.5,.52){$\D$}
\put(0,0){\circle*{.2}}
\put(1,0){\circle*{.2}}
\put(0,1){\circle*{.2}}
\put(1,1){\circle*{.2}}
\end{picture}
\end{center}
\medskip
\caption{
The trivial $2 \times 2$ - origami $\Ezwei$:
opposite edges are glued by translations.
\label{Ezwei}
}
\end{figure}
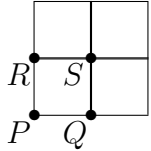

For an origami $p:X\to E$ we will
always consider $X$ endowed 
with the lifted translation structure. 
We denote by $\aff(X)$ the group of orientation preserving 
affine homeomorphisms of $X$ and by $\Gamma(X)$ the 
Veech group of $X$, i.e. the image of the derivative map 
$D: \aff(X) \to \slzwei(\RR)$.
Observe that the {\em Veech group of an origami} as considered in 
\cite{Sc2} consists only of the derivatives of those affine homeomorphisms that preserve 
the fibre of the point $\infty$. However, if the derived vectors of the saddle 
connections on $X$
span the lattice $\ZZ \oplus \ZZ i$, this is indeed the full Veech group, see
e.g. \cite[Lemma 2.3]{HL}.
We will furthermore frequently use that in this case each affine homeomorphism
of $X$ descends via the covering map $p$ to $E$. If $p$ factors through
$[\cdot 2]$, i.e. there is a covering $p': O \to E[2]$ such that
$p = [\cdot 2] \circ p'$, then the same holds for $p'$ and $E[2]$, see e.g.  
\cite[Prop. 2.6]{Sc2}.\\
Following the notation in \cite{GJ} we denote 
for a covering $h:X \to Y$ of translation surfaces  
\[
  \begin{array}{lcl}
    \aff_h(X) &=& 
       \{f \in \aff(X)| f \mbox{ descends via $h$ to } Y\} \mbox{ and }\\
    \aff^h(Y) &=& 
       \{f \in \aff(Y)| f \mbox{ lifts via $h$ to some homeomorphism of } X\}. 
  \end{array}
\]
Recall finally that the Veech group of the torus 
$E = \CC/(\ZZ \oplus \ZZ i)$ itself
is the group $\slzwei(\ZZ)$: An affine homeomorphism 
\[\CC \to \CC,\;\; {x\choose y} \mapsto A \cdot {x\choose y} + b\]
 of the universal
covering $\CC$ descends to an affine homeomorphism $f$ of $E$, 
if and only if $A$ is in $\slzwei(\ZZ)$.
If we require in addition that $f$ fixes the point $\infty = (0,0)$
on $E$, then we obtain for each $A\in \slzwei(\ZZ)$
a unique affine homeomorphism on $E$
with derivative $A$. Thus $D: \aff(E) \to \slzwei(\RR)$ restricts to an isomorphism
\[
  \begin{array}{l}
    \aff_0(E) = 
        \{f \in \aff(E)| f(0,0) = (0,0)\} \; \cong \; \slzwei(\ZZ).
  \end{array}
\]

In the following construction we 
start with the torus $E[2] = \CC/(2\ZZ \oplus 2\ZZ i)$ 
(see Figure~\ref{Ezwei}).
We fix the four points
$P = \infty = (0,0)$, $Q = (1,0)$, $S = (0,1)$ and $R = (1,1)$
on $E[2]$. Then the derivative map
$D: \aff(E[2]) \to \slzwei(\RR)$ further restricts to an isomorphism
\[
  \begin{array}{lcl}
    \aff_1(E[2]) 
    &=& \{f \in \aff(E[2])| f(P) = P, f(Q) = Q \mbox{ and } f(R) = R\}\\
    &=& \{f \in \aff(E[2])| f(P) = P, f(Q) = Q \mbox{ and } f(S) = S\}\\
	&\cong& \Gammazwei.
  \end{array}
\]

We will now construct origamis with Veech 
group $\Gamma = \pm\G{2k}$ proceeding in
the following two main steps.
In the first step, following the ideas from part of 
\cite{em}, we construct a covering
$\tilde{q}: Y \to E[2]$ with the following properties:
\begin{itemize}
\item[A)] 
  $\tilde{q}$ is ramified precisely over $P = (0,0)$, $Q = (1,0)$ and $S = (1,1)$.
\item[B)]
  The degree of the covering is equal to the index $[\Gammazwei:\Gamma]$.
\item[C)]
  All affine homeomorphisms in $\aff_1(E[2])$ have some lift in 
  $\aff(Y)$ via $\tilde{q}$.
  Denote by $({\aff_1})_{\tilde{q}}(Y)$ the group of those lifts. In
  particular, $D(({\aff_1})_{\tilde{q}}(Y))$ is equal to $\Gammazwei$.
\item[D)]
  There is a bijection $\theta$ between the fibre $ \tilde{q}^{-1}(R)$
  of $R$ and the set of left cosets $\Gammazwei/\Gamma$.
  Furthermore $\theta: \tilde{q}^{-1}(R) \to \Gammazwei/\Gamma$ 
  is equivariant
  with respect to $D:({\aff_1})_{\tilde{q}}(Y) \to \Gammazwei$.
  Here $({\aff_1})_{\tilde{q}}(Y)$ acts naturally on the fibre $\tilde{q}^{-1}(R)$
  and $\Gammazwei$ acts on the set of left cosets $\Gammazwei/\Gamma$
  by multiplication from the left. I.e.
  \[\forall\, f \in ({\aff_1})_{\tilde{q}}(Y): 
  (\cdot D(f)) \circ \theta = \theta \circ f,\] where $(\cdot D(f))$  denotes
  the multiplication with $D(f)$ from the left.\\
  Let $R_{\id}$ be the point in the fibre of $R$ with $\theta(R_{\id}) = \id \cdot \Gamma$.
\end{itemize}

In the second step\label{secondmainstep}, we obtain the desired origami
following the construction in \cite[Chap. 5]{Sc1} by choosing
a covering $r:Z \to Y$ with a suitable ramification
behaviour. For example, we choose the covering
in such a way that $\tilde{q}\circ r$
is ramified differently
above $P$, $Q$ and $S$. Furthermore
$r$  is ramified  above $R_{\id}$ 
differently than above all other points 
in $\tilde{q}^{-1}(R)$.
Thus if $f \in \aff(Z)$ descends
via $r$ to $\bar{f} \in \aff(Y)$,
then $\bar{f}$ is in $({\aff_1})_{\tilde{q}}(Y)$ and must fix $R_{\id}$. 
By the property D)  that we required for $\tilde{q}$, 
the second condition means 
that
left multiplication with $D(\bar{f})$ fixes the coset of the 
identity. Thus $D(f) = D(\bar{f})$ is in $\Gamma$.\\
It follows from the last paragraph that 
if each affine homeo\-morphism of $Z$
descends to $Y$, then we are done. To achieve this is a
technical difficulty, we have to take extra care of.
This is where we need that our group $\Gamma$
is the group  $\pm\G{2k}$. We study the action of 
$\aff(Y)$ on $\tilde{q}^{\,-1}(\{P,Q,R,S\}) \subset Y$. 
The Veech group of $Y$ turns out
to be the full group $\slzwei(\ZZ)$ (see Proposition~\ref{vgY}).
We find
a partition of $\tilde{q}^{-1}(\{P,Q,R,S\})$, 
such that if $r$ is ramified differently above different classes in the partition,
then all affine homeomorphisms descend  and we are done.
Proposition~\ref{mainprop} does this job.

\subsection{The Ellenberg/McReynolds construction}

In this paragraph we explain how one obtains 
a covering $\tilde{q}$ with the properties 
A) - D) as above. For this part
we just need that the group $\Gamma$ is a subgroup of $\Gammazwei$
containing $-I$.
As stated above we follow part of the proofs in \cite{em}, which 
gives us this beautiful construction.
We briefly describe the geometric interpretation behind it
and how this leads to what we want.
In Section~\ref{sectionX} and 
Section~\ref{step2}, we will explicitly define the origami
$Y$ (see Definition~\ref{defy}) and show that it has all properties 
which we need.
Hence the following
paragraph is just for giving a motivation how we obtained the surfaces
$Y$, but not necessary for the logic of our proofs. \\
It seems natural to describe the construction in terms of {\em fibre products} also called
{\em pullbacks}. 
They are more commonly used in the context of Algebraic geometry.
But recall that we may also take fibre products of topological spaces 
(see e.g. \cite[11.8, 11.12(2)]{AHS}). Our topological spaces here will be punctured
closed Riemann surfaces, i.e. closed Riemann surfaces with finitely many points removed,
thus we notate them in the following 
by $B^*$, $X^*$, ${X_1}^*$ and ${X_2}^*$.
For two continuous maps  
$p_1:{X_1}^* \to B^*$ and $p_2:{X_2}^* \to B^*$ the fibre product 
is the topological space 
\[X^* = \{(a,b) \in {X_1}^* \times {X_2}^*|\; p_1(a) = p_2(b)\}\]
endowed with the subspace topology 
together with the two maps
$q_1: X^* \to {X_1}^*$ and $q_2:X^* \to {X_2}^*$ which are just the projections.
Thus we have a commutative diagram:
\[\xymatrix{
  \ar @{} [drr] |{\square}
   X^* \ar[rr]^{q_1}\ar[d]_{q_2} &&
       {X_1}^* \ar[d]^{p_1} \\ 
  {X_2}^*  \ar[rr]^{p_2} &&  B^* 
}\]
The square in the middle denotes that this is a fibre product.
The fibre product has a nice universal property (which is actually used 
to define it for general categories, see e.g. \cite[11.8]{AHS}), but we do not further need it here.
We will just need the following properties:
\begin{itemize}
\item[F1)] 
  If $p_2$ is an embedding, then $q_1$ is one as well, and
  $X^{*}$ is the preimage of 
  ${X_2}^*$ via $p_1$.
\item[F2)]
  If $p_1$ and $p_2$ are unramified coverings of punctured closed 
  Riemann surfaces, then $q_1$ and $q_2$
  are also unramified coverings of punctured closed Riemann surfaces. $X^*$
  does not have to be connected, even if ${X_1}^*$ and ${X_2}^*$ are.
  See e.g. \cite[1.2]{FlorianDiss} for an explicit definition
  of the fibre product in the case of unramified (not necessarily connected) 
  coverings and an explicit calculation of the monodromies of 
  the coverings $q_1$ and $q_2$ in terms of the monodromies of $p_1$ and $p_2$.
\item[F3)]
  If,  in addition to the conditions in F2), 
  $X^*$ is connected, then we can as well describe
  the fibre product $X^*$ by its fundamental group. 
  More precisely, the fundamental
  groups of ${X_1}^*$, ${X_2}^*$ and $X^*$ then embed via the coverings
  into $\pi_1(B^*)$ and $\pi_1(X^*) = \pi_1({X_1}^*) \cap \pi_1({X_2}^*)$. This
  can be seen e.g. from Theorem~B in \cite{FlorianDiss} (in this
  case $\m_f \times m_g$ acts transitively).
\item[F4)]
  Suppose now that $p_1$ is a covering as before and 
  $p_2$ is an embedding of punctured
  closed surfaces. Then the map $(p_2)_{*}: \pi_1({X_2}^*) \to \pi_1(B^*)$ 
  induced by $p_2$ between fundamental groups is a surjection,
  whereas $ \pi_1({X_1}^*)$ embeds into $\pi_1(B^*)$ via $({p_1})_{*}$.
  The map $q_2$ is an unramified covering and
  we have:
  \[\pi_1(X^*) = (p_2)_{*}^{-1}(\pi_1({X_1}^*)) \subseteq \pi_1({X_2}^*)\]

\end{itemize}

In the following we use the degree 2 covering $h:E[2] \to \proj(\CC)$
defined as the quotient by the affine involution $\iota$ on $E[2]$ with derivative
$-I$ and fixed points $P$, $Q$, $R$ and $S$. The covering is
ramified at $P$, $Q$, $R$ and $S$ and we may choose the isomorphism 
$E[2]/\iota \cong \proj(\CC)$ in such a way that their images are $0$, $1$, $\infty$
and $\lambda = -1$. This leads to an unramified covering
\[h: E[2]^* = E[2]\backslash\{P,Q,R,S\} \to \projpppp = \proj\backslash\{0,1,\infty,\lambda\}.\]
We furthermore use the utile coincidence that we have:
\[\projppp = \proj(\CC)\backslash\{0,1,\infty\} \cong \HH/\PGamma(2) \cong \HH/\Gammazwei.\]
Thus the embedding  $\Gamma \into \Gammazwei$ defines an unramified
covering \[\beta:\hat{Y}^* = \HH/\Gamma \to \projppp.\]
We now remove the additional point $\lambda = -1$ on $\projppp$, consider
the embedding $i:\projpppp \into \projppp$ and define
$\hat{Y}^{**} = \hat{Y}^{*}\backslash \beta^{-1}(\lambda) = \beta^{-1}(\projpppp)$. 
We then 
take the fibre product $Y^*$ of $\beta: \hat{Y}^{**} \to \projpppp$ and $h: E[2]^* \to \projpppp$
and define $\tilde{q}$ to be the projection $Y^* \to E[2]^*$.
Hence we have the following commutative diagram:
\begin{equation}\label{lm-construction}
  \xymatrix{
    \ar @{} [drr] |{\square}
     Y^* \ar[rr] \ar[d]^{\tilde{q}}&&
      \ar @{} [drr] |{\square}
     \hat{Y}^{**} \ar@{^{(}->}[rr]\ar[d]^{\beta} &&
    \hat{Y}^* \ar[d]^{\beta} && 
   \\
   E[2]^* \ar[rr]^{h} &&
   \projpppp \ar@{^{(}->}[rr]^{i} &&
    \projppp 
  }
\end{equation}
The construction in (\ref{lm-construction}) was
already used in \cite{Moeller} in order to construct origamis from dessins d'enfants.
With the help of this it was deduced that the absolute Galois group Gal$(\Qquer/\QQ)$
acts transitively on origami curves from the fact that it does it on dessins d'enfants.
In \cite[Theorem 3]{FlorianDiss} the monodromies of the coverings are explicitly calculated.
In particular, one directly sees from the monodromy
of $\tilde{q}$ in the proof of Theorem 3 on p.52 ($\tilde{q}$ is called $\pi$ there, 
its monodromy is $m_{\pi}$,
the pair $(\sigma_x,\sigma_y)$
describes the monodromy of $\beta$, $d$ is the degree of $\beta$) 
that the monodromy of $\tilde{q}$ acts transitively
on $\{1, \ldots, d\}$, since $\hat{Y}^*$ is connected. Thus also $Y^*$ is connected.\\
We now obtain the fundamental group of $Y^*$ as follows: Following the notations
in \cite[Lemma 3.2 and Prop. 3.1]{em} denote 
\[\Delta = \Gamma = \pi_1(\projppp).\]  
The last equality means
that we have identified these two groups by a fixed isomorphism.
From F4) and F3) we see:
\[ 
\begin{array}{lcl} 
  \pi_1(\hat{Y}^{**}) &=& (i)_{*}^{-1}(\pi_1(\hat{Y}^{*})) =: G(\Delta) \mbox{ and }\\
  \pi_1(Y^*)         &=&  G(\Delta) \cap \pi_1(E[2]^*) =: G(\Delta)_0.
\end{array}
\]
Lemma~3.2  and the discussion in Section 2 in \cite{em} shows that
any element $\overline{A} \in \PGamma(2)$ defines  a homeomorphism
on $\projpppp$ via the pushing map which lifts to $\hat{Y}^{**}$ and acts on the 
punctures in the fibre of
$\lambda$ in the desired way. Furthermore, each homeomorphism on $\projpppp$
has two lifts on $E[2]^{*}$ which fix the four punctures of $E[2]^{*}$ pointwise. 
We can choose affine representatives of those in their homotopy class. They
will then have derivative $A$ and $-A$ (for an appropriate choice of the isomorphism
$\HH/\Gammazwei \cong \projppp$ in the beginning), see \cite[Section 2]{em}. 
It follows e.g. from the proof of Proposition~3.1 in  \cite{em} that they lift to 
$Y^*$. The lifts then act on the fibre of $S = (1,1)$ in the desired way.\\
The fact  we have just seen, namely that in the construction in (\ref{lm-construction}) the group
$\Gammazwei$ is always a subgroup
of the Veech group $\Gamma(Y^*)$
was already in \cite{Moeller} a crucial 
part in the proof for the faithfulness of the action of Gal$(\Qquer/\QQ)$
on Teichm\"uller curves.
How one explicitly obtains the Veech group in terms of the monodromy $\beta$
is given in \cite[Theorem 4]{FlorianDiss}.\\
Filling in the punctures into $Y^*$ and $E[2]^*$ and extending
$\tilde{q}$ to a ramified covering of the closed Riemann surfaces finally gives
the desired map $\tilde{q}: Y \to E[2]$.\\ 

In our special case, where $\Gamma = \pm\G{2k}$,
we have the intermediate cover 
\[\HH/\PG{2k} \stackrel{\beta_1}{\longrightarrow} \HH/\PGammasechs 
                \stackrel{\beta_2}{\longrightarrow} \HH/\PGammazwei.\]
Recall that we have studied the coverings $\beta_1$ and $\beta_2$ in  
Section~\ref{pg}. We can now determine the fibre product in two steps. First, $X^*$ is
the fibre product of the degree 12 cover 
$\beta_2: \HH/\PGammasechs \to \HH/\PGammazwei = \projppp$ 
and the map $i\circ h: E[2]^* \to \projppp$  having the two projections
$X^* \to \HH/\PGammasechs$ and $p: X^* \to E[2]^*$. Secondly, 
$Y^*$ is the fibre product of $X^* \to \HH/\PGammasechs$ and the 
cyclic covering $\beta_1: \HH/\PG{2k} \to \HH/\PGammasechs$ having 
the projections $Y^* \to \HH/\PG{2k}$ and $q: Y^* \to X^*$. The 
desired covering $\tilde{q}$ is then $\tilde{q} = p\circ q$. We explicitly
give the translation surface $X^*$ in Section~\ref{sectionX} and $Y^*$ 
in Section~\ref{step2} and show that they have the properties we need.
In Section~\ref{step3} we make explicit which criterions coverings 
$r: Z^* \to Y^*$ must fulfill that we can use them in the second step
described on page~\pageref{secondmainstep}. 
We finally prove Theorem~\ref{t.A} in Section~\ref{proofoftheorem}.\\

For the coverings $p$ and $q$ we more precisely show:

\begin{itemize}
\item 
  The covering $p: X\to \Ezwei$ is normal, of degree 12 and its
  Galois group is $\PGammazwei/\PGammasechs  \cong 
  \pslzwei(\ZZ/3\ZZ)$ (see (\ref{quotient})). 
  The Veech group of the origami $X$ is $\slzwei(\ZZ)$. Its affine group  
  contains a subgroup $\Gaffzwei$ isomorphic to 
  $\Gammazwei$, see Proposition~\ref{propx} and Lemma~\ref{actionx}.\\  
  Furthermore, the cover $p$ is unramified over the point $\Cc$.
  We give an explicit identification of the 12 
  preimages $\Cc_1$, \ldots, $\Cc_{12}$ with the cosets in 
  $\PGammazwei/\PGammasechs$
  in such a way that the action of $\Gaffzwei$ on them is equal to the
  action of $\PGammazwei$ on the cosets  up to the 
  automorphism $\gamma$, see Lemma~\ref{actionx}.
\item
  The covering $q: Y \to X$ is normal and of degree $2k$ with 
  Galois group $\ZZ/(2k\ZZ)$. Similarly as 
  before we obtain a suitable identification of the fibre of 
  $p \circ q: Y \to \Ezwei$  over $\Cc$
  with $\PGammazwei/\PG{2k}$, see Section~\ref{step2}.
 \item
    $\PG{2k}$ is 
   the stabiliser in the Veech group $\Gamma(Y)$ 
   of a certain partition  of $\mathcal{S}$,
   the set of punctures on $Y$, see Section~\ref{step2}. 
   Coverings  $Z$ which ``ramify with respect to this partition'' have a 
   Veech group whose image in $\pslzwei(\ZZ)$ is contained in 
   $\PG{2k}$, see Section~\ref{step3}. 
\item
  In Section~\ref{proofoftheorem}  
   we give explicit examples of such origamis.
\end{itemize}

\subsection{The origami {\boldmath $X$}}
\label{sectionX}

We define the origami $X$ as follows: It consists of $4\times 12 = 48$ squares
labelled by $(A,h)$ with $A \in \PGammasechs\backslash\PGammazwei$ and 
$h \in \{1,2,3,4\}$. The horizontal gluing rules are:

\begin{equation}\label{gluexhorizontal}
\begin{array}{ll}
  (A,1) \mapsto (A,2),  & (A,2) \mapsto (A\cdot L^{-2}, 1),\\
  (A,3) \mapsto (A,4),  & (A,4) \mapsto (A\cdot L^2,3)
\end{array}
\end{equation}
The vertical gluings are:
\begin{equation}\label{gluexvertical}
\begin{array}{ll}
  (A,1) \mapsto (A\cdot L^2,3)  &  (A,2) \mapsto (A\cdot L^{-2}, 4)\\
  (A,3) \mapsto (A\cdot T^{-2},1) &  (A,4) \mapsto (A\cdot T^2, 2)
\end{array}
\end{equation}
You can easily read up from the Cayley graph of $\PGammazwei/\PGammasechs$ 
in Figure~\ref{cayleygraph} (recall that $\PGammasechs$ is normal in $\PGammazwei$) 
that Figure~\ref{xhorizontal} below shows the origami $X$.

\newcommand{\aaa}{(A_6,2)}\newcommand{\bbb}{(A_1,1)}\newcommand{\ccc}{(A_1,2)}
\newcommand{\ddd}{(A_5,1)}\newcommand{\eee}{(A_5,2)}\newcommand{\fff}{(A_6,1)}
\newcommand{\gggg}{(A_6,4)}\newcommand{\hhh}{(A_5,3)}\newcommand{\iii}{(A_5,4)} 
\newcommand{\jjj}{(A_1,3)}\newcommand{\kkkk}{(A_1,4)}\newcommand{\llll}{(A_6,3)}
\newcommand{\mmm}{(A_9,2)}\newcommand{\nnn}{(A_{11},1)}\newcommand{\ooo}{(A_{11},2)} 
\newcommand{\ppp}{(A_3,1)}\newcommand{\qqq}{(A_3,2)}\newcommand{\rrr}{(A_9,1)}
\newcommand{\sss}{(A_9,4)}\newcommand{\ttt}{(A_3,3)}\newcommand{\uuu}{(A_3,4)} 
\newcommand{\vvv}{(A_{11},3)}\newcommand{\www}{(A_{11},4)}\newcommand{\xxx}{(A_9,3)}
\newcommand{\bluea}{(A_2,2)} \newcommand{\blueb}{(A_{10},1)}\newcommand{\bluec}{(A_{10},2)}
\newcommand{\blued}{(A_8,1)}\newcommand{\bluef}{(A_8,2)}\newcommand{\blueg}{(A_2,1)}
\newcommand{\blueh}{(A_2,4)}\newcommand{\bluei}{(A_8,3)}\newcommand{\bluej}{(A_8,4)}
\newcommand{\bluek}{(A_{10},3)}\newcommand{\bluel}{(A_{10},4)}\newcommand{\bluem}{(A_2,3)}
\newcommand{\reda}{(A_4,2)}\newcommand{\redb}{(A_{12},1)}\newcommand{\redc}{(A_{12},2)}
\newcommand{\rede}{(A_7,1)}\newcommand{\redf}{(A_7,2)}\newcommand{\redg}{(A_4,1)}
\newcommand{\redh}{(A_4,4)}\newcommand{\redi}{(A_7,3)}\newcommand{\redj}{(A_7,4)}
\newcommand{\redk}{(A_{12},3)}\newcommand{\redl}{(A_{12},4)}\newcommand{\redm}{(A_4,3)}
\newcommand{\rh}{\sst (A_8,2)}\newcommand{\ri}{\sst(A_9,1)}\newcommand{\rj}{\sst(A_6,2)}
\newcommand{\rk}{\sst(A_{10},1)}\newcommand{\rl}{\sst(A_{11},2)}\newcommand{\rrm}{\sst(A_5,1)}
\newcommand{\ra}{\sst(A_5,4)}\newcommand{\rb}{\sst(A_{11},3)}\newcommand{\rc}{\sst(A_{10},4)}
\newcommand{\re}{\sst(A_6,3)}\newcommand{\rf}{\sst(A_9,4)}\newcommand{\rg}{\sst(A_8,3)}
\newcommand{\ba}{\sst(A_1,4)}\newcommand{\bb}{\sst(A_{12},3)}\newcommand{\bc}{\sst(A_{11}4,)}
\newcommand{\bd}{\sst(A_5,3)}\newcommand{\bbf}{\sst(A_4,4)}\newcommand{\bg}{\sst(A_3,3)}
\newcommand{\bh}{\sst(A_3,2)}\newcommand{\bi}{\sst(A_4,1)}\newcommand{\bj}{\sst(A_5,2)}
\newcommand{\bk}{\sst(A_{11},1)}\newcommand{\bl}{\sst(A_{12},2)}\newcommand{\bm}{\sst(A_1,1)}
\newcommand{\va}{\sst(A_6,4)}\newcommand{\vb}{\sst(A_{10},3)}\newcommand{\vc}{\sst(A_{12},4)}
\newcommand{\ve}{\sst(A_1,3)}\newcommand{\vf}{\sst(A_2,4)}\newcommand{\vg}{\sst(A_7,3)}
\newcommand{\vh}{\sst(A_7,2)}\newcommand{\vi}{\sst(A_2,1)}\newcommand{\vj}{\sst(A_1,2)}
\newcommand{\vk}{\sst(A_{12},1)}\newcommand{\vl}{\sst(A_{10},2)}\newcommand{\vm}{\sst(A_6,1)}
\newcommand{\ya}{\sst(A_7,4)}\newcommand{\yb}{\sst(A_2,3)}\newcommand{\yc}{\sst(A_3,4)}
\newcommand{\ye}{\sst(A_4,3)}\newcommand{\yf}{\sst(A_8,4)}\newcommand{\yg}{\sst(A_9,3)}
\newcommand{\yh}{\sst(A_9,2)}\newcommand{\yi}{\sst(A_8,1)}\newcommand{\yj}{\sst(A_4,2)}
\newcommand{\yk}{\sst(A_3,1)}\newcommand{\yl}{\sst(A_2,2)}\newcommand{\ym}{\sst(A_7,1)}

\newcommand{\Aone}{\A_1}\newcommand{\Atwo}{\A_2}\newcommand{\Athree}{\A_3}\newcommand{\Afour}{\A_4}
\newcommand{\Bone}{\B_1}\newcommand{\Btwo}{\B_2}\newcommand{\Bthree}{\B_3}\newcommand{\Bfour}{\B_4}
\newcommand{\Cone}{\Cc_1}\newcommand{\Ctwo}{\Cc_2}\newcommand{\Cthree}{\Cc_3}\newcommand{\Cfour}{\Cc_4}
\newcommand{\Done}{\D_1}\newcommand{\Dtwo}{\D_2}\newcommand{\Dthree}{\D_3}\newcommand{\Dfour}{\D_4}
\newcommand{\Cfive}{\Cc_5}\newcommand{\Csix}{\Cc_6}\newcommand{\Cseven}{\Cc_7}\newcommand{\Ceight}{\Cc_8}
\newcommand{\Cnine}{\Cc_9}\newcommand{\Cten}{\Cc_{10}}\newcommand{\Celeven}{\Cc_{11}}\newcommand{\Ctwelve}{\Cc_{12}}

\begin{figure}[htb!]
\scalebox{.65}{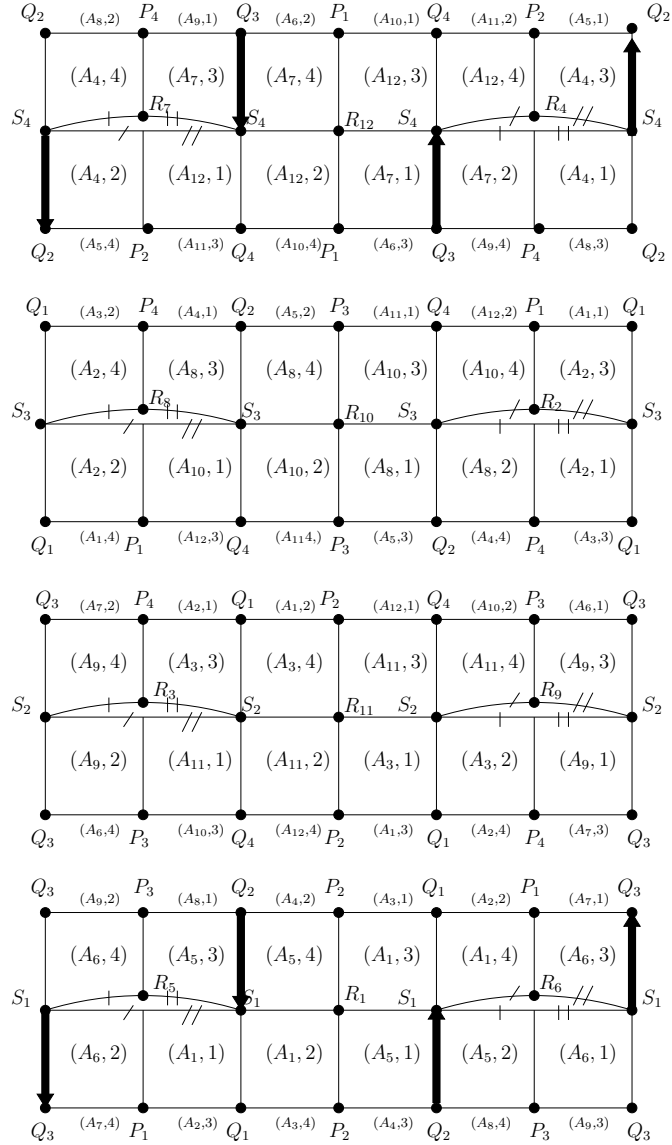}
\caption{
  The horizontal cylinder-decomposition of the origami $X$. An edge labelled
  by a pair $(A_i,h)$ is glued to the suitable edge of Square $(A_i,h)$. 
  Slits with same labels within a double cylinder
  are glued. Opposite vertical unlabelled edges are glued. The vertical black arrows
  indicate slits for a later  construction in Section~\ref{step2}.
  \label{xhorizontal}
}
\end{figure}

Figure~\ref{xvertical}
below shows the same origami with its vertical cylinder decomposition.
\newcommand{\vfa}{(A_{12},3)}
\newcommand{\vfb}{(A_4,1)}
\newcommand{\vfc}{(A_5,3)}
\newcommand{\vfd}{(A_1,1)}
\newcommand{\vfe}{(A_3,3)}
\newcommand{\vff}{(A_{11},1)}
\newcommand{\vfg}{(A_{11},4)}
\newcommand{\vfh}{(A_3,2)}
\newcommand{\vfi}{(A_1,4)}
\newcommand{\vfj}{(A_5,2)}
\newcommand{\vfk}{(A_4,4)}
\newcommand{\vfl}{(A_{12},2)}
\newcommand{\vsa}{(A_{11},3)}
\newcommand{\vsb}{(A_9,1)}
\newcommand{\vsc}{(A_6,3)}
\newcommand{\vsd}{(A_5,1)}
\newcommand{\vse}{(A_8,3)}
\newcommand{\vsf}{(A_{10},1)}
\newcommand{\vsg}{(A_{10},4)}
\newcommand{\vsh}{(A_8,2)}
\newcommand{\vsi}{(A_5,4)}
\newcommand{\vsj}{(A_6,2)}
\newcommand{\vsk}{(A_9,4)}
\newcommand{\vsl}{(A_{11},2)}
\newcommand{\vta}{(A_{10},3)}
\newcommand{\vtb}{(A_2,1)}
\newcommand{\vtc}{(A_1,3)}
\newcommand{\vtd}{(A_6,1)}
\newcommand{\vte}{(A_7,3)}
\newcommand{\vtf}{(A_{12},1)}
\newcommand{\vtg}{(A_{12},4)}
\newcommand{\vth}{(A_7,2)}
\newcommand{\vti}{(A_6,4)}
\newcommand{\vtj}{(A_1,2)}
\newcommand{\vtk}{(A_2,4)}
\newcommand{\vtl}{(A_{10},2)}
\newcommand{\vva}{(A_4,3)}
\newcommand{\vvb}{(A_7,1)}
\newcommand{\vvc}{(A_9,3)}
\newcommand{\vvd}{(A_3,1)}
\newcommand{\vve}{(A_2,3)}
\newcommand{\vvf}{(A_8,1)}
\newcommand{\vvg}{(A_8,4)}
\newcommand{\vvh}{(A_2,2)}
\newcommand{\vvi}{(A_3,4)}
\newcommand{\vvj}{(A_9,2)}
\newcommand{\vvk}{(A_7,4)}
\newcommand{\vvl}{(A_4,2)}

\begin{figure}[hb!]
  \scalebox{.55}{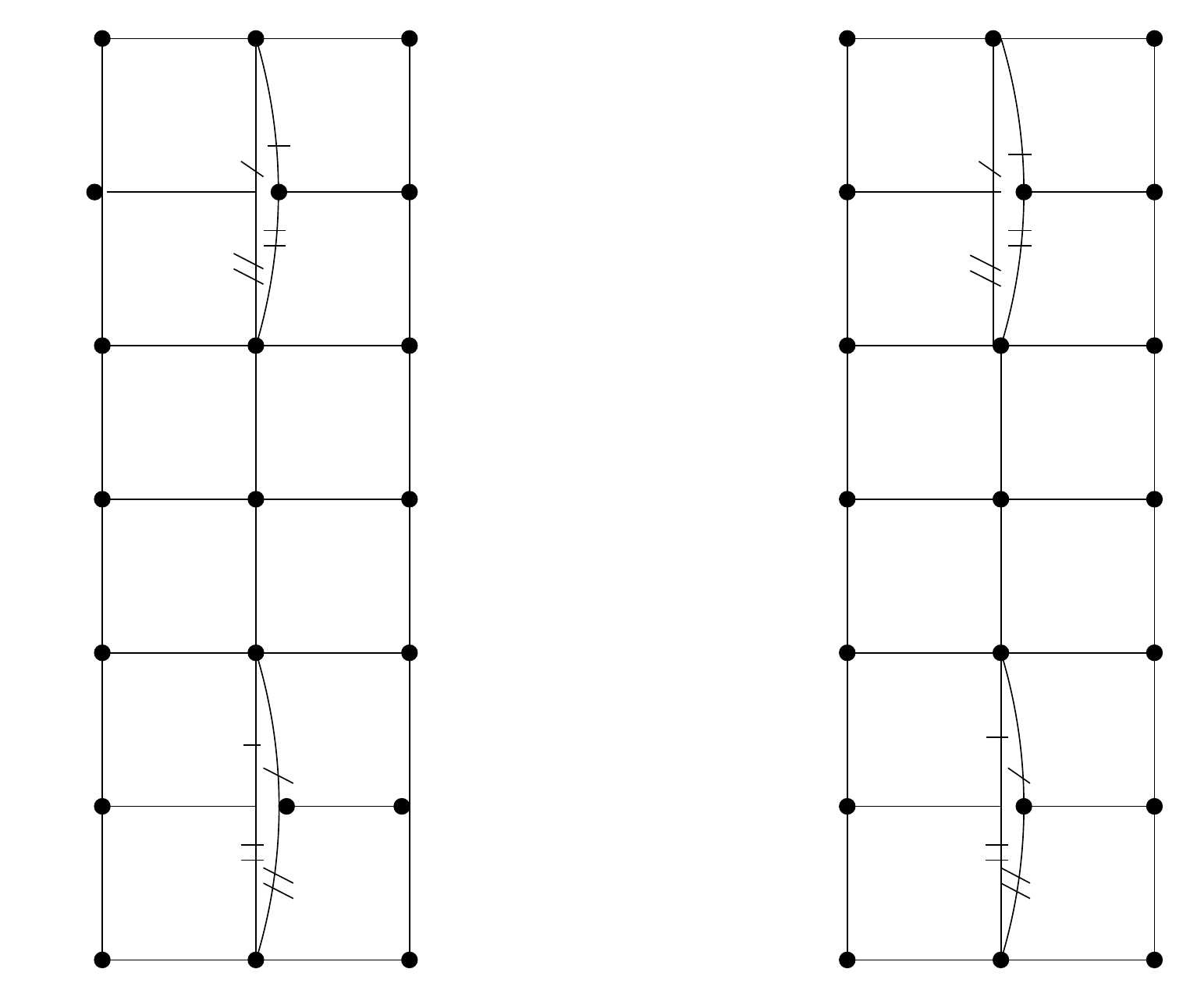}\\[5mm]   
  \scalebox{.55}{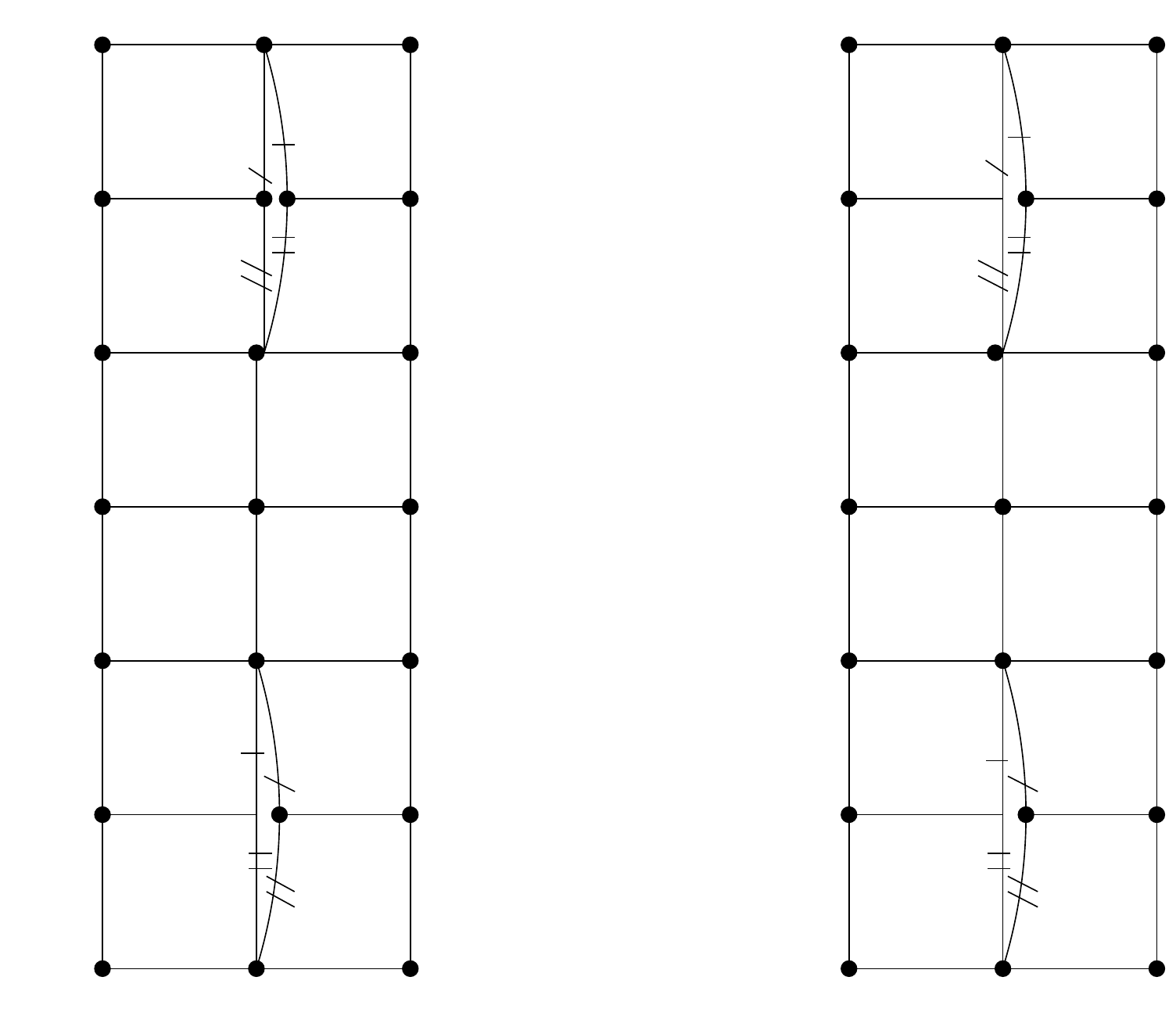}\\[5mm]
\caption{
  The vertical cylinder-decomposition of the origami $X$. An edge labelled
  by a pair $(A_i,h)$ is glued to the suitable edge of Square $(A_i,h)$. 
  Slits with same labels within a double cylinder
  are glued. Opposite horizontal  unlabelled edges are glued.
  \label{xvertical}
}
\end{figure}

Observe from Figure~\ref{xhorizontal} that $X$ decomposes into eight 
horizontal cylinders 
which we have arranged in the figure into four {\em double cylinders}.
The vertices glue to the points:
$\A_1$, \ldots, $\A_4$, $\B_1$, \ldots, $\B_4$, $\D_1$, \ldots, $\D_4$ each of cone 
angle $6\pi$
and $\Cc_1$, \ldots, $\Cc_{12}$ having cone angle $2\pi$, respectively.  
In particular
we have twelve singularities of cone angle $6\pi$. Thus the origami
is in the stratum $\mathcal{H}(2^{[12]})$ and its genus is $13$.
The map 
$p: X \to \Ezwei$ which maps the square $(A,i)$ to the square $i$ on $\Ezwei$
defines a well-defined covering of degree 12.
Furthermore, the development vectors of the saddle connections
on $X$ span the lattice $\ZZ^2$. Hence the Veech group is contained in 
$\slzwei(\ZZ)$ and affine homeomorphisms preserve the vertices
of the squares that form $X$.
Finally, if  $B \in \PGammazwei/\PGammasechs$ is
the coset of $\hat{B}$, then we can directly see from the gluing rules 
in (\ref{gluexhorizontal}) and (\ref{gluexvertical})
and from the fact that $\PGammasechs$ is normal in $\PGammazwei$   that the 
map that sends the square $(A,i)$ to $(\hat{B}\cdot A,i)$ 
is well-defined, depends only on $B$  and respects the gluing rules.
Thus it induces
a well-defined translation $\tau_{B}$ of $X$. Since a translation has to 
preserve cone angles and the vertices of the squares, 
it permutes the $\Cc_i$'s. Furthermore
a translation is determined by the image of one of the $\Cc_i$'s,
since they are regular points. Therefore $X$ has 
not more than these 12  translations.
We summarise the properties of the origami $X$ in the following proposition.

\begin{proposition}\label{propx}
The origami $X$ defined above has the following properties:
\begin{enumerate}
\item[i)]
  Its genus is $13$ and it has twelve singularities 
  $\A_i$, $\B_i$, $\D_i$ ($i \in \{1, \ldots, 4\}$) of cone angle $6\pi$. Its group
  of translations $\emtrans(X)$  is 
  \[
  \emtrans(X) = \{\tau_B| B \in \PGammazwei/\PGammasechs\}
  \cong \PGammazwei/\PGammasechs
  \]
  (with $\tau_B$ defined as above), and $X/\emtrans(X) \cong \Ezwei$. In particular the covering 
  $p: X \to \Ezwei$ is normal
  with $\PGammazwei/\PGammasechs$ as group of deck transformations.
\item[ii)]
  The Veech group of $X$ is $\emslzwei(\ZZ)$.
\end{enumerate}
\end{proposition}

\begin{proof}
{\bf i)} was shown in the paragraph before Proposition~\ref{propx}
and follows from the definition of $X$. To see {\bf ii)} we
define the affine homeomorphisms $f_T$ and $f_L$ as follows. $f_T$ has 
derivative $T$ and acts on the edges on the boundaries of the four 
horizontal double 
cylinders as a shift of 3. Hence e.g. the lower edge of the square $(A_1,1)$
is mapped to the lower edge of the square $(A_5,2)$.
The upper edge of $(A_1,1)$ is mapped to the upper edge of $(A_6,1)$,
sine the derivative is $T$.
I.e. the lower edge of $(A_6,3)$ is mapped to the lower edge of $(A_5,3)$.
The upper edge of $(A_6,3)$ is then mapped to the upper edge of $(A_5, 4)$.
You can read off from Figure~\ref{xhorizontal} that this 
is consistent with the gluings and thus gives a well-defined affine 
homeomorphism. 
On the middle lines of the double cylinders, it acts
on the upper edges of the lower cylinder as shift by 2 to the left;
on the lower edges of the upper cylinder it acts as shift by 2
to the right.
Similarly define $f_L$ to have derivative $L$ and
act on the edges of the boundaries of the four vertical double cylinders 
as shift of 3 (see Figure~\ref{xvertical}). The matrices $T$ and $L$ generate 
$\slzwei(\ZZ)$, thus the Veech group equals $\slzwei(\ZZ)$.
\end{proof}

In order to obtain the desired origami $Z$ in the end, we need 
how $\Gammazwei$ or more precisely  
\[
  \Gaffzwei = <f_{T^2}, f_{L^2}, f_{-I}> 
\]
acts on the set of vertices of $X$. Here we have by definition 
$f_{T^2}=f_T^2$, $f_{L^2}=f_L^2$ 
and $f_T, f_L, f_{-I}$ are defined as in the proof of Proposition~\ref{propx} ii). 
Observe that the image of $\Gaffzwei$ 
in $\slzwei(\ZZ)$ by the derivative map
is $\Gammazwei$. We study this action in the two following lemmas.
\begin{lemma}\label{actiononvertices}
The two affine homeomorphisms $f_T$ and $f_L$ 
act on the set $\mathcal{S}$ of vertices of $X$ as the following permutations:\\[1mm]
\[
\begin{array}{ll}
  f_T:& 
   (\Cc_1\; \Cc_6\; \Cc_5)(\Cc_{11}\; \Cc_9\; \Cc_3)(\Cc_{10}\; \Cc_2\; \Cc_8)(\Cc_{12}\; \Cc_{4}\; \Cc_{7})\\
      &\quad\quad \circ \; (\A_1\; \B_2)(\A_2\; \B_3)(\A_3\; \B_1)(\A_4\; \B_4).\\[1mm]
  f_L:&
   (\Cc_1\; \Cc_{11}\; \Cc_4)(\Cc_5\; \Cc_{10}\; \Cc_9)(\Cc_{6}\; \Cc_{12}\; \Cc_2)(\Cc_3\; \Cc_{8}\; \Cc_7)\\
      &\quad\quad \circ \; (\B_1 \; \D_4)(\B_2\; \D_2)(\B_3\; \D_3)(\B_4\; \D_1)
\end{array}
\]
${f_T}^2$ and ${f_L}^2$ fix the boundary of 
the horizontal, respectively
of the vertical double cylinders pointwise.
Furthermore $f_{-I} = (f_T \circ {f_L}^{-1})^3$ acts trivially 
on $\mathcal{S}$ and has derivative 
  $-I$.
\end{lemma}

\begin{proof}
This can be directly read off 
from Figure~\ref{xhorizontal} and Figure~\ref{xvertical}.
The short calculation $(T\cdot L^{-1})^3 = -I$ shows the claim about the 
derivative of $f_{-I}$.
\end{proof}

\begin{lemma}\label{actionx}
  The action of $\Gaffzwei$ on the set of vertices of $X$
  has the following properties:
  \begin{enumerate}
  \item[i)]
    $\Gaffzwei$ is the full pointwise stabiliser in the affine group $\emaff(X)$ 
    of the set 
    \[\mathcal{S}_{\mbox{\em \fs sing}} = \{\A_i, \B_i, \D_i|\, i \in \{1,\ldots, 4\}\}.\]
  \item[ii)]
    If we denote the vertex $\Cc_i$ also by $\Cc_{A_i}$ with 
    $A_i \in \empslzwei(\ZZ/6\ZZ)$ defined as in Remark~\ref{matrices},
    then $\Gaffzwei$ acts on the set $\{\Cc_i|\, i \in \{1,\ldots, 12\}\}$ by:
    $$f_1 =  f_{T^2}: \Cc_{A_i} \mapsto \Cc_{A_{i}\cdot y^{-1}},\quad 
    f_2 = f_{L^2}:  \Cc_{A_i} \mapsto \Cc_{A_{i}\cdot y^{-1}x} $$
        and
    $$f_3= f_{-I}: \Cc_{A_i} \mapsto \Cc_{A_i}.$$
	Here, the  
        images of $x$ and $y$ in  $\empslzwei(\ZZ/6\ZZ)$ are also denoted $x$ and $y$ ( by a slight abuse of notation).
        We may identify $R_{A_i}$ with the coset ${A_i}^{-1}\cdot\PGammasechs$ and then obtain an action of
        $\PGammazwei$ from the left.
      \item[iii)]
	The derivative map gives an isomorphism from $\Gaffzwei$ to
	$\Gammazwei$.
  \end{enumerate}
\end{lemma}

\begin{proof}
It directly follows from Lemma~\ref{actiononvertices} that
the generators $f_1$, $f_2$ and $f_3$ of 
$\Gaffzwei$ act trivially 
on $\mathcal{S}_{\mbox{\fs sing}}$. We have to further show that $\Gaffzwei$ is the full pointwise
stabiliser. Suppose that an affine homeomorphism $f \in \aff(X)$
acts trivially on $\mathcal{S}_{\mbox{\fs sing}}$. 
Let $\bar{f}$ be its descend to $E[2]$ via $p$.
Since $f$ fixes the points in $\mathcal{S}_{\mbox{\fs sing}}$
pointwise, $\bar{f}$ has to fix their images ${0\choose 0} + (2\ZZ)^2$,
${1\choose 0} + (2\ZZ)^2$, ${1\choose 1} + (2\ZZ)^2$ as well.
Hence the derivative $D(f)$ of $f$ is in the main congruence group
$\Gammazwei$. It follows that there
is some $f_0$ in $\Gaffzwei$ with the same derivative. Thus $f = \tau\circ f_0$
for some translation $\tau$. Since $f_0$ also preserves $\mathcal{S}_{\mbox{\fs sing}}$
pointwise, $\tau$ has to do it as well. The lower edge of the square
$(A_1, 1)$ is a saddle connection with developing
vector ${1\choose 0}$ which starts in $\A_1$ and ends in $\B_1$. It must
be mapped by $\tau$ to a saddle connection with the same vector
and the same starting and end point. But this edge is the only 
segment with this property. Hence it is fixed, $\tau$ is the identity
and $f$ is in $\Gaffzwei$. This shows {\bf i)}. 
The last argument in particular shows that $\Gaffzwei$ contains
only the trivial translation. Since the kernel of the derivative map $D$
consist  precisely of the translations, we obtain 
an isomorphism $D:\Gaffzwei \to \Gammazwei$.
Hence {\bf iii)} holds.
The statement in {\bf ii)} can be 
directly read off from the Cayley graph in Figure~\ref{cayleygraph} and
Lemma~\ref{actiononvertices}.
\end{proof}

\subsection{The origami \boldmath{$Y$}}
\label{step2}

We now achieve the second step, that is, to define the origami $Y$ and show
that it has the desired properties.

\begin{definition}\label{defy}
  Define the origami $Y$ as follows
  \begin{itemize}
  \item 
    Take $2k$ copies of the origami $X$. Label their squares by
    the elements $(A, h, j)$ where $(A,h)$ is the label it had in $X$ and
    $j \in \{1, \ldots, 2k\}$ is the number of the copy. 
  \item
    Slit them along the eight 
    vertical edges which are highlighted by black arrows in
    Figure~\ref{xhorizontal}.
  \item
    Reglue the slits as follows: If square $(A,h,j)$ is the square on the
    left of a slit marked by an arrow pointing upwards, 
    then glue it to $(r(A,h), j+1)$. If it is the
    square on the left of a slit marked by an arrow pointing downside, 
    then glue it to $(r(A,h),j-1)$.
    Here we denote by $r(A,h)$ the right neighbour of $(A,h)$
    in the original origami $X$.
  \end{itemize}
\end{definition}

One easily checks that $Y$ is the origami shown in Figure~\ref{yhorizontal} and 
in Figure~\ref{yvertical}.

\newcommand{\ips}{j+1)}
\renewcommand{\aaa}{(A_6,2,j)}\renewcommand{\bbb}{(A_1,1,j)}
\renewcommand{\ccc}{(A_1,2,j)}\renewcommand{\ddd}{(A_5,1,j)}
\renewcommand{\eee}{(A_5,2,}\renewcommand{\fff}{(A_6,1,}
\renewcommand{\gggg}{(A_6,4,}\renewcommand{\hhh}{(A_5,3,}
\renewcommand{\iii}{(A_5,4,j)}\renewcommand{\jjj}{(A_1,3,j)}
\renewcommand{\kkkk}{(A_1,4,j)}\renewcommand{\llll}{(A_6,3,j)}

\renewcommand{\mmm}{(A_9,2,j)}\renewcommand{\nnn}{(A_{11},1,j)}
\renewcommand{\ooo}{(A_{11},2,j)}\renewcommand{\ppp}{(A_3,1,j)}
\renewcommand{\qqq}{(A_3,2,j)}\renewcommand{\rrr}{(A_9,1,j)}
\renewcommand{\sss}{(A_9,4,j)}\renewcommand{\ttt}{(A_3,3,j)}
\renewcommand{\uuu}{(A_3,4,j)}\renewcommand{\vvv}{(A_{11},3,j)}
\renewcommand{\www}{(A_{11},4,j)}\renewcommand{\xxx}{(A_9,3,j)}

\renewcommand{\bluea}{(A_2,2,j)}\renewcommand{\blueb}{(A_{10},1,j)}
\renewcommand{\bluec}{(A_{10},2,j)}\renewcommand{\blued}{(A_8,1,j)}
\renewcommand{\bluef}{(A_8,2,j)}\renewcommand{\blueg}{(A_2,1,j)}
\renewcommand{\blueh}{(A_2,4,j)}\renewcommand{\bluei}{(A_8,3,j)}
\renewcommand{\bluej}{(A_8,4,j)}\renewcommand{\bluek}{(A_{10},3,j)}
\renewcommand{\bluel}{(A_{10},4,j)}\renewcommand{\bluem}{(A_2,3,j)}

\renewcommand{\reda}{(A_4,2,j)}\renewcommand{\redb}{(A_{12},1,j)}
\renewcommand{\redc}{(A_{12},2,j)}\renewcommand{\rede}{(A_7,1,j)}
\renewcommand{\redf}{(A_7,2,}\renewcommand{\redg}{(A_4,1,}
\renewcommand{\redh}{(A_4,4,}\renewcommand{\redi}{(A_7,3,}
\renewcommand{\redj}{(A_7,4,j)}\renewcommand{\redk}{(A_{12},3,j)}
\renewcommand{\redl}{(A_{12},4,j)}\renewcommand{\redm}{(A_4,3,j)}

\renewcommand{\rh}{\sst (A_8,2,j+1)}\renewcommand{\ri}{\sst(A_9,1,j+1)}
\renewcommand{\rj}{\sst(A_6,2,j)}\renewcommand{\rk}{\sst(A_{10},1,j)}
\renewcommand{\rl}{\sst(A_{11},2,j)}\renewcommand{\rrm}{\sst(A_5,1,j)}
\renewcommand{\ra}{\sst(A_5,4,j)}\renewcommand{\rb}{\sst(A_{11},3,j)}
\renewcommand{\rc}{\sst(A_{10},4,j)}\renewcommand{\re}{\sst(A_6,3,j)}
\renewcommand{\rf}{\sst(A_9,4,j+1)}\renewcommand{\rg}{\sst(A_8,3,j+1)}

\renewcommand{\ba}{\sst(A_1,4,j)}\renewcommand{\bb}{\sst(A_{12},3,j)}
\renewcommand{\bc}{\sst(A_{11},4,j)}\renewcommand{\bd}{\sst(A_5,3,j)}
\renewcommand{\bbf}{\sst(A_4,4,j)}\renewcommand{\bg}{\sst(A_3,3,j)}
\renewcommand{\bh}{\sst(A_3,2,j)}\renewcommand{\bi}{\sst(A_4,1,j)}
\renewcommand{\bj}{\sst(A_5,2,j)}\renewcommand{\bk}{\sst(A_{11},1,j)}
\renewcommand{\bl}{\sst(A_{12},2,j)}\renewcommand{\bm}{\sst(A_1,1,j)}

\renewcommand{\va}{\sst(A_6,4,j)}\renewcommand{\vb}{\sst(A_{10},3,j)}
\renewcommand{\vc}{\sst(A_{12},4,j)}\renewcommand{\ve}{\sst(A_1,3,j)}
\renewcommand{\vf}{\sst(A_2,4,j)}\renewcommand{\vg}{\sst(A_7,3,j)}
\renewcommand{\vh}{\sst(A_7,2,j)}\renewcommand{\vi}{\sst(A_2,1,j)}
\renewcommand{\vj}{\sst(A_1,2,j)}\renewcommand{\vk}{\sst(A_{12},1,j)}
\renewcommand{\vl}{\sst(A_{10},2,j)}\renewcommand{\vm}{\sst(A_6,1,j)}

\renewcommand{\ya}{\sst(A_7,4,j)}\renewcommand{\yb}{\sst(A_2,3,j)}
\renewcommand{\yc}{\sst(A_3,4,j)}\renewcommand{\ye}{\sst(A_4,3,j)}
\renewcommand{\yf}{\sst(A_8,4,j+1)}\renewcommand{\yg}{\sst(A_9,3,j+1)}
\renewcommand{\yh}{\sst(A_9,2,j+1)}\renewcommand{\yi}{\sst(A_8,1,j+1)}
\renewcommand{\yj}{\sst(A_4,2,j)}\renewcommand{\yk}{\sst(A_3,1,j)}
\renewcommand{\yl}{\sst(A_2,2,j)}\renewcommand{\ym}{\sst(A_7,1,j)}

\renewcommand{\Aone}{\A_1^j}\renewcommand{\Atwo}{\A_2^j}
\renewcommand{\Athree}{\A_3^j}\renewcommand{\Afour}{\A_4^j}
\renewcommand{\Bone}{\B_1^j}\renewcommand{\Btwo}{\B_2^j}
\renewcommand{\Bthree}{\B_3^j}\renewcommand{\Bfour}{\B_4^j}
\renewcommand{\Cone}{\Cc_1^j}\renewcommand{\Ctwo}{\Cc_2^j}
\renewcommand{\Cthree}{\Cc_3^j}\renewcommand{\Cfour}{\Cc_4^j}
\renewcommand{\Done}{\D_1^j}\renewcommand{\Dtwo}{\D_2^j}
\renewcommand{\Dthree}{\D_3^j}\renewcommand{\Dfour}{\D_4^j}
\renewcommand{\Cfive}{\Cc_5^{j+1}}\renewcommand{\Csix}{\Cc_6^j}
\renewcommand{\Cseven}{\Cc_7^{j+1}}\renewcommand{\Ceight}{\Cc_8^j}
\renewcommand{\Cnine}{\Cc_9^j}\renewcommand{\Cten}{\Cc_{10}^j}
\renewcommand{\Celeven}{\Cc_{11}^j}\renewcommand{\Ctwelve}{\Cc_{12}^j}

\newcommand{\Bonex}{\B_1^{j-1}}\newcommand{\Afourx}{\A_4^{j-1}}
\newcommand{\Btwox}{\B_2^{j-1}}\newcommand{\Atwox}{\A_2^{j-1}}
\newcommand{\Btwop}{\B_2^{j+1}}\newcommand{\Athreep}{\A_3^{j+1}}
\newcommand{\Bthreep}{\B_3^{j+1}}\newcommand{\Afourp}{\A_4^{j+1}}
\newcommand{\Donex}{\D_1^{j-1}}

\begin{figure}
\scalebox{.8}{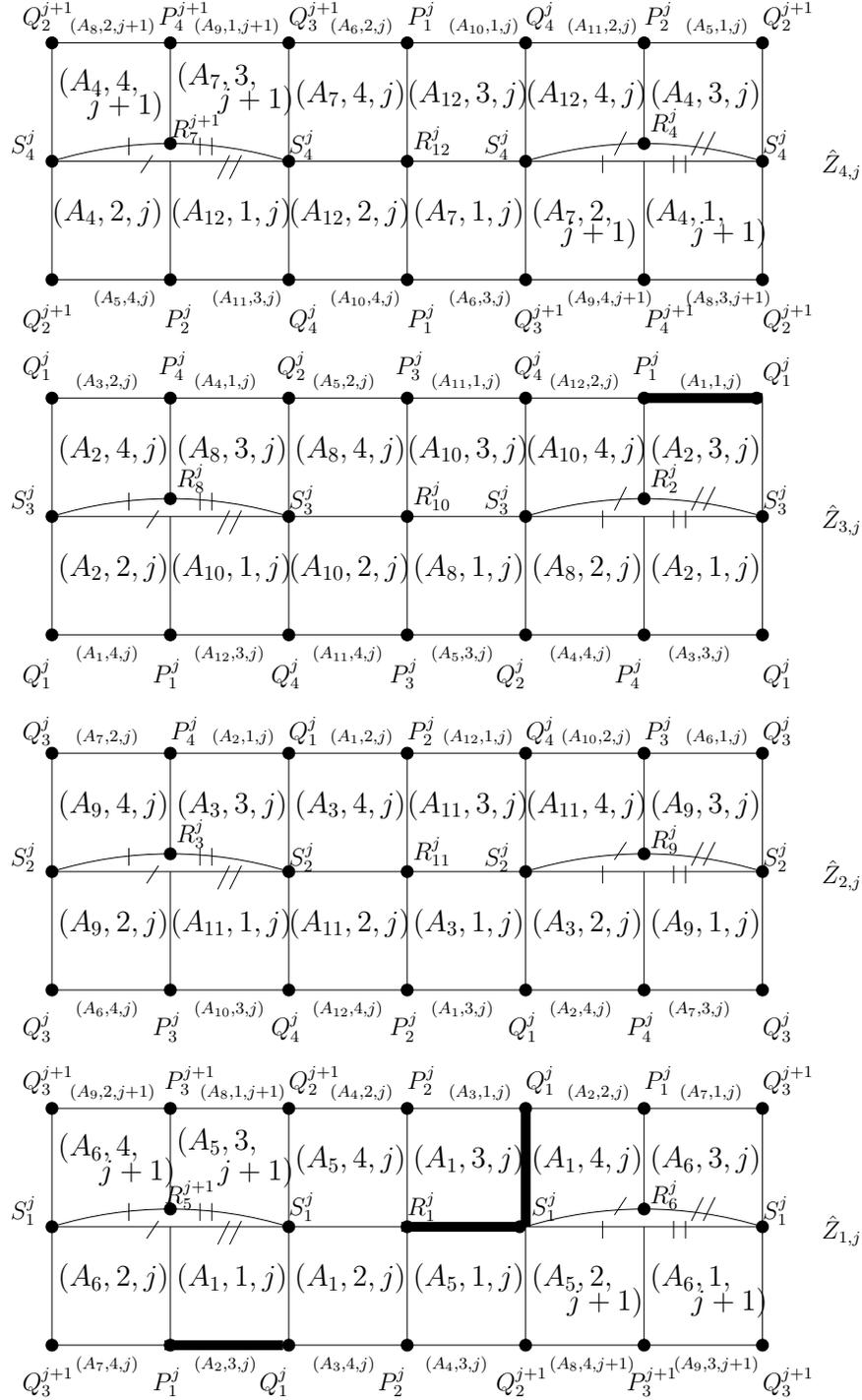}
\caption{
  The $i$-th part of the decomposition into horizontal
  double cylinders of the origami $Y$.
  An edge labelled by 
  $(A_i,h,j)$ is glued to the suitable edge of Square $(A_i,h,j)$. 
  Slits with same labels within a double cylinder
  are glued. Opposite vertical unlabelled edges are glued. 
  The black bars
  indicate selected edges for a later  construction in 
  Section~\ref{proofoftheorem}.
  \label{yhorizontal} \label{ydch} \label{zexample}
}
\end{figure}

\newcommand{\jjp}{j)}
\newcommand{\jplus}{j+1}

\renewcommand{\aaa}{(A_6,2,}\renewcommand{\bbb}{(A_1,1,}
\renewcommand{\ccc}{(A_1,2,}\renewcommand{\ddd}{(A_5,1,}
\renewcommand{\eee}{(A_5,2,}\renewcommand{\fff}{(A_6,1,}
\renewcommand{\gggg}{(A_6,4,}\renewcommand{\hhh}{(A_5,3,}
\renewcommand{\iii}{(A_5,4,}\renewcommand{\jjj}{(A_1,3,}
\renewcommand{\kkkk}{(A_1,4,}\renewcommand{\llll}{(A_6,3,}

\renewcommand{\mmm}{(A_9,2,}\renewcommand{\nnn}{(A_{11},1,}
\renewcommand{\ooo}{(A_{11},2,}\renewcommand{\ppp}{(A_3,1,}
\renewcommand{\qqq}{(A_3,2,}\renewcommand{\rrr}{(A_9,1,}
\renewcommand{\sss}{(A_9,4,}\renewcommand{\ttt}{(A_3,3,}
\renewcommand{\uuu}{(A_3,4,}\renewcommand{\vvv}{(A_{11},3,}
\renewcommand{\www}{(A_{11},4,}\renewcommand{\xxx}{(A_9,3,}

\renewcommand{\bluea}{(A_2,2,}\renewcommand{\blueb}{(A_{10},1,}
\renewcommand{\bluec}{(A_{10},2,}\renewcommand{\blued}{(A_8,1,}
\renewcommand{\bluef}{(A_8,2,}\renewcommand{\blueg}{(A_2,1,}
\renewcommand{\blueh}{(A_2,4,}\renewcommand{\bluei}{(A_8,3,}
\renewcommand{\bluej}{(A_8,4,}\renewcommand{\bluek}{(A_{10},3,}
\renewcommand{\bluel}{(A_{10},4,}\renewcommand{\bluem}{(A_2,3,}

\renewcommand{\reda}{(A_4,2,}\renewcommand{\redb}{(A_{12},1,}
\renewcommand{\redc}{(A_{12},2,}\renewcommand{\rede}{(A_7,1,}
\renewcommand{\redf}{(A_7,2,}\renewcommand{\redg}{(A_4,1,}
\renewcommand{\redh}{(A_4,4,}\renewcommand{\redi}{(A_7,3,}
\renewcommand{\redj}{(A_7,4,}\renewcommand{\redk}{(A_{12},3,}
\renewcommand{\redl}{(A_{12},4,}\renewcommand{\redm}{(A_4,3,}

\renewcommand{\vfa}{(A_{12},3,j)}\renewcommand{\vfb}{(A_4,1,j+1)}
\renewcommand{\vfc}{(A_5,3,j+1)}\renewcommand{\vfd}{(A_1,1,j)}
\renewcommand{\vfe}{(A_3,3,j)}\renewcommand{\vff}{(A_{11},1,j)}
\renewcommand{\vfg}{(A_{11},4,j)}\renewcommand{\vfh}{(A_3,2,j)}
\renewcommand{\vfi}{(A_1,4,j)}\renewcommand{\vfj}{(A_5,2,j+1)}
\renewcommand{\vfk}{(A_4,4,j+1)}\renewcommand{\vfl}{(A_{12},2,j)}

\renewcommand{\vsa}{(A_{11},3,j)}\renewcommand{\vsb}{(A_9,1,j)}
\renewcommand{\vsc}{(A_6,3,j-1)}\renewcommand{\vsd}{(A_5,1,j-1)}
\renewcommand{\vse}{(A_8,3,j)}\renewcommand{\vsf}{(A_{10},1,j)}
\renewcommand{\vsg}{(A_{10},4,j)}\renewcommand{\vsh}{(A_8,2,j)}
\renewcommand{\vsi}{(A_5,4,j-1)}\renewcommand{\vsj}{(A_6,2,j-1)}
\renewcommand{\vsk}{(A_9,4,j)}\renewcommand{\vsl}{(A_{11},2,j)}

\renewcommand{\vta}{(A_{10},3,j)}\renewcommand{\vtb}{(A_2,1,j)}
\renewcommand{\vtc}{(A_1,3,j)}\renewcommand{\vtd}{(A_6,1,j+1)}
\renewcommand{\vte}{(A_7,3,j+1)}\renewcommand{\vtf}{(A_{12},1,j)}
\renewcommand{\vtg}{(A_{12},4,j)}\renewcommand{\vth}{(A_7,2,j+1)}
\renewcommand{\vti}{(A_6,4,j+1)}\renewcommand{\vtj}{(A_1,2,j)}
\renewcommand{\vtk}{(A_2,4,j)}\renewcommand{\vtl}{(A_{10},2,j)}

\renewcommand{\vva}{(A_4,3,j-1)}\renewcommand{\vvb}{(A_7,1,j-1)}
\renewcommand{\vvc}{(A_9,3,j)}\renewcommand{\vvd}{(A_3,1,j)}
\renewcommand{\vve}{(A_2,3,j)}\renewcommand{\vvf}{(A_8,1,j)}
\renewcommand{\vvg}{(A_8,4,j)}\renewcommand{\vvh}{(A_2,2,j)}
\renewcommand{\vvi}{(A_3,4,j)}\renewcommand{\vvj}{(A_9,2,j)}
\renewcommand{\vvk}{(A_7,4,j-1)}\renewcommand{\vvl}{(A_4,2,j-1)}

\renewcommand{\Cseven}{\Cc_7^j}
\newcommand{\Dones}{\D_1^{j+1}}\newcommand{\Dthrees}{\D_3^{j+1}}
\newcommand{\Dfourx}{\D_4^{j-1}}
\newcommand{\Cfiveorig}{\Cc_5^j}

\begin{figure}
  \scalebox{.67}{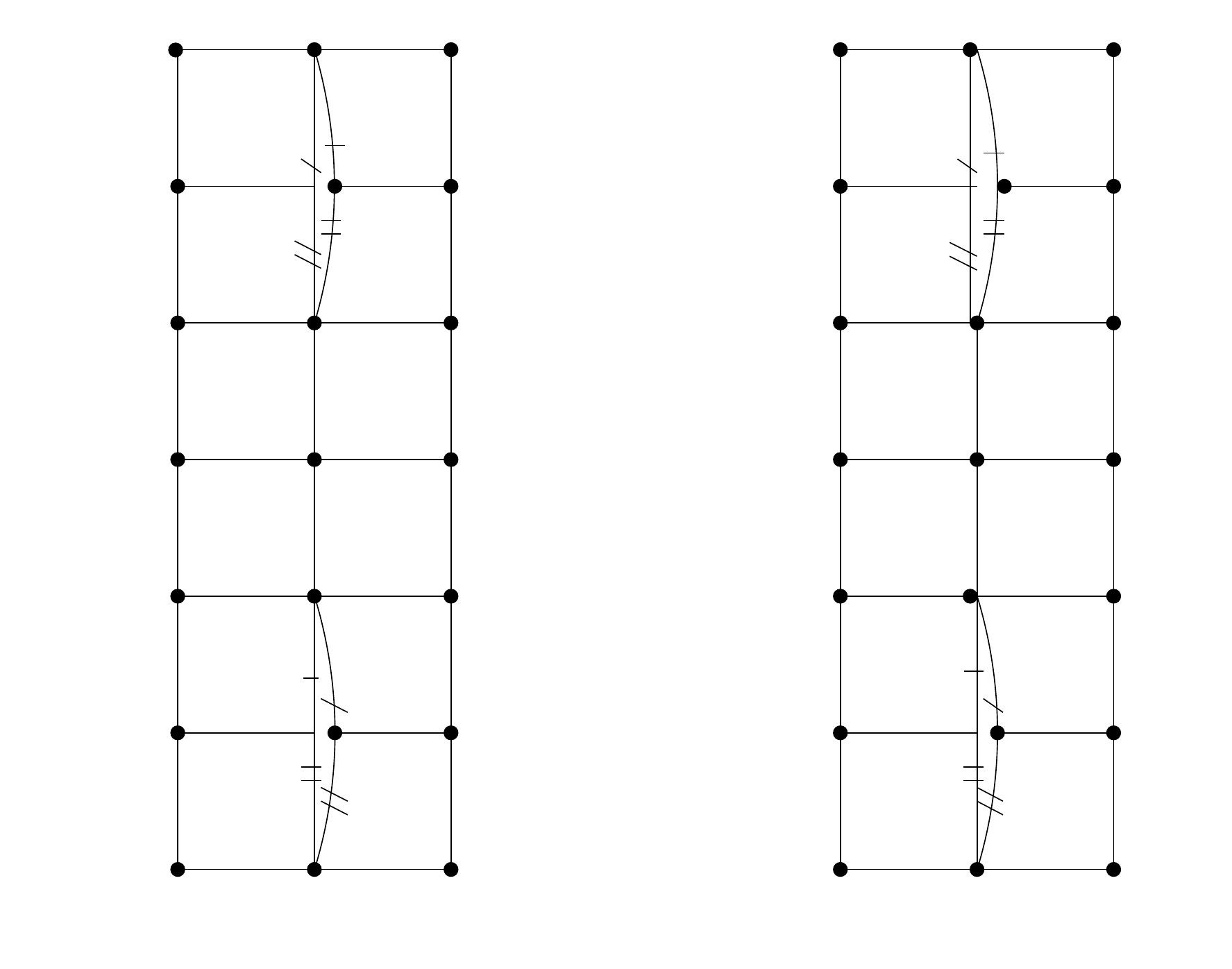}\\[3mm]   
  \scalebox{.67}{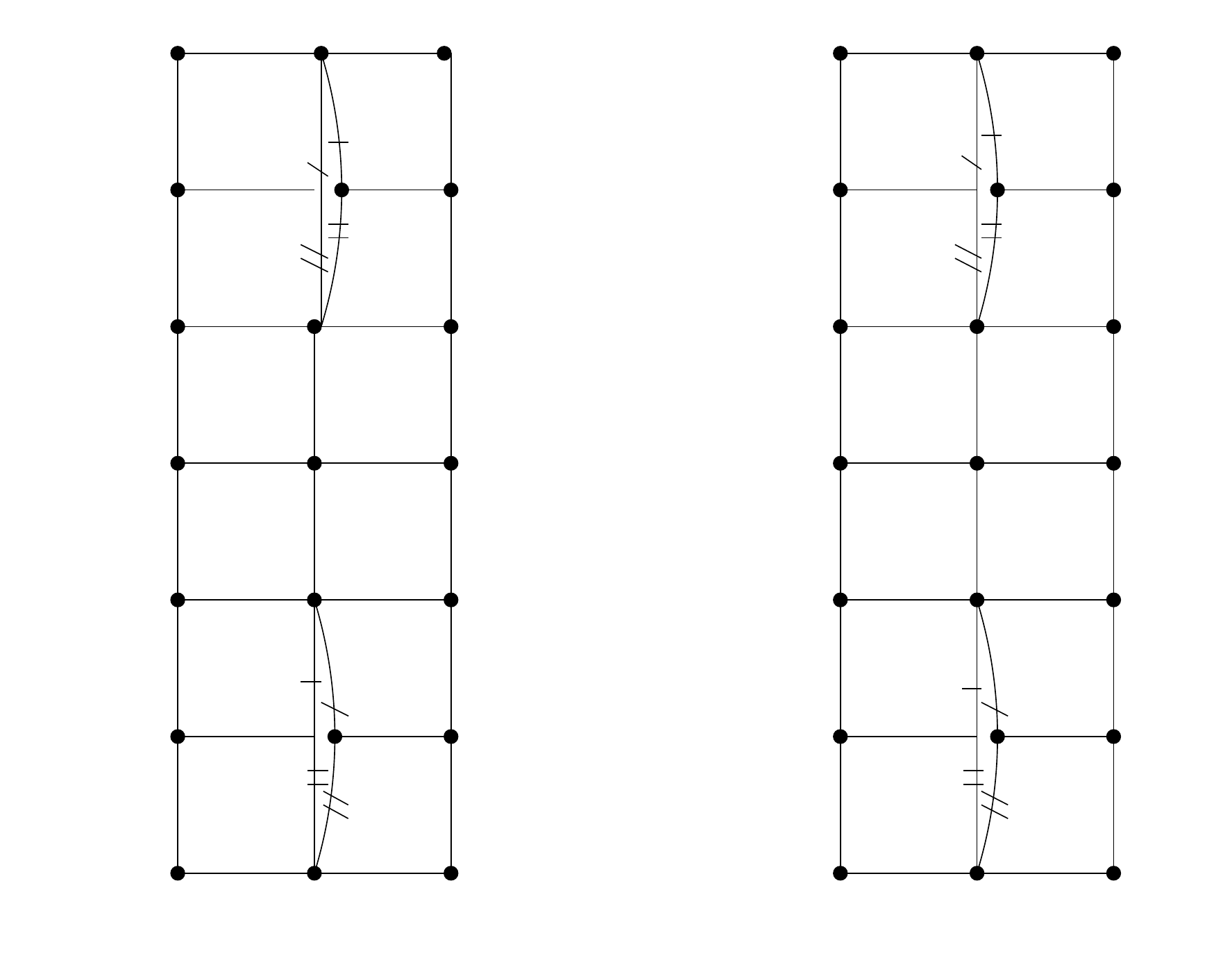}\\[2mm]
\caption{
  The $i$-th part of the decomposition of the origami $Y$ into 
  vertical cylinders. 
  An edge labelled by 
  $(A_i,h,j)$ is glued to the suitable edge of Square $(A_i,h,j)$. 
  Slits with same labels within a double cylinder
  are glued. Opposite horizontal unlabelled edges are glued. 
  \label{yvertical}
}
\end{figure}

 
\begin{lemma}\label{origamiy}
  For the origami $Y$ constructed in Definition~\ref{defy}, we have:
  \begin{enumerate}
  \item[i)]
    The origami $Y$ allows a normal covering $q: Y \to X$ of degree $2k$
    with Galois group $\emdeck(Y/X) \cong \ZZ/(2k\ZZ)$.
  \item[ii)]
    $Y$ decomposes into $4\cdot 2k$ horizontal double cylinders
    which are isometric to those on $X$. The same is true
    for the vertical direction.
  \item[iii)] 
    The covering $q:Y \to X$ is unramified. Thus the point $\Cc$
    has $12\cdot 2k$ preimages 
    under the map $p \circ q$. The points $\A$, $\B$ and $\D$ have $4\cdot 2k$
    preimages, respectively. 
  \item[iv)] 
    The genus of $Y$ is $24\cdot k + 1$.
  \end{enumerate}
\end{lemma}

\begin{proof}
  For {\bf i)} define $q$ by mapping Square $(A,h,j)$ of $Y$ to Square $(A,h)$
  of $X$. {\bf ii)} can be read off from Figure~\ref{xhorizontal}. The horizontal
  double cylinders are shown
  in Figure~\ref{yhorizontal} and the vertical ones in Figure~\ref{yvertical}.  
  It can be furthermore directly read off from Figure~\ref{xhorizontal} 
  that the monodromy 
  around each vertex is 0. Thus the covering is unramified and
  {\bf iii)} holds. Finally, the Riemann Hurwitz formula
  then gives us $2g_Y - 2 = (2g_X -2)\cdot 2k$, where $g_Y$ and $g_X$ are the genus
  of $Y$ and $X$, respectively. The claim follows, since $g_X = 13$ by 
  Proposition~\ref{propx}.
\end{proof}

\begin{remark}\label{notationcvertices}
We label the $12\cdot 2k$ preimages of $\Cc$ on the surface $Y$ by  
 $\Cij$ ($i \in \{1, \ldots,12\}$, $j \in \{1, \ldots, 2k\}$)
where $\Cij$ is the left lower vertex of the square $(A_i,3,j)$.
Recall the identification from Remark~\ref{identify} of
the cosets of the group
$G_6(2k) = \gamma(\PG{2k})$ with 
$\mathcal{C}(\PGammazwei:\PGammasechs)\times \ZZ/(2k\ZZ)$.
This assigns to $(A_i,j)$ a coset $G_6(2k)\cdot g_{i,j}$ with $g_{i,j} \in \PGammazwei$. We will
denote the vertex $\Cij$ also by $\Cc_{G_6(2k){g_{i,j}}}$.
Finally, we label the preimages of $\A_i$, $\B_i$ and $\D_i$
by $\Aij$, $\Bij$ and $\Dij$ ($i \in \{1, \ldots, 4\}$ and $j \in \{1,\ldots, 2k\}$), respectively, 
as shown in Figure~\ref{yhorizontal} and Figure~\ref{yvertical}.
Note, in particular, that $q(\Aij) = \A_i$
and similarly for the $\B$'s, $\Cc$'s and $\D$'s.
\end{remark}

Observe from Figure~\ref{yhorizontal} that the affine
homeomorphisms $f_T^2$ and $f_L^2$ on $X$ lift via $q$ to 
affine maps $\hat{f}^2_T$ and $\hat{f}^2_L$ on $Y$  which 
again fix the boundaries of
the horizontal double cylinders, respectively the boundaries
of the vertical double cylinders,  pointwise. Furthermore
you can read off from Figure~\ref{yhorizontal} that the rotation
of angle $\pi$ around one of the points $\Cim$ is well-defined. 
This defines an affine
homeomorphism $\hat{f}_{-I}$ with derivative $-I$
which fixes all points $\Cc^m_i$.

\begin{lemma}\label{actiony}
  The action of the group $\Gaffzweihat  = \; <\hat{f}_T^2,$, $\hat{f}_L^2, \hat{f}_{-I}>$
  has the following properties:
  \begin{enumerate}
    \item[i)] $\Gaffzweihat$ acts trivially on the set 
      $$\widehat{\mathcal{S}_{\mbox{\em \fs sing}}} = 
      \{\Aij, \Bij, \Dij|i \in\{1,\ldots, 4\}, j \in \{1, \ldots, 2k\}\}.$$
      \item[ii)] $\Gaffzweihat$
	is the full pointwise stabiliser \emstab($\widehat{\mathcal{S}_{\mbox{\fs {\em sing}}}}$) 
	in $\emaff(Y)$ of 
	the set $\widehat{\mathcal{S}_{\mbox{\em \fs sing}}}$. 
	Furthermore, $\Gaffzweihat$ is isomorphic to $\Gamma_2$.
    \item[iii)] 
      Recalling the notation from Remark~\ref{notationcvertices},
      we have, for all $g \in \PGammazwei$:
      \[
         \begin{array}{lcl}
	   \hat{f}_{T}^2(\Cc_{G_6(2k)\cdot g}) &=&  \Cc_{G_6(2k)\cdot gy^{-1}}, \\
	   \hat{f}_{L}^2(\Cc_{G_6(2k)\cdot g}) &=& \Cc_{G_6(2k)\cdot g y^{-1}x}, \\
	   \hat{f}_{-I}(\Cc_{G_6(2k)\cdot g}) &=& \Cc_{G_6(2k)\cdot g},  \mbox{ i.e. $\hat{f}_{-I}$ 
             fixes all vertices $\Cc_{G_6(2k)\cdot g}$}.
	 \end{array}
      \]

      \item[iv)] Let $\varphi: \Gaffzweihat \to \PGammazwei$ be the homomorphism
	defined by $\varphi(\hat{f}_T^2) = y$, $\varphi(\hat{f}_L^2) = x^{-1}y$
	and $\varphi(\hat{f}_{-I}) = $ {\em id}. Then we have for $\hat{f} \in \Gaffzweihat$
        and $g \in \PGammazwei$: 
	\[
	\begin{array}{l}
	  \hat{f}(\Cc_{G_6(2k)\cdot g}) =  \Cc_{G_6(2k)\cdot g \varphi(f)^{-1}}, 
	\end{array}
	\]
	Observe (for later use) that $\varphi = \gamma \circ D$, where 
        $D: \Gaffzweihat \to \PGammazwei \cong F_2$ is the derivative map and 
        $\gamma$ is the automorphism $x \mapsto y$, $y\mapsto x^{-1}y$
        defined in Lemma~\ref{change}.
        This in particular ensures that $\varphi$ is well-defined. 
  \end{enumerate}
\end{lemma}

\begin{proof}
  The claims on $\hat{f}_{-I}$ can be directly read off
  from Figure~\ref{yhorizontal}.
  Recall further that $\hat{f}_T^2$ and $\hat{f}_L^2$ act as horizontal (respectively vertical)
  multi partial Dehn twists on the horizontal (resp. vertical) double cylinders shown
  in Figure~\ref{yhorizontal} (resp. Figure~\ref{yvertical}) which fixes the boundaries and
  shifts by 2 on the middle lines. This immediately gives {\bf i)}. 
   For the proof of {\bf ii)} 
  we proceed similarly as in Lemma~\ref{actionx}: the only
  translation on $Y$ that fixes  
  $\widehat{\mathcal{S}_{\mbox{\fs sing}}}$
  pointwise is the identity, since such 
  a translation has to fix with $\A^1_1$ and $\B^1_1$
  also the horizontal edge between them. With the same arguments
  as in Lemma~\ref{actionx}, {\bf ii)} follows.
  Thus the map 
  $D:\stab(\widehat{\mathcal{S}_{\mbox{\fs sing}}}) \to \Gamma_2$
  is injective. It is surjective, since $T^2$, $L^2$ and $-I$
  generate $\Gamma_2$. 
  
  To prove {\bf iii)}, we read off from Figure~\ref{ydch}
  that $\hat{f}_T^2$ acts on the $\Cij$'s by the following
  permutations:
  \[
    \hat{f}_T^2: 
    (\Cc^j_1,  \Cc^{j+1}_5, \Cc^{j}_6)(\Cc^j_{11}, \Cc^j_{3}, \Cc^j_9)
    (\Cc^j_{10}, \Cc^j_8, \Cc^{j}_2)(\Cc^j_{12}, \Cc^{j+1}_7, \Cc^j_4),
  \]
  where  $j$ runs through  $\{1, \ldots, 2k\}$.
  By Proposition~\ref{actiononcosets},  this
  is precisely what multiplication with $y^{-1}$ does. In the same way,
  one obtains from Figure~\ref{yvertical}:
  \[
  \hat{f}_L^2:
  (\Cc^j_1, \Cc^j_{4}, \Cc^j_{11})(\Cc^j_5, \Cc^j_9, \Cc^j_{10})
  (\Cc^j_6, \Cc^j_2, \Cc^j_{12})(\Cc^j_3, \Cc^j_{7},\Cc^j_8)
  \]
  and again reads off from Proposition~\ref{actiononcosets} that this 
  corresponds to multiplication by $y^{-1}x$.
  
  Finally, {\bf iv)} follows from {\bf ii)}
  and {\bf iii)} by induction on the generators of $\Gaffzweihat$.
\end{proof}

We in particular
obtain the following corollary, which will be helpful 
in Section~\ref{step3}.

\begin{corollary} (to Lemma~\ref{actiony})\label{cor-util}\\
Let $\hat{f}$ be an affine homeomorphism of $Y$ such that 
$D(\hat{f}) \in \Gamma_2$. If we have for some $j \in \{1, \ldots, 2k\}$
that $\hat{f}(\A^j_1) = \A^j_1$ and $\hat{f}(\B^{j}_1) = \B^{j'}_1$
with some $j' \in \{1, \ldots, 2k\}$, then $\hat{f}$ is in $\Gaffzweihat$.
\end{corollary}

\begin{proof}
By Lemma~\ref{actiony} ii) there is an element $\hat{f}'$  in $\Gaffzweihat$
with $D(\hat{f}') = D(\hat{f})$. From $\hat{f}' \in \Gaffzweihat$ it follows
that $\hat{f}'(\A^j_1) = \A^j_1$ and $\hat{f}'(\B^{j'}_1) = \B^{j'}_1$. Hence 
$\tau = (\hat{f}')^{-1}\circ \hat{f}$ fixes $\A^j_1$ and we have 
$\tau(\B^{j}_1) = \B^{j'}_1$.
Since $\tau$ is a translation, the horizontal segment of length 1 between 
$\A^j_1$ and $\B^j_1$ must be mapped to a segment of same direction 
and same length. We read off from Figure~\ref{ydch} that $j' = j$.
Thus the translation $\tau$ fixes the horizontal segment from 
$\A^j_1$ to $\B^{j}_1$ and is therefore
the identity. This shows the claim.
\end{proof}

It turns out that we cannot only lift $(f_{T})^2$ and $(f_L)^2$
from $X$ to $Y$ but also $f_T$ and $f_L$, which needs a bit more care.
We do not need this for the proof of Theorem~\ref{t.A}. But since
it is an interesting statement on its own, we include the proof
in the following proposition. 

\begin{proposition}\label{vgY}
The Veech group of $Y$ is SL$(2,\ZZ)$. 
\end{proposition}

\begin{proof}
  We define lifts of $f_T$ and $f_L$ as follows:\\
  We denote the $4\cdot(2k)$ horizontal double cylinders by
  $\hat{Z}_{1,1}$, \ldots, $\hat{Z}_{4,2k}$ (see Figure~\ref{ydch}).
  Consider the two permutations 
  \[
  \begin{array}{lcl}
    \tau &=& (2,\ 2k)(3,\ 2k-1)\ldots(k,\ k+2) \; \mbox{ and }\\
    \rho &=& (1,\ 2)(3,\ 2k)(4,\ 2k-1)\ldots(k+1,\ k+2)\\
    
  \end{array}
  \]
  Define the map $\hat{f}_T$ as lift of $f_T$ that permutes
  the double cylinders $\hat{Z}_{1,j}$ and $\hat{Z}_{4,j}$  
  ($j\in \{1, \ldots, 2k\}$)
  by the permutation $\tau$, 
  and the $\hat{Z}_{2,j}$'s and $\hat{Z}_{3,j}$'s by 
  the permutation $\rho$, respectively, 
  i.e.$\hat{f}_T$ maps  $\hat{Z}_{1,j}$ to $\hat{Z}_{1,\tau(j)}$,  
  $\hat{Z}_{2,j}$ to $\hat{Z}_{2,\rho(j)}$,  
  $\hat{Z}_{3,j}$ to $\hat{Z}_{3,\rho(j)}$ and finally $\hat{Z}_{4,j}$ to $\hat{Z}_{4,\tau(j)}$. 
  We have to check
  that this is well-defined with respect to the gluings of the double cylinders. 
  This can be read
  off from Figure~\ref{ydch} as follows: 
  Recall that $f_T$ (resp. $f_L$) acts on the boundary of the horizontal (resp. vertical)
  cylinders as shift by 3.
  Thus e.g. the top edge 
  of square $(A_6,4,j+1)$, which lies in cylinder $\hat{Z}_{1,j}$, 
  is mapped by $\hat{f}_T^2$ to the top edge 
  of square $(A_1,3,\tau(j))$ in cylinder $\hat{Z}_{1,\tau(j)}$. 
  Square $(A_6,4,j+1)$
  is glued to Square $(A_9,2,j+1)$ in cylinder $\hat{Z}_{2,j+1}$. 
  The lower edge of Square $(A_9,2,j+1)$ 
  is mapped to the lower edge of Square 
  $(A_3,1,\rho(j+1))$. Thus we need that $\tau(j)= \rho(j+1)$. 
  Similarly the other
  five top edges of the double cylinder $\hat{Z}_{1,j}$ lead to the conditions:
  \[
  \begin{array}{l}
    \tau(j) = \rho(j+1),\; \tau(j) = \tau(j), \;
    \tau(j)+1 = \rho(j),\\ \tau(j) + 1 = \rho(j), 
    \mbox{ and } \tau(j)= \tau(j).
    \end{array}
  \]
  Furthermore, doing the same calculations for the 
  other three cylinders, we obtain equivalent relations. Thus
  $\hat{f}_T$ is well-defined, since $\tau$ and $\sigma$ 
  satisfy these conditions.\\
  The definition of $\hat{f}_L$ works in a similar way. Consider
  the double cylinders $\hat{C}_{1,j}$, \ldots, $\hat{C}_{4,j}$ (see Figure~\ref{yvertical}).
  Define the map $\hat{f}_L$ as lift of $f_L$ that permutes
  the double cylinders $\hat{C}_{1,j}$ and 
  $\hat{C}_{3,j}$  ($j\in \{1, \ldots, 2k\}$)
  by the permutation $\tau$, and the $\hat{C}_{2,j}$'s and 
  $\hat{C}_{4,j}$'s by 
  the permutation $\rho$. Similarly as above
  one obtains for each edge that lies on 
  the boundary of the vertical double cylinders an
  equation. One easily sees  that all non-trivial
  equations are equivalent to:
  \[\forall j:\, \rho(j) = \tau(j)+1 \mbox{ and } \rho(j+1)= \tau(j)\]
  Again $\rho$ and $\tau$ satisfy these conditions.
\end{proof}

\subsection{A criterion for the desired origamis}\label{step3}

In this section we present a sufficiency condition for coverings
$Z$ of the origami $Y$ from Section~\ref{step2}, that the
Veech group of $Z$ is contained in $\thegroup$. 
In Section~\ref{proofoftheorem}, we will present some explicit examples for 
such origamis $Z$.\\

Let $h: X_1 \to X_2$ be a finite covering of Riemann surfaces.
Recall for the following that the {\em ramification data}
of a point $P \in X_2$ consists of 
the ramification indices of all preimages
of $P$ counted with multiplicity. We denote this multiple set
by $\rmd(P,h)$. It will be a crucial point in the proof of 
Proposition~\ref{mainprop} that if $h$ is a translation covering
of finite translation surfaces (recall that we allow by definition
singular points on the surface) and $\hat{f}$ is an affine homeomorphism
of $X_1$ which descends to $f$ on $X_2$, then we have for all points $P \in X_2$:
$\rmd(f(P),h) = \rmd(P,h)$, i.e. a descend preserves ramification data.

\begin{proposition}\label{mainprop}
  Let $Y$ be the origami from Definition~\ref{defy}.
  Recall that from the construction we had coverings $q: Y \to X$
  and $p:X \to E[2]$ of degree $2k$ and $12$, respectively.
  Suppose that $r:Z \to Y$ is a finite covering
  such that 
  \begin{enumerate}    
    \item[A)]
      $\emrmd(\A, p \circ q \circ r)$, $\emrmd(\B, p \circ q \circ r)$, 
      $\emrmd(\Cc, p \circ q \circ r)$ and $\emrmd(\D, p \circ q \circ r)$
      are pairwise distinct.
    \item[B)]
      $\emrmd(\A_1,q \circ r) \neq \emrmd(\A_i, q\circ r)$ and 
      $\emrmd(\B_1,q \circ r) \neq \emrmd(\B_i, q\circ r)$ for $i \in \{2,3,4\}$.

    \item[C)]
      For some $j \in \{1, \ldots, 2k\}$ we have:
      $\emrmd(\A^j_1,r) \neq \emrmd(\A^{j'}_1,r)$ for all $j' \in \{1, \ldots, 2k\}$,
      $j' \neq j$.
    \item[D)]
      $\emrmd(\Cc^1_1,r) \neq \emrmd(\Cc^j_i,r)$ 
      for $j \in \{1,\ldots, 2k\}$, $i \in \{1,\ldots,12\}$ and $(i,j) \neq (1,1)$.
  \end{enumerate}
Then, the Veech group of $Z$ is a subgroup of $\thegroup$. 
\end{proposition}

It will be a main point of the proof of Proposition~\ref{mainprop} 
that all affine homeomorphisms of $Z$ descend to $Y$. This
is enforced by the ramification behaviour of 
the maps $r$, $q$ and $p$. 
A central step for showing this will be to
work with the universal 
covering $\univZstar$ of the punctured
surface $Z^*$ (defined below), look at affine
homeomorphisms there
and find criterions
ensuring that they descend to the finite surface $Y^*$. We do this
in a sequence of lemmata which prepare the proof of 
Proposition~\ref{mainprop}.

Consider in the following the sequence:
\begin{equation}\label{covs}
   \univZstar 
   \stackrel{u}{\to} Z^* 
   \stackrel{r}{\to} Y^*
   \stackrel{q}{\to} X^*
   \stackrel{p}{\to} E[2]^*
   \stackrel{\cdot [2]}{\to} E^*
\end{equation}
Here $E^* = E\backslash\{\infty\}$ ($E$ and $\infty$ 
defined as in the beginning of Section~\ref{s.EMcRSch} and $X^*$, $Y^*$ and $Z^*$  are $X$,
$Y$ and $Z$, respectively, with the preimages of $\infty$ under the corresponding
covering map removed. Hence $r$, $q$ and $p$ in the sequence in (\ref{covs})
are unramified coverings. Let furthermore $u: \univZstar \to Z^*$
be a universal covering. We lift the translation structure
on $Z^*$ to $\univZstar$ and call it $\tilde{\mu}$. Since all
translation structures that we consider were obtained as lifts from 
the structure $\mu$ on $E$, all coverings in (\ref{covs}) are translation coverings. 
Furthermore, by the
uniformisation theorem $(\univZstar, \tilde{\mu})$
is as Riemann surface biholomorphic to the Poincar\'e upper half plane $\HH$.
The group of deck transformations
\[\begin{array}{l}
  \deck(\univZstar/Z^*) \subseteq \deck(\univZstar/Y^*) \subseteq 
   \deck(\univZstar/X^*)\\[2mm]
   \quad \subseteq \deck(\univZstar/E[2]^*) \subseteq
   \deck(\univZstar/E^*) \cong \pi_1(E^*) 
  \end{array}
\]
of the coverings $u$, $r \circ u$, $q \circ r \circ u$, 
$p \circ q \circ r \circ u$ and $(\cdot [2])\circ p \circ q \circ r \circ u$
act as Fuchsian groups on $\HH$. They have all finite
index in 
$\deck(\univZstar/E^*)$, since they arise from finite coverings to $E^*$.
Thus the set of cusps (i.e. fixed points on the boundary of $\HH$
of parabolic elements) of these five groups coincide. We 
denote it by $\cusps$.
The covering $u$ may be extended to a continuous 
map from $\overline{\HH} = \HH \cup \cusps$ (endowed with the horocycle 
topology) and we obtain the chain of continuous maps:
\begin{equation*}
   \overline{\HH} 
   \stackrel{u}{\to} Z 
   \stackrel{r}{\to} Y
   \stackrel{q}{\to} X
   \stackrel{p}{\to} E[2]
   \stackrel{\cdot [2]}{\to} E
\end{equation*}

\begin{definition}
Let $\mathcal{S}$ be a set, let $\mathcal{B} = \{b_1, \ldots, b_l\}$
be a partition of $\mathcal{S}$ and let $f$ act by a permutation on $\mathcal{S}$. We say
that {\em $f$ fixes the partition}, if 
$\forall i \in \{1, \ldots, l \}:\; f(b_i) = b_i$.
\end{definition}

For the following lemma recall e.g. from \cite[Prop. 2.1, Prop. 2.6]{Sc2} that any 
affine homeomorphism
of $\univZstar$ descends to $E^*$ via $u \circ r \circ q \circ p \circ \cdot[2]$,
as well as to $E[2]^*$ via $u \circ r \circ q \circ p$.
In particular any affine homeomorphism of $\HH$ 
can thus be continuously extended to $\HH \cup \cusps$. 
 
\begin{lemma}\label{stababcd}
In the above situation consider the partition of the set of cusps 
$\cusps$ induced by the four points $\A$, $\B$, $\Cc$ and $\D$ on $E[2]$,
i.e. the partition that consists of the four classes 
\[
  \begin{array}{l}
  \cusps_{\A}  = (p\circ q\circ r \circ u)^{-1}(\A),\;\;
  \cusps_{\B} = (p\circ q\circ r \circ u)^{-1}(\B), \\
  \cusps_{\Cc} = (p\circ q\circ r \circ u)^{-1}(\Cc) \mbox{ \em  and }
  \cusps_{\D} = (p\circ q\circ r \circ u)^{-1}(\D).
  \end{array}
\]
Let $\mathcal{H} = \{ f \in \emaff(\univZstar, \tilde{\mu})| 
       \mbox{$f$ fixes the partition }
       \cusps_{\A} \sqcup \cusps_{\B} \sqcup \cusps_{\Cc} \sqcup \cusps_{\D}\}$.
Observe that here (and in the following) we use 
the extension of $f$ to  $\overline{\HH}$. In this sense $f$,
acts on $\cusps$.
We then have: 
\[ 
\begin{array}{lcl}
  h \in \mathcal{H}
  &\Leftrightarrow& \mbox{ the descend of $h$ to $E[2]$ fixes the four points $\A$, $\B$, $\Cc$ and $\D$}\\
  &\Leftrightarrow& D(h) \in \Gamma_2
\end{array}
\]
\end{lemma}

\begin{proof}
The first equivalence directly follows from the definition
of the partition. For the second equivalence observe that 
the action of an affine homeomorphism $f$ of $E[2]$ on 
$\A$, $\B$, $\Cc$ and $\D$ is equivalent to the action of 
the image of $D(f)$ in $\slzwei(\ZZ/2\ZZ)$ on 
$0\choose 0$, $1\choose 0$, $0\choose 1$ and $1\choose 1$
in $\ZZ/2\ZZ$.
\end{proof}

The following simple argument will be crucial later,
therefore we provide a proof of it.

\begin{lemma}\label{basic}
Let $p:A^* \to B^*$ be an unramified normal covering 
and let $u: \widetilde{A^*} \to A^*$ be
a universal covering. Let $f$ be a homeomorphism of $B^*$ and $\tilde{f}$
a lift of $f$ to $\widetilde{A^*}$ via $p \circ u$. Suppose that 
$f$ can be lifted to $A^*$. Then $\tilde{f}$ descends to 
$A^*$.
\end{lemma}

\begin{proof}
This follows directly from the fact
that $p$ is normal, since if there is a lift $\hat{f}$ of $f$
to $A^*$, then we may lift $\hat{f}$ to some $\tilde{f}'$
on $\widetilde{A^*}$. Then $\tilde{f}' \circ \tilde{f}^{-1}$
is in $\deck(\widetilde{A^*}/B^*)$. But $\deck(\widetilde{A^*}/A^*)$
is normal in this group, thus $\tilde{f}' \circ \tilde{f}^{-1}$
descends to $A^*$ and thus also  $\tilde{f}$ does.
\end{proof}

\begin{lemma}\label{descendlemma}
  Suppose we are in the situation of Lemma~\ref{stababcd} and 
  let $X^*$ and $Y^*$ be the punctured surfaces from above. 
  \begin{enumerate}
  \item[i)]
    All elements of $\HHH$ descend to $X^*$.
  \item[ii)]
    Define $\cusps_{\A_i} := (q \circ r \circ u)^{-1}(\A_i)$, 
    $\cusps_{\B_i} := (q \circ r \circ u)^{-1}(\B_i)$ and
    $$
       \HHH_2 = \{f \in \HHH| 
       f(\cusps_{\A_1}) =\cusps_{\A_1} \mbox{ and } f(\cusps_{\B_1}) =\cusps_{\B_1} \}.
    $$
    Then all elements of \,$\HHH_2$ descend to $Y^*$.
  \end{enumerate}
\end{lemma}

\begin{proof}
{\bf i)}
Let $\tilde{f}$ be in $\HHH$.
By Lemma~\ref{stababcd} $\tilde{f}$
descends to some  affine homeomorphism $\bar{f}$ on 
$E[2]$ which fixes $\A$, $\B$, $\Cc$ and
$\D$ pointwise and whose derivative $A = D(f)$
is in $\Gamma_2$. By Lemma~\ref{actionx}
we have an affine homeomorphism of $X^*$ with derivative $A$
which fixes the points $\A_i$, $\B_i$, $\D_i$ and permutes
the $\Cc_i$. Thus its descend $\bar{f}'$ to $E[2]$ 
fixes the points $\A$, $\B$, $\Cc$, $\D$ and also has derivative $A$.
This determines it uniquely, thus $\bar{f} = \bar{f}'$.
By Lemma~\ref{basic} it follows that $\tilde{f}$ descends to $X^*$.

{\bf ii)} Let now $\tilde{f}$ be in $\HHH_2$.
We may use similar arguments as in {\bf i)} replacing
$E[2]$ and $X^*$ by $X^*$ and $Y^*$ (resp.): by {\bf i)} we have
that $\tilde{f}$ descends to some $\bar{f}$ on $X^*$. By assumption it fixes the
point $\A_1$ and $\B_1$. By Lemma~\ref{actiony} there
is an affine homeomorphism of $Y^*$ 
with same derivative, which fixes $\A^1_1$ and $\B^1_1$ and
descends to some $\bar{f}'$  on $X^*$
which fixes $\A_1$ and $\B_1$. Thus $\bar{f} \circ \bar{f}^{-1}$
is a translation of $X^*$ which fixes $\A_1$ and $\B_1$. In the proof
of Lemma~\ref{actionx} we have seen that this implies that
$\bar{f} \circ \bar{f}^{-1}$ is the identity and hence $\bar{f} = \bar{f}'$.
We may again conclude using Lemma~\ref{basic} and the fact
that the covering $q$ is normal.
\end{proof}

\begin{proof}[Proof of Proposition~\ref{mainprop}]
Let $\fhathat$ be in $\aff(Z)$. Again we use that
since the derived vectors
of the saddle connections of $Z$ span the lattice $\ZZ^2$, 
$\fhathat$ preserves the vertices of the squares
that form $Z$. In particular $\fhathat$ restricts
to an affine homeomorphism of $Z^*$.
 Let $\tilde{f}$ be a lift
of it to the universal covering $\HH$. It follows from A) that
$\tilde{f} \in \HHH$ and thus in particular
its derivative is in $\Gamma_2$. It furthermore follows
from  A) and B) that $\tilde{f}$ is even in $\HHH_2$.
Thus by Lemma~\ref{descendlemma} $\tilde{f}$ descends to
$\hat{f}$ on $Y^*$. From C)
we obtain that $\hat{f}$ fixes $\A^j_1$. Furthermore we obtain from
A) and B) that $\B^j_1$ is mapped to $\B^{j'}_{1}$ for some $j' \in \{1, \ldots, 2k\}$. 
It follows from Corollary~\ref{cor-util} that $\hat{f}$ is in $\Gaffzweihat$. 
From D) we obtain that $\hat{f}$ also fixes $\Cc^1_1$.
Recall from Lemma~\ref{actiony} that 
$\hat{f}(\Cc_{G_6(2k)\cdot g}) = \Cc_{G_6(2k)\cdot g\varphi(f)^{-1}}$
with $\varphi = \gamma\circ D$.
Since we have $\Cc^1_1 = \Cc_{\mbox {\fs id}}$, this implies 
$G_6(2k) = G_6(2k)\cdot \varphi(f)^{-1} $ and thus
$\varphi(f) = \gamma(D(f))$ is in $G_6(2k) = \gamma(\PGammasechs(2k))$.
Hence we have $D(f)$ is in $\PGammasechs(2k)$.
\end{proof}

\section{Proof of Theorem~\ref{t.A}}\label{proofoftheorem}

We finally complete the 
proof of Theorem~\ref{t.A} by giving a specific family $Z_k$ of examples
which satisfies 
the conditions in Proposition~\ref{mainprop}.  
For a given $k$, the origami $Z = Z_k$ is a covering of degree 2 of the origami $Y$ 
from Section~\ref{step2} and constructed as follows:
We take two copies of $Y$ and slit them along the following three edges: 
the upper edge of Square $(A_2,3,2k)$ and  of $(A_5,1,1)$
and the right edge of $(A_1,3,1)$ (see Figure~\ref{zexample}). 
Each slit is reglued
with the corresponding slit of the other copy. In the following
definition we give a formal description of this origami
by its permutations. Recall for this that the $96k$ squares of the origami $Y$ 
correspond to the elements
in the set $ M_Y = \PGammazwei/\PGammasechs \times \{1,2,3,4\} \times \ZZ/(2k\ZZ)$.

\begin{definition}\label{origamiz}

Let $\sigma^Y_a$ and $\sigma^Y_b$ in $S_{M_Y}$ be the permutations which define
the horizontal and the vertical gluings of $Y$. Furthermore, let 
$h: M_Y \to \ZZ/2\ZZ$ and $v: M_Y \to \ZZ/2\ZZ$ be the two maps defined by 
\[
  \begin{array}{lcl}
    h((A,i,j)) = 1 &\Leftrightarrow& (A,i,j) = (A_1,3,1),\\
    v((A,i,j)) = 1 &\Leftrightarrow& (A,i,j) \in \{(A_2,3,2k), (A_5,1,1) \}.
  \end{array}
\]
Now, define $Z = Z_k$ to be the origami
which consists of $192k$ squares labelled by the elements of 
$M_Z = \PGammazwei/\PGammasechs \times \{1,2,3,4\} \times \ZZ/(2k\ZZ)\times \ZZ/(2\ZZ)$
with the following gluing rules: 
\[
   \begin{array}{l}
     \sigma^Z_a: (A,i,j,l) \mapsto (\sigma^Y_a(A,i,j),l + h(A,i,j)) \mbox{ defines the horizontal gluings}\\ 
     \sigma^Z_b: (A,i,j,l) \mapsto (\sigma^Y_b(A,i,j),l + v(A,i,j)) \mbox{ defines  the vertical gluings.}
   \end{array}
\]
\end{definition}

\begin{proposition}
The origami $Z$ from Definition~\ref{origamiz} has the following properties:
\begin{itemize}
\item[i)] 
  It allows a degree 2 covering $r$ to the origami $Y$.
\item[ii)] 
  The covering $r$ defined in i) is ramified precisely
  over the four points $\Cc^1_1$, $\A^{2k}_1$, $\B^1_1$ and $\B^{2k}_1$. The genus
of $Z$ is $48k+3$. 
\item[iii)]
  Its Veech group is contained in $\thegroup$. 
\end{itemize}
\end{proposition}

\begin{proof}
The map $(A,i,j,l) \mapsto (A,i,j)$ defines a covering
of degree 2. Hence we have {\bf i)}. 
For {\bf ii)}, you may further read off from Figure~\ref{zexample}
that a small loop on $Y$ around a vertex meets an odd number of cuts
if and only if the vertex is $\A^{2k}_1$, $\B^{2k}_1$, $\Cc^1_1$ or $\B^1_1$.
Thus these are the ramification points.
Applying the Riemann-Hurwitz formula gives the genus, since by 
Proposition~\ref{origamiy} the genus of $Y$ is $24k+1$.\\
Recall from Proposition~\ref{propx} that the ramification index of $p: X \to E[2]$ 
at $\A_i$, $\B_i$, $\D_i$ ($i \in \{1,2,3,4\}$) is 3, respectively,
and $p$ is unramified at all $\Cc_j$ ($j \in \{1, \ldots, 12\}$). Furthermore 
$q:Y \to X$ is unramified by 
Lemma~\ref{origamiy}. Thus we obtain from ii) for the ramification data:
\[ 
   \begin{array}{l}
     \rmd(\A,p\circ q \circ r) = \{6, \underbrace {3, \ldots, 3}_{2\cdot(8k-1)}\},\;\;
     \rmd(\B,p\circ q \circ r) = \{6,6,\underbrace {3, \ldots, 3}_{2\cdot(8k-2)}\}\\[8mm]
     \rmd(\Cc,p\circ q \circ r) = \{2, \underbrace {1, \ldots, 1}_{2\cdot(24k-1)}\},\;\;
     \rmd(\D,p\circ q \circ r) = \{\underbrace {3, \ldots, 3}_{16\cdot k}\}
   \end{array}
\]
In particular, these four ramification data are different and 
Condition A) in Proposition~\ref{mainprop} is satisfied.
Furthermore, we have for $i \neq 1$: 
\[
\begin{array}{lcl}
  \rmd(\A_1,q\circ r) = \{2, \underbrace {1,\ldots, 1}_{2(2k-1)}\}
  &\neq& \rmd(\A_i,q\circ r) = \{\underbrace {1,\ldots, 1}_{4k}\} \mbox{ and }\\
  \rmd(\B_1,q\circ r) = \{2,2, \underbrace {1,\ldots, 1}_{2(2k-2)}\}
  &\neq& \rmd(\B_i,q\circ r) = \{\underbrace {1,\ldots, 1}_{4k}\}.
\end{array}
\]
Thus Condition B) holds. Finally, the ramification behaviour of $r$
shown in ii) ensures that Condition C) and D) are satisfied.
Thus the claim follows from Proposition~\ref{mainprop}.
\end{proof}

In particular, this proposition combined with Proposition~\ref{p.A} says that the Teichm\"uller curve of the origami $Z_k$ has complementary series for every $k\geq 3$. In particular, our ``smallest'' example of Teichm\"uller curve with complementary series corresponds to $Z_3$. The following corollary gives a description of $Z_3$ in terms of pairs of permutations:

\begin{corollary}\label{c.origamiZ3}
The Teichm\"uller curve of the following origami
(see Figure~\ref{picture} in Appendix B)  allows a 
complementary series:
\begin{longtable}{lcl}
  $\sigma_a$ &=& 
  $(1, 13, 193, 207, 243, 253)
  (2, 14, 194, 208, 244, 254)$\\
  &&
  $(3, 15, 195, 209, 245, 255)
  (4, 16, 196, 210, 246, 256)$\\
  &&
  $(5, 17, 197, 211, 247, 257)
  (6, 18, 198, 212, 248, 258)$\\
  &&
  $(7, 19, 199, 213, 249, 259)
  (8, 20, 200, 214, 250, 260)$\\
  &&
  $(9, 21, 201, 215, 251, 261)
  (10, 22, 202, 216, 252, 262)$\\
  &&
  $(11, 23, 203, 205, 241, 263)
  (12, 24, 204, 206, 242, 264)$\\
  &&
  $(25, 38, 266, 280, 220, 230, 26, 37, 265, 279, 219, 229)$\\
  &&
  $(27, 39, 267, 281, 221, 231)
  (28, 40, 268, 282, 222, 232)$\\
  &&
  $(29, 41, 269, 283, 223, 233)
  (30, 42, 270, 284, 224, 234)$\\
  &&
  $(31, 43, 271, 285, 225, 235)
  (32, 44, 272, 286, 226, 236)$\\
  &&
  $(33, 45, 273, 287, 227, 237)
  (34, 46, 274, 288, 228, 238)$\\ 
  &&
  $(35, 47, 275, 277, 217, 239)
  (36, 48, 276, 278, 218, 240)$\\
  &&
  $(49, 61, 433, 445, 337, 349)
  (50, 62, 434, 446, 338, 350)$\\
  &&
  $(51, 63, 435, 447, 339, 351)
  (52, 64, 436, 448, 340, 352)$\\
  &&
  $(53, 65, 437, 449, 341, 353)
  (54, 66, 438, 450, 342, 354)$\\
  &&
  $(55, 67, 439, 451, 343, 355)
  (56, 68, 440, 452, 344, 356)$\\
  &&
  $(57, 69, 441, 453, 345, 357)
  (58, 70, 442, 454, 346, 358)$\\
  &&
  $(59, 71, 443, 455, 347, 359)
  (60, 72, 444, 456, 348, 360)$\\
  &&
  $(73, 85, 361, 373, 457, 469)
  (74, 86, 362, 374, 458, 470)$\\
  &&
  $(75, 87, 363, 375, 459, 471)
  (76, 88, 364, 376, 460, 472)$\\
  &&
  $(77, 89, 365, 377, 461, 473)
  (78, 90, 366, 378, 462, 474)$\\
  &&
  $(79, 91, 367, 379, 463, 475)
  (80, 92, 368, 380, 464, 476)$\\
  &&
  $(81, 93, 369, 381, 465, 477)
  (82, 94, 370, 382, 466, 478)$\\
  &&
  $(83, 95, 371, 383, 467, 479)
  (84, 96, 372, 384, 468, 480)$\\
  &&
  $(97, 109, 385, 397, 481, 493)
  (98, 110, 386, 398, 482, 494)$\\
  &&
  $(99, 111, 387, 399, 483, 495)
  (100, 112, 388, 400, 484, 496)$\\
  &&
  $(101, 113, 389, 401, 485, 497)
  (102, 114, 390, 402, 486, 498)$\\
  &&
  $(103, 115, 391, 403, 487, 499)
  (104, 116, 392, 404, 488, 500)$\\
  &&
  $(105, 117, 393, 405, 489, 501)
  (106, 118, 394, 406, 490, 502)$\\
  &&
  $(107, 119, 395, 407, 491, 503)
  (108, 120, 396, 408, 492, 504)$\\
  &&
  $(121, 133, 505, 517, 409, 421)
  (122, 134, 506, 518, 410, 422)$\\
  &&
  $(123, 135, 507, 519, 411, 423)
  (124, 136, 508, 520, 412, 424)$\\
  &&
  $(125, 137, 509, 521, 413, 425)
  (126, 138, 510, 522, 414, 426)$\\
  &&
  $(127, 139, 511, 523, 415, 427)
  (128, 140, 512, 524, 416, 428)$\\
  &&
  $(129, 141, 513, 525, 417, 429)
  (130, 142, 514, 526, 418, 430)$\\
  &&
  $(131, 143, 515, 527, 419, 431)
  (132, 144, 516, 528, 420, 432)$\\
  &&
  $(145, 167, 539, 551, 299, 301)
  (146, 168, 540, 552, 300, 302)$\\
  &&
  $(147, 157, 529, 541, 289, 303)
  (148, 158, 530, 542, 290, 304)$\\
  &&
  $(149, 159, 531, 543, 291, 305)
  (150, 160, 532, 544, 292, 306)$\\
  &&
  $(151, 161, 533, 545, 293, 307)
  (152, 162, 534, 546, 294, 308)$\\
  &&
  $(153, 163, 535, 547, 295, 309)
  (154, 164, 536, 548, 296, 310)$\\
  &&
  $(155, 165, 537, 549, 297, 311)
  (156, 166, 538, 550, 298, 312)$\\
  &&
  $(169, 183, 315, 325, 553, 565)
  (170, 184, 316, 326, 554, 566)$\\
  &&
  $(171, 185, 317, 327, 555, 567)
  (172, 186, 318, 328, 556, 568)$\\
  &&
  $(173, 187, 319, 329, 557, 569)
  (174, 188, 320, 330, 558, 570)$\\
  &&
  $(175, 189, 321, 331, 559, 571)
  (176, 190, 322, 332, 560, 572)$\\
  &&
  $(177, 191, 323, 333, 561, 573)
  (178, 192, 324, 334, 562, 574)$\\
  &&
  $(179, 181, 313, 335, 563, 575)
  (180, 182, 314, 336, 564, 576)$,\\
  $\sigma_b$ &=&
  $(1, 265, 289, 553, 433, 73)
  (2, 266, 290, 554, 434, 74)$\\
  &&
  $(3, 267, 291, 555, 435, 75)
  (4, 268, 292, 556, 436, 76)$\\
  &&
  $(5, 269, 293, 557, 437, 77)
  (6, 270, 294, 558, 438, 78)$\\
  &&
  $(7, 271, 295, 559, 439, 79)
  (8, 272, 296, 560, 440, 80)$\\
  &&
  $(9, 273, 297, 561, 441, 81)
  (10, 274, 298, 562, 442, 82)$\\
  &&
  $(11, 275, 299, 563, 443, 83, 12, 276, 300, 564, 444, 84)$\\
  &&
  $(13, 229, 157, 565, 493, 133)
  (14, 230, 158, 566, 494, 134)$\\
  &&
  $(15, 231, 159, 567, 495, 135)
  (16, 232, 160, 568, 496, 136)$\\
  &&
  $(17, 233, 161, 569, 497, 137)
  (18, 234, 162, 570, 498, 138)$\\
  &&
  $(19, 235, 163, 571, 499, 139)
  (20, 236, 164, 572, 500, 140)$\\
  &&
  $(21, 237, 165, 573, 501, 141)
  (22, 238, 166, 574, 502, 142)$\\
  &&
  $(23, 239, 167, 575, 503, 143)
  (24, 240, 168, 576, 504, 144)$\\
  &&
  $(25, 97, 505, 529, 169, 193, 26, 98, 506, 530, 170, 194)$\\
  &&
  $(27, 99, 507, 531, 171, 195)
  (28, 100, 508, 532, 172, 196)$\\
  &&
  $(29, 101, 509, 533, 173, 197)
  (30, 102, 510, 534, 174, 198)$\\
  &&
  $(31, 103, 511, 535, 175, 199)
  (32, 104, 512, 536, 176, 200)$\\
  &&
  $(33, 105, 513, 537, 177, 201)
  (34, 106, 514, 538, 178, 202)$\\
  &&
  $(35, 107, 515, 539, 179, 203)
  (36, 108, 516, 540, 180, 204)$\\
  &&
  $(37, 61, 469, 541, 325, 253)
  (38, 62, 470, 542, 326, 254)$\\
  &&
  $(39, 63, 471, 543, 327, 255)
  (40, 64, 472, 544, 328, 256)$\\
  &&
  $(41, 65, 473, 545, 329, 257)
  (42, 66, 474, 546, 330, 258)$\\
  &&
  $(43, 67, 475, 547, 331, 259)
  (44, 68, 476, 548, 332, 260)$\\
  &&
  $(45, 69, 477, 549, 333, 261)
  (46, 70, 478, 550, 334, 262)$\\
  &&
  $(47, 71, 479, 551, 335, 263)
  (48, 72, 480, 552, 336, 264)$\\
  &&
  $(49, 361, 145, 313, 385, 121)
  (50, 362, 146, 314, 386, 122)$\\
  &&
  $(51, 363, 147, 315, 387, 123)
  (52, 364, 148, 316, 388, 124)$\\
  &&
  $(53, 365, 149, 317, 389, 125)
  (54, 366, 150, 318, 390, 126)$\\
  &&
  $(55, 367, 151, 319, 391, 127)
  (56, 368, 152, 320, 392, 128)$\\
  &&
  $(57, 369, 153, 321, 393, 129)
  (58, 370, 154, 322, 394, 130)$\\
  &&
  $(59, 371, 155, 323, 395, 131)
  (60, 372, 156, 324, 396, 132)$\\
  &&
  $(85, 109, 421, 301, 181, 349)
  (86, 110, 422, 302, 182, 350)$\\
  &&
  $(87, 111, 423, 303, 183, 351)
  (88, 112, 424, 304, 184, 352)$\\
  &&
  $(89, 113, 425, 305, 185, 353)
  (90, 114, 426, 306, 186, 354)$\\
  &&
  $(91, 115, 427, 307, 187, 355)
  (92, 116, 428, 308, 188, 356)$\\
  &&
  $(93, 117, 429, 309, 189, 357)
  (94, 118, 430, 310, 190, 358)$\\
  &&
  $(95, 119, 431, 311, 191, 359)
  (96, 120, 432, 312, 192, 360)$\\
  &&
  $(205, 277, 397, 517, 445, 373)
  (206, 278, 398, 518, 446, 374)$\\
  &&
  $(207, 279, 399, 519, 447, 375)
  (208, 280, 400, 520, 448, 376)$\\
  &&
  $(209, 281, 401, 521, 449, 377)
  (210, 282, 402, 522, 450, 378)$\\
  &&
  $(211, 283, 403, 523, 451, 379)
  (212, 284, 404, 524, 452, 380)$\\
  &&
  $(213, 285, 405, 525, 453, 381)
  (214, 286, 406, 526, 454, 382)$\\
  &&
  $(215, 287, 407, 527, 455, 383)
  (216, 288, 408, 528, 456, 384)$\\
  &&
  $(217, 337, 457, 481, 409, 241)
  (218, 338, 458, 482, 410, 242)$\\
  &&
  $(219, 339, 459, 483, 411, 243)
  (220, 340, 460, 484, 412, 244)$\\
  &&
  $(221, 341, 461, 485, 413, 245)
  (222, 342, 462, 486, 414, 246)$\\
  &&
  $(223, 343, 463, 487, 415, 247)
  (224, 344, 464, 488, 416, 248)$\\
  &&
  $(225, 345, 465, 489, 417, 249)
  (226, 346, 466, 490, 418, 250)$\\
  &&
  $(227, 347, 467, 491, 419, 251)
  (228, 348, 468, 492, 420, 252)$
\end{longtable}
The genus of the origami is 147.
It consists of 95 horizontal and 94 vertical cylinders and it is in 
the stratum $\mathcal{H}(1,5,5,5,\underbrace{2,\ldots,2}_{138})$.
\end{corollary}

\section*{Acknowledgements} 

The authors are grateful to A. Avila, N. Bergeron, P. Hubert and J.-C. Yoccoz for discussions related to this note. 
The second author would like to thank Florian Nisbach for proofreading of some parts of the text. Both authors also 
would like to thank the following institutions for their hospitality during the preparation of this text: HIM - Bonn (in the occasion of the ``Geometry and Dynamics of Teich\-m\"uller Spaces'' trimester, May - August,  2010) and MFO - Oberwolfach (in the occasion of the ``Billiards, Flat Surfaces, and Dynamics on Moduli Spaces'' workshop, May 8 - 14, 2011).

\newpage

\appendix

\section{Existence of arithmetic Teichm\"uller curves with complementary series}\label{a.A}
The facts in this Appendix are certainly well-known among experts, but we include this Appendix for sake of completeness. In the sequel, we will follow some arguments that grew out of conversations of the first author with A. Avila and J.-C. Yoccoz.

Consider the level $2$ principal congruence subgroup $\Gamma_2$ of $SL(2,\mathbb{Z})$. Recall that its image $P\Gamma_2$ in $PSL(2,\mathbb{Z})$ is the free subgroup generated by 
$$x:=\left(\begin{array}{cc}1 & 2 \\ 0 & 1\end{array}\right) \quad \textrm{and} \quad y:=\left(\begin{array}{cc} 1 & 0 \\ 2 & 1\end{array}\right).$$
It follows that we have a natural surjective homomorphism $\rho:P\Gamma_2\to\mathbb{Z}$ obtained by counting the number of occurrences of $x$ into a given word $w=w(x,y)\in P\Gamma_2$. In particular, by taking the reduction modulo $N$, we have a family of natural surjective homomorphisms $\rho_N:\Gamma_2\to
\mathbb{Z}/N\ZZ$. We define 
$$\Gamma_2(N):=\textrm{Ker}(\rho_N).$$

\begin{proposition}The regular representation of $\textrm{SL}(2,\mathbb{R})$ on $L^2(\mathbb{H}/\Gamma_2(N))$ exhibits complementary series for all sufficiently large $N$.
\end{proposition}

\begin{proof}Denoting by $\mathcal{F}_1:=\{z\in\mathbb{H}: |\textrm{Re}(z)|\leq 1/2, |z|\geq 1\}$ the standard fundamental domain of the action of $SL(2,\mathbb{Z})$ on $\mathbb{H}$, one has that $\mathcal{F}_2:=\bigcup\limits_{l=1}^6\alpha_l^{-1}(\mathcal{F}_1)$, where $\alpha_1=\textrm{Id}$, $\alpha_2 = T = \left(\begin{array}{cc}1&1\\0&1\end{array}\right)$, $\alpha_3=S=\left(\begin{array}{cc}0&-1\\1&0\end{array}\right)$, $\alpha_4=TS$, $\alpha_5=ST$ and $\alpha_6=T^{-1}ST$, is a fundamental domain for the action of $\Gamma_2$ on $\mathbb{H}$. In particular, by definition, $\mathcal{F}_N:=\bigcup\limits_{j=0}^{N-1}x^{j}(\mathcal{F}_2)$ is a fundamental domain for the action of $\Gamma_2(N)$ on $\mathbb{H}$. 

Let $U:=\mathcal{F}_2\cap\{z\in\mathbb{H}:|\textrm{Im}(z)|<1\}$ and $V:=x^{n}(U)$ where $n=\lfloor N/2\rfloor$. Since the hyperbolic distance $\rho(z,w)$ between $z,w\in\mathbb{H}$ verifies $\cosh\rho(z,w)=1+|z-w|^2/2\textrm{Im}(z)\textrm{Im}(w)$, it follows that the hyperbolic distance $d(U,V)$ between $U$ and $V$ satisfies
$$\cosh d(U,V)\geq1+2(n-1)^2$$
because $|\textrm{Im}(z)|, |\textrm{Im}(w)|\leq 1$ and $|z-w|\geq 2(n-1)$ for any $z\in U$, $w\in V=x^n(U)$.
From our choices of $U$ and $V$ (inside the fundamental domain $\mathcal{F}_N$), it follows that $U$ and $V$ are far apart by $\textrm{arccosh}(1+2(n-1)^2)$ (at least) on $\mathbb{H}/\Gamma_2(N)$. 

In other words, for all $0\leq t\leq\textrm{arccosh}(1+2(n-1)^2)$, we have that $U\cap a(t)V=\emptyset$. Here, $a(t)=\textrm{diag}(e^t,e^{-t})$ is the diagonal subgroup of $\textrm{SL}(2,\mathbb{R})$ and we're slightly abusing the notation by identifying $U,V\subset\mathbb{H}/\Gamma_2(N)=SO(2,\mathbb{R})\backslash SL(2,\mathbb{R})/\Gamma_2(N)$ with their lifts to $SL(2,\mathbb{R})/\Gamma_2(N)$. Therefore, by taking $f=\frac{\chi_U}{\sqrt{\textrm{Area}(U)}}$, $g=\frac{\chi_V}{\sqrt{\textrm{Area}(V)}}$ and $u=f-\int f$, $v=g-\int g$, we get
\begin{eqnarray*}
\int u\cdot v\circ a(t)&=&\int\left(f-\int f\right)\cdot\left(g\circ a(t)-\int g\right) \\
&=& \int (f\cdot g\circ a(t))-\int f\int g \\ 
&=& -\int f \int g
\end{eqnarray*}
for every $0\leq t\leq\textrm{arccosh}(1+2(n-1)^2)$. Here $\textrm{Area}$ is the normalised hyperbolic area form of $\mathbb{H}/\Gamma_2(N)$.

Assume that the regular representation of $SL(2,\mathbb{R})$ on $L^2(\mathbb{H}/\Gamma_2(N))$ doesn't exhibit complementary series. By Ratner's work~\cite{Rt}, it follows from the identity above that, for all $0\leq t\leq\textrm{arccosh}(1+2(n-1)^2)$, 
\begin{eqnarray*}
\textrm{Area}(U)&=&\sqrt{\textrm{Area}(U)}\cdot \sqrt{\textrm{Area}(V)}=\int f \int g=\left|\int u\cdot v\circ a(t)\right|\\ 
&\leq& \tilde{K}\|u\|_{L^2}\|v\|_{L^2}te^{-t} \\ 
&=& \tilde{K}\left(\|f\|_{L^2}^2-\|f\|_{L^1}^2\right)^{1/2}\left(\|g\|_{L^2}^2-\|g\|_{L^1}^2\right)^{1/2}te^{-t} \\
&=& \tilde{K}(1-\textrm{Area}(U))te^{-t}\\
&\leq& \tilde{K}te^{-t}
\end{eqnarray*}
where $\tilde{K}>0$ is a universal constant. In our case, since the hyperbolic area of $\mathcal{F}_2-U$ is $2$ and the hyperbolic area of $\mathcal{F}_2$ is $2\pi$ (equal to the area of $\mathbb{H}/\Gamma_2$), we have $\textrm{Area}(U)=(2\pi-2)/2\pi N$.

Thus, in the previous estimate, we see that the absence of complementary series implies 
$$\frac{2\pi-2}{2\pi N}\leq\tilde{K} t_Ne^{-t_N}$$
with $t_N=\textrm{arccosh}(1+2(\lfloor \frac{N}{2}\rfloor-1)^2)$. Since this inequality is false for all sufficiently large $N$, the proof of the proposition is complete.
\end{proof}

\begin{remark}The constant $\tilde{K}$ in Ratner's work~\cite{Rt} can be rendered explicit (by bookkeeping it along Ratner's arguments). By following~\cite{Rt} closely, we found that one can take $\tilde{K}=\frac{(32+\sqrt{2})}{3e^3(1-e^{-4})^2}+(1+2\sqrt{2})e$ in the proof of the previous proposition. In particular, by inserting this into the estimate 
$$\textrm{Area}(U)\leq \tilde{K}(1-\textrm{Area}(U))te^{-t}$$
derived above, one eventually find that complementary series show up as soon as $N\geq 170$.
\end{remark}

\begin{remark}The proof of the previous proposition can be adapted to show that the size of the spectral gap becomes arbitrarily small for a sufficiently large $N$. Indeed, if there was a uniform size $\sigma>0$ for the spectral gap along the family $\mathbb{H}/\Gamma_2(N)$, one could apply Ratner's work~\cite{Rt} to deduce 
$$\textrm{Area}(U)\leq\tilde{K}t_Ne^{-\sigma\cdot t_N}$$
where $U$ and $t_N$ are as in the proof of the previous proposition. Hence, 
$$\frac{2\pi-2}{2\pi N}\leq\tilde{K} t_Ne^{-t_N}$$
with $t_N=\textrm{arccosh}(1+2(\lfloor \frac{N}{2}\rfloor-1)^2)$, a contradiction for all sufficiently large $N$.
\end{remark}

\newpage

\section{Picture of the origami in the case $k=3$}

Figure~\ref{picture} shows the origami $Z$ from 
Corollary~\ref{c.origamiZ3}.\\

\newcommand{\sa}{$157$}
\renewcommand{\sb}{$529$}
\renewcommand{\sc}{$541$}
\newcommand{\sd}{$289$}
\newcommand{\se}{$303$}
\renewcommand{\sf}{$147$}
\newcommand{\sg}{$183$}
\newcommand{\sh}{$315$}
\newcommand{\si}{$325$}
\newcommand{\sj}{$553$}
\newcommand{\sk}{$565$}
\renewcommand{\sl}{$169$}
\newcommand{\ta}{$158$}
\newcommand{\tb}{$530$}
\newcommand{\tc}{$542$}
\newcommand{\td}{$290$}
\newcommand{\te}{$304$}
\newcommand{\tf}{$148$}
\newcommand{\tg}{$184$}
\renewcommand{\th}{$316$}
\newcommand{\ti}{$326$}
\newcommand{\tj}{$554$}
\newcommand{\tk}{$566$}
\newcommand{\tl}{$170$}
\newcommand{\ua}{$61$}
\newcommand{\ub}{$433$}
\newcommand{\uc}{$445$}
\newcommand{\ud}{$337$}
\newcommand{\ue}{$349$}
\newcommand{\uf}{$49$}
\newcommand{\ug}{$85$}
\newcommand{\uh}{$361$}
\newcommand{\ui}{$373$}
\newcommand{\uj}{$457$}
\newcommand{\uk}{$469$}
\newcommand{\ul}{$73$}
\renewcommand{\va}{$62$}
\renewcommand{\vb}{$434$}
\renewcommand{\vc}{$446$}
\newcommand{\vd}{$338$}
\renewcommand{\ve}{$350$}
\renewcommand{\vf}{$50$}
\renewcommand{\vg}{$86$}
\renewcommand{\vh}{$362$}
\renewcommand{\vi}{$374$}
\renewcommand{\vj}{$458$}
\renewcommand{\vk}{$470$}
\renewcommand{\vl}{$74$}

\noindent
\scalebox{.45}{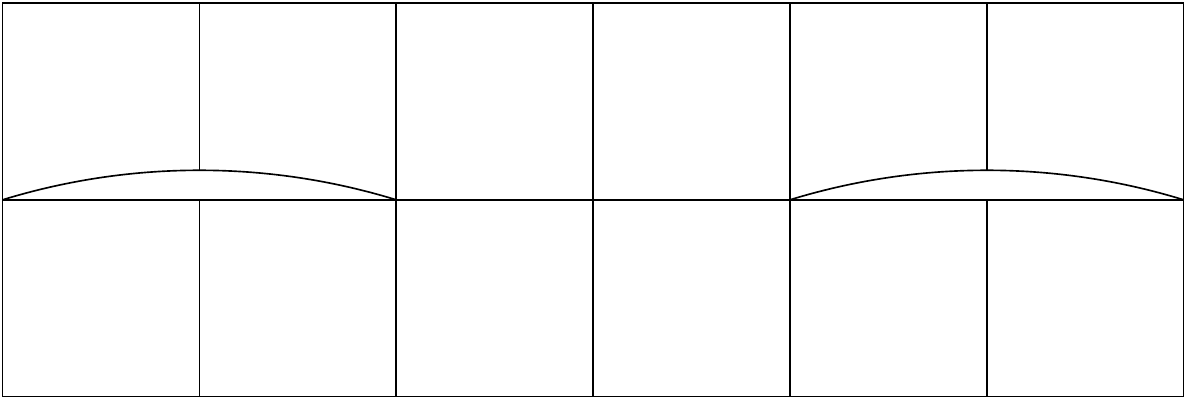}\hspace*{5mm}
\scalebox{.45 }{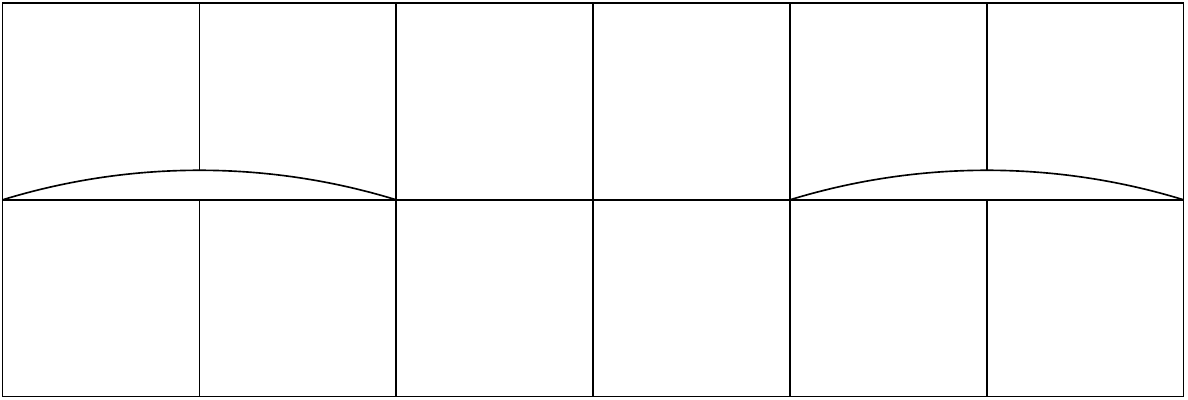}\\[3mm]
\scalebox{.45}{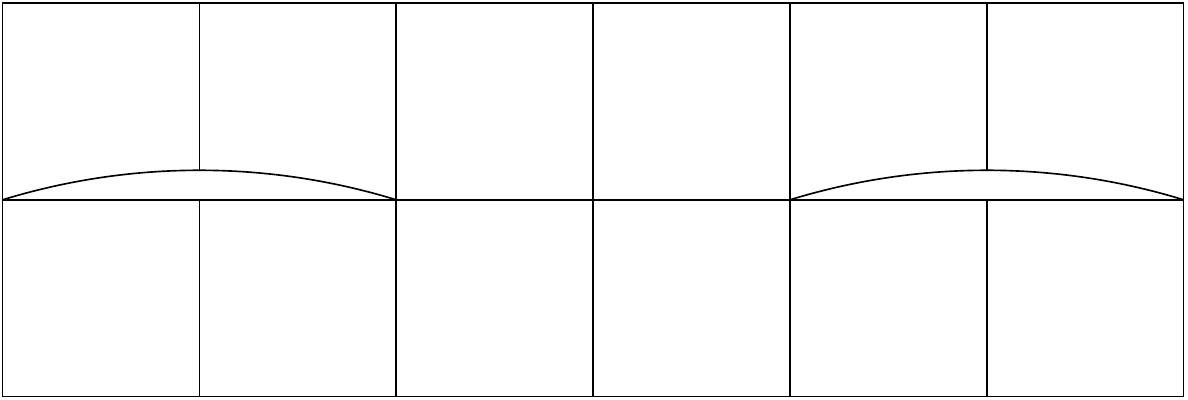}\hspace*{5mm}
\scalebox{.45}{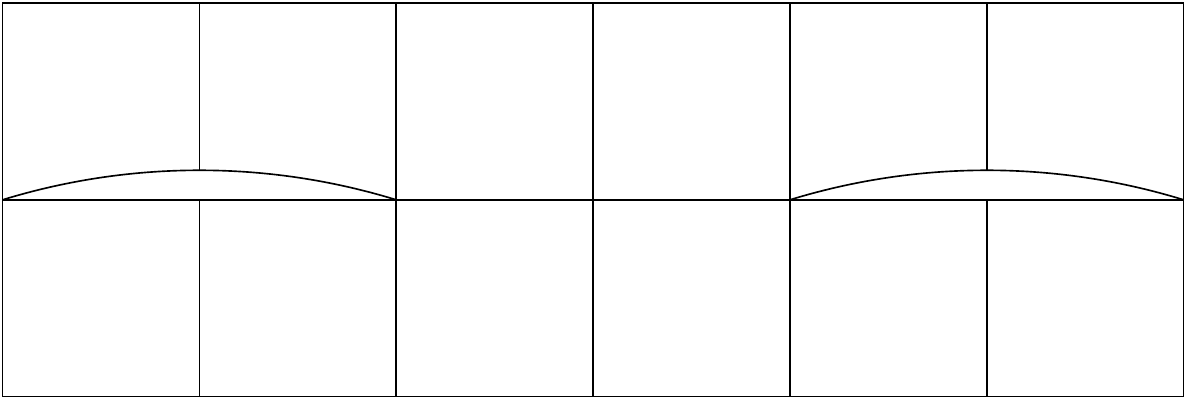}\\[3mm]
\renewcommand{\sa}{$397$}
\renewcommand{\sb}{$481$}
\renewcommand{\sc}{$493$}
\renewcommand{\sd}{$97$}
\renewcommand{\se}{$109$}
\renewcommand{\sf}{$385$}
\renewcommand{\sg}{$421$}
\renewcommand{\sh}{$121$}
\renewcommand{\si}{$133$}
\renewcommand{\sj}{$505$}
\renewcommand{\sk}{$517$}
\renewcommand{\sl}{$409$}
\renewcommand{\ta}{$398$}
\renewcommand{\tb}{$482$}
\renewcommand{\tc}{$494$}
\renewcommand{\td}{$98$}
\renewcommand{\te}{$110$}
\renewcommand{\tf}{$386$}
\renewcommand{\tg}{$422$}
\renewcommand{\th}{$122$}
\renewcommand{\ti}{$134$}
\renewcommand{\tj}{$506$}
\renewcommand{\tk}{$518$}
\renewcommand{\tl}{$410$}
\renewcommand{\ua}{$253$}
\renewcommand{\ub}{$1$}
\renewcommand{\uc}{$13$}
\renewcommand{\ud}{$193$}
\renewcommand{\ue}{$207$}
\renewcommand{\uf}{$243$}
\renewcommand{\ug}{$279$}
\renewcommand{\uh}{$219$}
\renewcommand{\ui}{$229$}
\renewcommand{\uj}{$25$}
\renewcommand{\uk}{$37$}
\renewcommand{\ul}{$265$}
\renewcommand{\va}{$254$}
\renewcommand{\vb}{$2$}
\renewcommand{\vc}{$14$}
\renewcommand{\vd}{$194$}
\renewcommand{\ve}{$208$}
\renewcommand{\vf}{$244$}
\renewcommand{\vg}{$280$}
\renewcommand{\vh}{$220$}
\renewcommand{\vi}{$230$}
\renewcommand{\vj}{$26$}
\renewcommand{\vk}{$38$}
\renewcommand{\vl}{$266$}

\scalebox{.45 }{\input{example4a.pdf_t}}\hspace*{5mm}
\scalebox{.45 }{\input{example4b.pdf_t}}\\[3mm]
\scalebox{.45 }{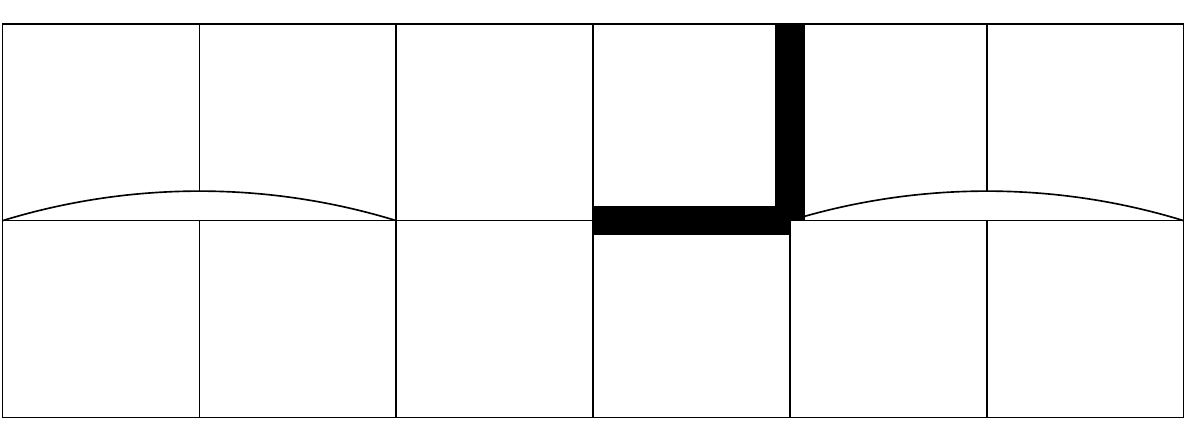}\hspace*{5mm}
\scalebox{.45 }{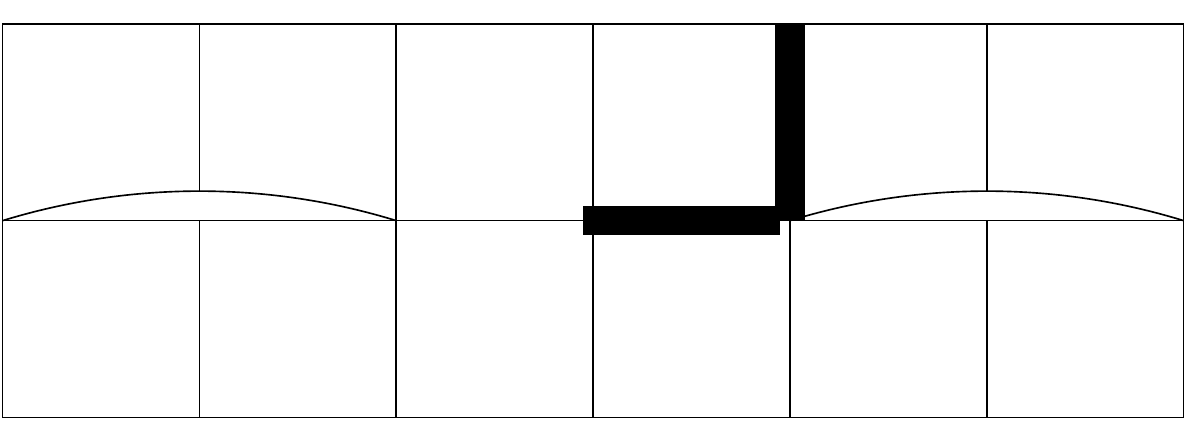}
\hspace* {3mm}{\fs $j = 1$}\\[10mm]
\renewcommand{\sa}{$159$}
\renewcommand{\sb}{$531$}
\renewcommand{\sc}{$543$}
\renewcommand{\sd}{$291$}
\renewcommand{\se}{$305$}
\renewcommand{\sf}{$149$}
\renewcommand{\sg}{$185$}
\renewcommand{\sh}{$317$}
\renewcommand{\si}{$327$}
\renewcommand{\sj}{$555$}
\renewcommand{\sk}{$567$}
\renewcommand{\sl}{$171$}
\renewcommand{\ta}{$160$}
\renewcommand{\tb}{$532$}
\renewcommand{\tc}{$544$}
\renewcommand{\td}{$292$}
\renewcommand{\te}{$306$}
\renewcommand{\tf}{$150$}
\renewcommand{\tg}{$186$}
\renewcommand{\th}{$318$}
\renewcommand{\ti}{$328$}
\renewcommand{\tj}{$556$}
\renewcommand{\tk}{$568$}
\renewcommand{\tl}{$172$}
\renewcommand{\ua}{$63$}
\renewcommand{\ub}{$435$}
\renewcommand{\uc}{$447$}
\renewcommand{\ud}{$339$}
\renewcommand{\ue}{$351$}
\renewcommand{\uf}{$51$}
\renewcommand{\ug}{$87$}
\renewcommand{\uh}{$363$}
\renewcommand{\ui}{$375$}
\renewcommand{\uj}{$459$}
\renewcommand{\uk}{$471$}
\renewcommand{\ul}{$75$}
\renewcommand{\va}{$64$}
\renewcommand{\vb}{$436$}
\renewcommand{\vc}{$448$}
\renewcommand{\vd}{$340$}
\renewcommand{\ve}{$352$}
\renewcommand{\vf}{$52$}
\renewcommand{\vg}{$88$}
\renewcommand{\vh}{$364$}
\renewcommand{\vi}{$376$}
\renewcommand{\vj}{$460$}
\renewcommand{\vk}{$472$}
\renewcommand{\vl}{$76$}

\scalebox{.45 }{\input{example4a.pdf_t}}\hspace*{5mm}
\scalebox{.45 }{\input{example4b.pdf_t}}\\[3mm]
\scalebox{.45 }{\input{example4c.pdf_t}}\hspace*{5mm}
\scalebox{.45 }{\input{example4d.pdf_t}}\\[3mm]
\renewcommand{\ua}{$255$}
\renewcommand{\ub}{$3$}
\renewcommand{\uc}{$15$}
\renewcommand{\ud}{$195$}
\renewcommand{\ue}{$209$}
\renewcommand{\uf}{$245$}
\renewcommand{\ug}{$281$}
\renewcommand{\uh}{$221$}
\renewcommand{\ui}{$231$}
\renewcommand{\uj}{$27$}
\renewcommand{\uk}{$39$}
\renewcommand{\ul}{$267$}
\renewcommand{\ta}{$400$}
\renewcommand{\tb}{$484$}
\renewcommand{\tc}{$496$}
\renewcommand{\td}{$100$}
\renewcommand{\te}{$112$}
\renewcommand{\tf}{$388$}
\renewcommand{\tg}{$424$}
\renewcommand{\th}{$124$}
\renewcommand{\ti}{$136$}
\renewcommand{\tj}{$508$}
\renewcommand{\tk}{$520$}
\renewcommand{\tl}{$412$}
\renewcommand{\sa}{$399$}
\renewcommand{\sb}{$483$}
\renewcommand{\sc}{$495$}
\renewcommand{\sd}{$99$}
\renewcommand{\se}{$111$}
\renewcommand{\sf}{$387$}
\renewcommand{\sg}{$423$}
\renewcommand{\sh}{$123$}
\renewcommand{\si}{$135$}
\renewcommand{\sj}{$507$}
\renewcommand{\sk}{$519$}
\renewcommand{\sl}{$411$}
\renewcommand{\va}{$256$}
\renewcommand{\vb}{$4$}
\renewcommand{\vc}{$16$}
\renewcommand{\vd}{$196$}
\renewcommand{\ve}{$210$}
\renewcommand{\vf}{$246$}
\renewcommand{\vg}{$282$}
\renewcommand{\vh}{$222$}
\renewcommand{\vi}{$232$}
\renewcommand{\vj}{$28$}
\renewcommand{\vk}{$40$}
\renewcommand{\vl}{$268$}

\scalebox{.45 }{\input{example4a.pdf_t}}\hspace*{5mm}
\scalebox{.45 }{\input{example4b.pdf_t}}\\[3mm]
\scalebox{.45 }{\input{example4c.pdf_t}}\hspace*{5mm}
\scalebox{.45 }{\input{example4d.pdf_t}}
\hspace* {3mm}{\fs $j = 2$}
\newpage
\renewcommand{\sa}{$161$}
\renewcommand{\sb}{$533$}
\renewcommand{\sc}{$545$}
\renewcommand{\sd}{$293$}
\renewcommand{\se}{$307$}
\renewcommand{\sf}{$151$}
\renewcommand{\sg}{$187$}
\renewcommand{\sh}{$319$}
\renewcommand{\si}{$329$}
\renewcommand{\sj}{$557$}
\renewcommand{\sk}{$569$}
\renewcommand{\sl}{$173$}
\renewcommand{\ta}{$162$}
\renewcommand{\tb}{$534$}
\renewcommand{\tc}{$546$}
\renewcommand{\td}{$294$}
\renewcommand{\te}{$308$}
\renewcommand{\tf}{$152$}
\renewcommand{\tg}{$188$}
\renewcommand{\th}{$320$}
\renewcommand{\ti}{$330$}
\renewcommand{\tj}{$558$}
\renewcommand{\tk}{$570$}
\renewcommand{\tl}{$174$}
\renewcommand{\va}{$66$}
\renewcommand{\vb}{$438$}
\renewcommand{\vc}{$450$}
\renewcommand{\vd}{$342$}
\renewcommand{\ve}{$354$}
\renewcommand{\vf}{$54$}
\renewcommand{\ug}{$89$}
\renewcommand{\uh}{$365$}
\renewcommand{\ui}{$377$}
\renewcommand{\uj}{$461$}
\renewcommand{\uk}{$473$}
\renewcommand{\ul}{$77$}
\renewcommand{\ua}{$65$}
\renewcommand{\ub}{$437$}
\renewcommand{\uc}{$449$}
\renewcommand{\ud}{$341$}
\renewcommand{\ue}{$353$}
\renewcommand{\uf}{$53$}
\renewcommand{\vg}{$90$}
\renewcommand{\vh}{$366$}
\renewcommand{\vi}{$378$}
\renewcommand{\vj}{$462$}
\renewcommand{\vk}{$474$}
\renewcommand{\vl}{$78$}

\noindent
\scalebox{.45 }{\input{example4a.pdf_t}}\hspace*{5mm}
\scalebox{.45 }{\input{example4b.pdf_t}}\\[3mm]
\scalebox{.45 }{\input{example4c.pdf_t}}\hspace*{5mm}
\scalebox{.45 }{\input{example4d.pdf_t}}\\[3mm]
\renewcommand{\sa}{$401$}
\renewcommand{\sb}{$485$}
\renewcommand{\sc}{$497$}
\renewcommand{\sd}{$101$}
\renewcommand{\se}{$113$}
\renewcommand{\sf}{$389$}
\renewcommand{\sg}{$425$}
\renewcommand{\sh}{$125$}
\renewcommand{\si}{$137$}
\renewcommand{\sj}{$509$}
\renewcommand{\sk}{$521$}
\renewcommand{\sl}{$413$}
\renewcommand{\ta}{$402$}
\renewcommand{\tb}{$486$}
\renewcommand{\tc}{$498$}
\renewcommand{\td}{$102$}
\renewcommand{\te}{$114$}
\renewcommand{\tf}{$390$}
\renewcommand{\tg}{$426$}
\renewcommand{\th}{$126$}
\renewcommand{\ti}{$138$}
\renewcommand{\tj}{$510$}
\renewcommand{\tk}{$522$}
\renewcommand{\tl}{$414$}
\renewcommand{\ua}{$257$}
\renewcommand{\ub}{$5$}
\renewcommand{\uc}{$17$}
\renewcommand{\ud}{$197$}
\renewcommand{\ue}{$211$}
\renewcommand{\uf}{$247$}
\renewcommand{\ug}{$283$}
\renewcommand{\uh}{$223$}
\renewcommand{\ui}{$233$}
\renewcommand{\uj}{$29$}
\renewcommand{\uk}{$41$}
\renewcommand{\ul}{$269$}
\renewcommand{\va}{$258$}
\renewcommand{\vb}{$6$}
\renewcommand{\vc}{$18$}
\renewcommand{\vd}{$198$}
\renewcommand{\ve}{$212$}
\renewcommand{\vf}{$248$}
\renewcommand{\vg}{$284$}
\renewcommand{\vh}{$224$}
\renewcommand{\vi}{$234$}
\renewcommand{\vj}{$30$}
\renewcommand{\vk}{$42$}
\renewcommand{\vl}{$270$}

\scalebox{.45 }{\input{example4a.pdf_t}}\hspace*{5mm}
\scalebox{.45 }{\input{example4b.pdf_t}}\\[3mm]
\scalebox{.45 }{\input{example4c.pdf_t}}\hspace*{5mm}
\scalebox{.45 }{\input{example4d.pdf_t}}
\hspace* {3mm}{\fs $j = 3$}\\[10mm]
\renewcommand{\sa}{$163$}
\renewcommand{\sb}{$535$}
\renewcommand{\sc}{$547$}
\renewcommand{\sd}{$295$}
\renewcommand{\se}{$309$}
\renewcommand{\sf}{$153$}
\renewcommand{\sg}{$189$}
\renewcommand{\sh}{$321$}
\renewcommand{\si}{$331$}
\renewcommand{\sj}{$559$}
\renewcommand{\sk}{$571$}
\renewcommand{\sl}{$175$}
\renewcommand{\ta}{$164$}
\renewcommand{\tb}{$536$}
\renewcommand{\tc}{$548$}
\renewcommand{\td}{$296$}
\renewcommand{\te}{$310$}
\renewcommand{\tf}{$154$}
\renewcommand{\tg}{$190$}
\renewcommand{\th}{$322$}
\renewcommand{\ti}{$332$}
\renewcommand{\tj}{$560$}
\renewcommand{\tk}{$572$}
\renewcommand{\tl}{$176$}
\renewcommand{\ua}{$67$}
\renewcommand{\ub}{$439$}
\renewcommand{\uc}{$451$}
\renewcommand{\ud}{$343$}
\renewcommand{\ue}{$355$}
\renewcommand{\uf}{$55$}
\renewcommand{\ug}{$91$}
\renewcommand{\uh}{$367$}
\renewcommand{\ui}{$379$}
\renewcommand{\uj}{$463$}
\renewcommand{\uk}{$475$}
\renewcommand{\ul}{$79$}
\renewcommand{\va}{$68$}
\renewcommand{\vb}{$440$}
\renewcommand{\vc}{$452$}
\renewcommand{\vd}{$344$}
\renewcommand{\ve}{$356$}
\renewcommand{\vf}{$56$}
\renewcommand{\vg}{$92$}
\renewcommand{\vh}{$368$}
\renewcommand{\vi}{$380$}
\renewcommand{\vj}{$464$}
\renewcommand{\vk}{$476$}
\renewcommand{\vl}{$80$}

\scalebox{.45 }{\input{example4a.pdf_t}}\hspace*{5mm}
\scalebox{.45 }{\input{example4b.pdf_t}}\\[3mm]
\scalebox{.45 }{\input{example4c.pdf_t}}\hspace*{5mm}
\scalebox{.45 }{\input{example4d.pdf_t}}\\[3mm]
\renewcommand{\sa}{$403$}
\renewcommand{\sb}{$487$}
\renewcommand{\sc}{$499$}
\renewcommand{\sd}{$103$}
\renewcommand{\se}{$115$}
\renewcommand{\sf}{$391$}
\renewcommand{\sg}{$427$}
\renewcommand{\sh}{$127$}
\renewcommand{\si}{$139$}
\renewcommand{\sj}{$511$}
\renewcommand{\sk}{$523$}
\renewcommand{\sl}{$415$}
\renewcommand{\ta}{$404$}
\renewcommand{\tb}{$488$}
\renewcommand{\tc}{$500$}
\renewcommand{\td}{$104$}
\renewcommand{\te}{$116$}
\renewcommand{\tf}{$392$}
\renewcommand{\tg}{$428$}
\renewcommand{\th}{$128$}
\renewcommand{\ti}{$140$}
\renewcommand{\tj}{$512$}
\renewcommand{\tk}{$524$}
\renewcommand{\tl}{$416$}
\renewcommand{\ua}{$259$}
\renewcommand{\ub}{$7$}
\renewcommand{\uc}{$19$}
\renewcommand{\ud}{$199$}
\renewcommand{\ue}{$213$}
\renewcommand{\uf}{$249$}
\renewcommand{\ug}{$285$}
\renewcommand{\uh}{$225$}
\renewcommand{\ui}{$235$}
\renewcommand{\uj}{$31$}
\renewcommand{\uk}{$43$}
\renewcommand{\ul}{$271$}
\renewcommand{\va}{$260$}
\renewcommand{\vb}{$8$}
\renewcommand{\vc}{$20$}
\renewcommand{\vd}{$200$}
\renewcommand{\ve}{$214$}
\renewcommand{\vf}{$250$}
\renewcommand{\vg}{$286$}
\renewcommand{\vh}{$226$}
\renewcommand{\vi}{$236$}
\renewcommand{\vj}{$32$}
\renewcommand{\vk}{$44$}
\renewcommand{\vl}{$272$}

\scalebox{.45 }{\input{example4a.pdf_t}}\hspace*{5mm}
\scalebox{.45 }{\input{example4b.pdf_t}}\\[3mm]
\scalebox{.45 }{\input{example4c.pdf_t}}\hspace*{5mm}
\scalebox{.45 }{\input{example4d.pdf_t}}
\hspace* {3mm}{\fs $j = 4$}
\newpage
\enlargethispage{\baselineskip}
\renewcommand{\sa}{$165$}
\renewcommand{\sb}{$537$}
\renewcommand{\sc}{$549$}
\renewcommand{\sd}{$297$}
\renewcommand{\se}{$311$}
\renewcommand{\sf}{$155$}
\renewcommand{\sg}{$191$}
\renewcommand{\sh}{$323$}
\renewcommand{\si}{$333$}
\renewcommand{\sj}{$561$}
\renewcommand{\sk}{$573$}
\renewcommand{\sl}{$177$}
\renewcommand{\ta}{$166$}
\renewcommand{\tb}{$538$}
\renewcommand{\tc}{$550$}
\renewcommand{\td}{$298$}
\renewcommand{\te}{$312$}
\renewcommand{\tf}{$156$}
\renewcommand{\tg}{$192$}
\renewcommand{\th}{$324$}
\renewcommand{\ti}{$334$}
\renewcommand{\tj}{$562$}
\renewcommand{\tk}{$574$}
\renewcommand{\tl}{$178$}
\renewcommand{\ua}{$69$}
\renewcommand{\ub}{$441$}
\renewcommand{\uc}{$453$}
\renewcommand{\ud}{$345$}
\renewcommand{\ue}{$357$}
\renewcommand{\uf}{$57$}
\renewcommand{\ug}{$93$}
\renewcommand{\uh}{$369$}
\renewcommand{\ui}{$381$}
\renewcommand{\uj}{$465$}
\renewcommand{\uk}{$477$}
\renewcommand{\ul}{$81$}
\renewcommand{\va}{$70$}
\renewcommand{\vb}{$442$}
\renewcommand{\vc}{$454$}
\renewcommand{\vd}{$346$}
\renewcommand{\ve}{$358$}
\renewcommand{\vf}{$58$}
\renewcommand{\vg}{$94$}
\renewcommand{\vh}{$370$}
\renewcommand{\vi}{$382$}
\renewcommand{\vj}{$466$}
\renewcommand{\vk}{$478$}
\renewcommand{\vl}{$82$}

\noindent
\scalebox{.45 }{\input{example4a.pdf_t}}\hspace*{5mm}
\scalebox{.45 }{\input{example4b.pdf_t}}\\[3mm]
\scalebox{.45 }{\input{example4c.pdf_t}}\hspace*{5mm}
\scalebox{.45 }{\input{example4d.pdf_t}}\\[3mm]
\renewcommand{\sa}{$405$}
\renewcommand{\sb}{$489$}
\renewcommand{\sc}{$501$}
\renewcommand{\sd}{$105$}
\renewcommand{\se}{$117$}
\renewcommand{\sf}{$393$}
\renewcommand{\sg}{$429$}
\renewcommand{\sh}{$129$}
\renewcommand{\si}{$141$}
\renewcommand{\sj}{$513$}
\renewcommand{\sk}{$525$}
\renewcommand{\sl}{$417$}
\renewcommand{\ta}{$406$}
\renewcommand{\tb}{$490$}
\renewcommand{\tc}{$502$}
\renewcommand{\td}{$106$}
\renewcommand{\te}{$118$}
\renewcommand{\tf}{$394$}
\renewcommand{\tg}{$430$}
\renewcommand{\th}{$130$}
\renewcommand{\ti}{$142$}
\renewcommand{\tj}{$514$}
\renewcommand{\tk}{$526$}
\renewcommand{\tl}{$418$}
\renewcommand{\ua}{$261$}
\renewcommand{\ub}{$9$}
\renewcommand{\uc}{$21$}
\renewcommand{\ud}{$201$}
\renewcommand{\ue}{$215$}
\renewcommand{\uf}{$251$}
\renewcommand{\ug}{$287$}
\renewcommand{\uh}{$227$}
\renewcommand{\ui}{$237$}
\renewcommand{\uj}{$33$}
\renewcommand{\uk}{$45$}
\renewcommand{\ul}{$273$}
\renewcommand{\va}{$262$}
\renewcommand{\vb}{$10$}
\renewcommand{\vc}{$22$}
\renewcommand{\vd}{$202$}
\renewcommand{\ve}{$216$}
\renewcommand{\vf}{$252$}
\renewcommand{\vg}{$288$}
\renewcommand{\vh}{$228$}
\renewcommand{\vi}{$238$}
\renewcommand{\vj}{$34$}
\renewcommand{\vk}{$46$}
\renewcommand{\vl}{$274$}

\scalebox{.45 }{\input{example4a.pdf_t}}\hspace*{5mm}
\scalebox{.45 }{\input{example4b.pdf_t}}\\[3mm]
\scalebox{.45 }{\input{example4c.pdf_t}}\hspace*{5mm}
\scalebox{.45 }{\input{example4d.pdf_t}}
\hspace* {3mm}{\fs $j = 5$}\\[10mm]
\renewcommand{\sa}{$167$}
\renewcommand{\sb}{$539$}
\renewcommand{\sc}{$551$}
\renewcommand{\sd}{$299$}
\renewcommand{\se}{$301$}
\renewcommand{\sf}{$145$}
\renewcommand{\sg}{$181$}
\renewcommand{\sh}{$313$}
\renewcommand{\si}{$335$}
\renewcommand{\sj}{$563$}
\renewcommand{\sk}{$575$}
\renewcommand{\sl}{$179$}
\renewcommand{\ta}{$168$}
\renewcommand{\tb}{$540$}
\renewcommand{\tc}{$552$}
\renewcommand{\td}{$300$}
\renewcommand{\te}{$302$}
\renewcommand{\tf}{$146$}
\renewcommand{\tg}{$182$}
\renewcommand{\th}{$314$}
\renewcommand{\ti}{$336$}
\renewcommand{\tj}{$564$}
\renewcommand{\tk}{$576$}
\renewcommand{\tl}{$180$}
\renewcommand{\ua}{$71$}
\renewcommand{\ub}{$443$}
\renewcommand{\uc}{$455$}
\renewcommand{\ud}{$347$}
\renewcommand{\ue}{$359$}
\renewcommand{\uf}{$59$}
\renewcommand{\ug}{$95$}
\renewcommand{\uh}{$371$}
\renewcommand{\ui}{$383$}
\renewcommand{\uj}{$467$}
\renewcommand{\uk}{$479$}
\renewcommand{\ul}{$83$}
\renewcommand{\va}{$72$}
\renewcommand{\vb}{$444$}
\renewcommand{\vc}{$456$}
\renewcommand{\vd}{$348$}
\renewcommand{\ve}{$360$}
\renewcommand{\vf}{$60$}
\renewcommand{\vg}{$96$}
\renewcommand{\vh}{$372$}
\renewcommand{\vi}{$384$}
\renewcommand{\vj}{$468$}
\renewcommand{\vk}{$480$}
\renewcommand{\vl}{$84$}

\scalebox{.45 }{\input{example4a.pdf_t}}\hspace*{5mm}
\scalebox{.45 }{\input{example4b.pdf_t}}\\[3mm]
\scalebox{.45 }{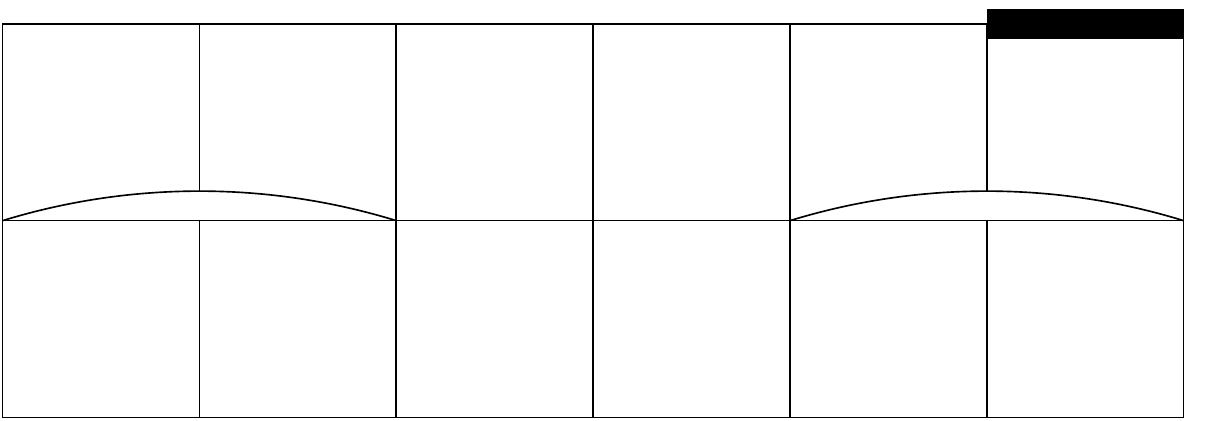}\hspace*{5mm}
\scalebox{.45 }{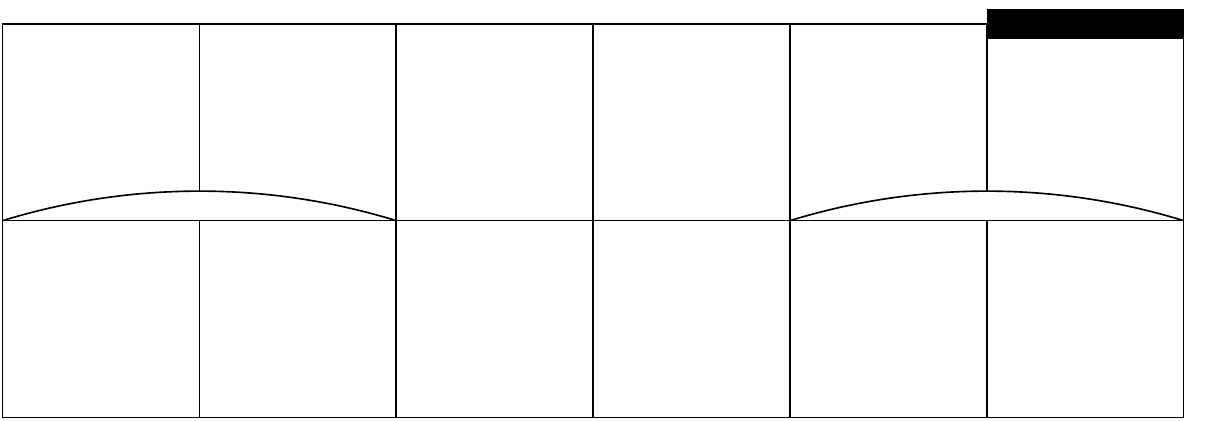}\\[3mm]
\renewcommand{\sa}{$407$}
\renewcommand{\sb}{$491$}
\renewcommand{\sc}{$503$}
\renewcommand{\sd}{$107$}
\renewcommand{\se}{$119$}
\renewcommand{\sf}{$395$}
\renewcommand{\sg}{$431$}
\renewcommand{\sh}{$131$}
\renewcommand{\si}{$143$}
\renewcommand{\sj}{$515$}
\renewcommand{\sk}{$527$}
\renewcommand{\sl}{$419$}
\renewcommand{\ta}{$408$}
\renewcommand{\tb}{$492$}
\renewcommand{\tc}{$504$}
\renewcommand{\td}{$108$}
\renewcommand{\te}{$120$}
\renewcommand{\tf}{$396$}
\renewcommand{\tg}{$432$}
\renewcommand{\th}{$132$}
\renewcommand{\ti}{$144$}
\renewcommand{\tj}{$516$}
\renewcommand{\tk}{$528$}
\renewcommand{\tl}{$420$}
\renewcommand{\ua}{$263$}
\renewcommand{\ub}{$11$}
\renewcommand{\uc}{$23$}
\renewcommand{\ud}{$203$}
\renewcommand{\ue}{$205$}
\renewcommand{\uf}{$241$}
\renewcommand{\ug}{$277$}
\renewcommand{\uh}{$217$}
\renewcommand{\ui}{$239$}
\renewcommand{\uj}{$35$}
\renewcommand{\uk}{$47$}
\renewcommand{\ul}{$275$}
\renewcommand{\va}{$264$}
\renewcommand{\vb}{$12$}
\renewcommand{\vc}{$24$}
\renewcommand{\vd}{$204$}
\renewcommand{\ve}{$206$}
\renewcommand{\vf}{$242$}
\renewcommand{\vg}{$278$}
\renewcommand{\vh}{$218$}
\renewcommand{\vi}{$240$}
\renewcommand{\vj}{$36$}
\renewcommand{\vk}{$48$}
\renewcommand{\vl}{$276$}

\scalebox{.45 }{\input{example4a.pdf_t}}\hspace*{5mm}
\scalebox{.45 }{\input{example4b.pdf_t}}\\[3mm]
\scalebox{.45 }{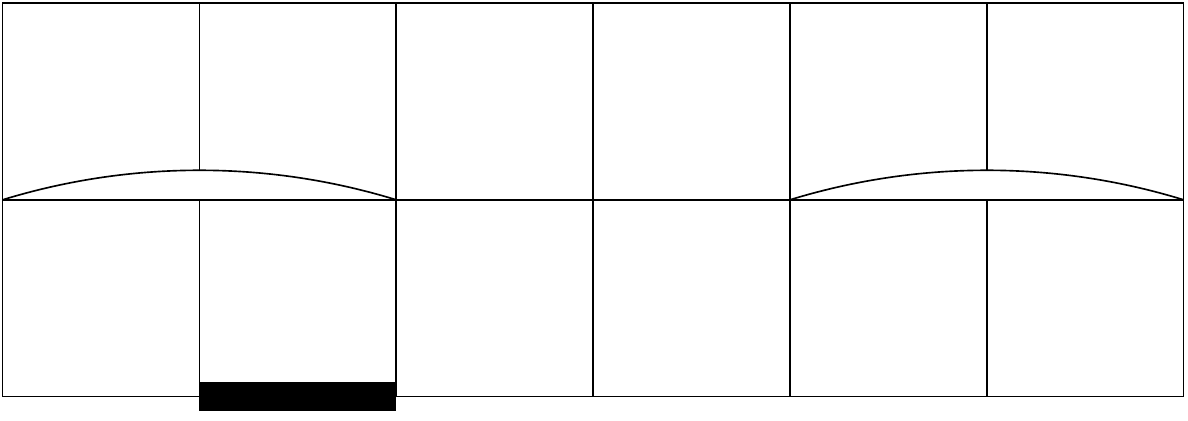}\hspace*{5mm}
\scalebox{.45 }{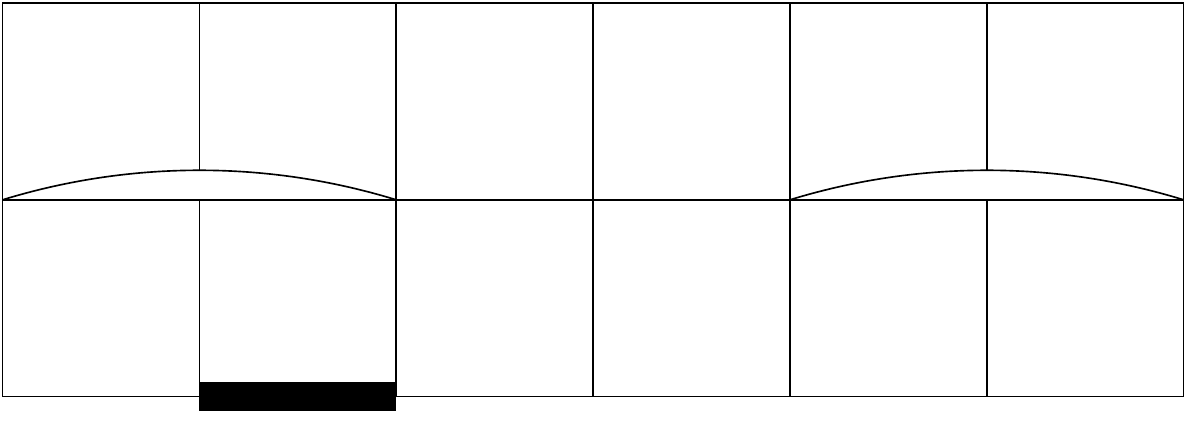}
\hspace* {3mm}{\fs $j = 6$}
\captionof{figure}{The origami $Z$. Change at the black bars
from the left to the right leaf and vice versa. See Corollary~\ref{c.origamiZ3}
and Figure~\ref{yhorizontal} for the gluing rules.
\label{picture}}

\end{document}